\documentclass[12pt]{article}

\makeatletter
\def\input@path{{styles/}}
\makeatother

\newcommand{\UsePackage}[1]{%
  \IfFileExists{styles/#1.sty}{%
      \usepackage{styles/#1}%
   }{%
      \IfFileExists{../styles/#1.sty}{%
         \usepackage{../styles/#1}%
      }{%
         \usepackage{#1}%
      }%
   }%
}

\usepackage[T1]{fontenc}
\usepackage{lmodern}
\usepackage{textcomp}

\usepackage{amsmath}%
\usepackage{amssymb}%
\usepackage[table]{xcolor}%

\usepackage{todonotes}
\usepackage[in]{fullpage}%

\usepackage[amsmath,thmmarks]{ntheorem}%
\theoremseparator{.}%

\usepackage{titlesec}%
\titlelabel{\thetitle. }%
\usepackage{xcolor}%
\usepackage{mleftright}%
\usepackage{xspace}%
\usepackage{graphicx}
\usepackage{hyperref}%
\usepackage{cleveref}

\hypersetup{%
      unicode,
      breaklinks,%
      colorlinks=true,%
      urlcolor=[rgb]{0.25,0.0,0.0},%
      linkcolor=[rgb]{0.5,0.0,0.0},%
      citecolor=[rgb]{0,0.2,0.445},%
      filecolor=[rgb]{0,0,0.4},
      anchorcolor=[rgb]={0.0,0.1,0.2}%
}
\usepackage[ocgcolorlinks]{ocgx2}

\theoremseparator{.}%

\theoremstyle{plain}%
\newtheorem{theorem}{Theorem}[section]

\newtheorem{lemma}[theorem]{Lemma}

\newtheorem{proposition}[theorem]{Proposition}

\theoremstyle{plain}%
\theoremheaderfont{\sf} \theorembodyfont{\upshape}%
\newtheorem*{remark:unnumbered}[theorem]{Remark}%
\newtheorem{remark}[theorem]{Remark}%
\newtheorem{definition}[theorem]{Definition}

\newcommand{\myqedsymbol}{\rule{2mm}{2mm}}

\theoremheaderfont{\em}%
\theorembodyfont{\upshape}%
\theoremstyle{nonumberplain}%
\theoremseparator{}%
\theoremsymbol{\myqedsymbol}%
\newtheorem{proof}{Proof:}%

\providecommand{\emphind}[1]{}%
\renewcommand{\emphind}[1]{\emph{#1}\index{#1}}

\definecolor{blue25emph}{rgb}{0, 0, 11}

\providecommand{\emphic}[2]{}
\renewcommand{\emphic}[2]{\textcolor{blue25emph}{%
      \textbf{\emph{#1}}}\index{#2}}

\providecommand{\emphi}[1]{}%
\renewcommand{\emphi}[1]{\emphic{#1}{#1}}

\definecolor{almostblack}{rgb}{0, 0, 0.3}

\providecommand{\emphw}[1]{}%
\renewcommand{\emphw}[1]{{\textcolor{almostblack}{\emph{#1}}}}%

\providecommand{\emphOnly}[1]{}%
\renewcommand{\emphOnly}[1]{\emph{\textcolor{blue25}{\textbf{#1}}}}

\newcommand{\HLink}[2]{\hyperref[#2]{#1~\ref*{#2}}}
\newcommand{\HLinkSuffix}[3]{\hyperref[#2]{#1\ref*{#2}{#3}}}

\providecommand{\deflab}[1]{}
\renewcommand{\deflab}[1]{\label{def:#1}}

\providecommand{\eqlab}[1]{}%
\renewcommand{\eqlab}[1]{\label{equation:#1}}

\newcommand{\remove}[1]{}%

\usepackage[inline]{enumitem}

\newlist{compactenumA}{enumerate}{5}%
\setlist[compactenumA]{topsep=0pt,itemsep=-1ex,partopsep=1ex,parsep=1ex,%
   label=(\Alph*)}%

\newlist{compactenuma}{enumerate}{5}%
\setlist[compactenuma]{topsep=0pt,itemsep=-1ex,partopsep=1ex,parsep=1ex,%
   label=(\alph*)}%

\newlist{compactenumI}{enumerate}{5}%
\setlist[compactenumI]{topsep=0pt,itemsep=-1ex,partopsep=1ex,parsep=1ex,%
   label=(\Roman*)}%

\newlist{compactenumi}{enumerate}{5}%
\setlist[compactenumi]{topsep=0pt,itemsep=-1ex,partopsep=1ex,parsep=1ex,%
   label=(\roman*)}%

\newlist{compactitem}{itemize}{5}%
\setlist[compactitem]{topsep=0pt,itemsep=-1ex,partopsep=1ex,parsep=1ex,%
   label=\ensuremath{\bullet}}%

\numberwithin{figure}{section}%
\numberwithin{table}{section}%
\numberwithin{equation}{section}%

\newcommand{\norm}[1]{\| #1 \|}

\newcommand{\nbH}{\nabla_H}
\newcommand{\dt}{\partial_t}
\newcommand{\dz}{\partial_z}
\newcommand{\avg}[1]{\overline{#1}}

\begin{document}

\title{Global Well-Posedness and Persistence of $H^2$ Regularity for the Three-Dimensional Primitive Equations with Horizontal and Topography-Adapted Thermal Diffusion}

\author{
Qimin Wang\thanks{
Department of Mathematics,
University of Southern California,
Los Angeles, CA 90089, USA.
Email: \href{mailto:qiminw@usc.edu}{qiminw@usc.edu}.
}
\and
Mohammed Ziane\thanks{
Department of Mathematics,
University of Southern California,
Los Angeles, CA 90089, USA.
Email: \href{mailto:ziane@usc.edu}{ziane@usc.edu}.
}
}

\date{\today}

\maketitle

\begin{abstract}
We study the three-dimensional primitive equations with full viscosity and horizontal thermal diffusion in bounded domains. We establish global well-posedness for admissible $H^1$ initial data and persistence of $H^2$ regularity on the flat-bottom cylinder. The higher regularity analysis combines fractional hydrostatic Stokes regularity with anisotropic estimates of the vertical transport. We also consider variable-depth domains with positive $h\in C^3(\overline M)$ and a diffusion operator adapted to the topography compatible with layers $z/h(x,y)=\mathrm{constant}$. This operator reduces to $\Delta_H$ at constant depth and generally differs from the Cartesian horizontal Laplacian over a varying bottom. Its weighted formulation yields global strong solutions for admissible $H^1$ data. Persistence of $H^2$ regularity is proved under the additional assumption $\partial_\nu h=0$ on $\partial M$, where $\nu$ is the outward horizontal unit normal. 
\end{abstract}

\section{Introduction}
\label{sec:introduction}
The primitive equations are fundamental models in geophysical fluid dynamics for describing the large-scale motion of the oceans and atmosphere. They are derived from the three-dimensional incompressible Navier–Stokes equations under the Boussinesq and hydrostatic approximations, and play a central role in weather prediction, climate dynamics, and ocean circulation theory. We refer to \cite{Lewandowski,Majda,Pedlosky,Vallis,Washington} and the references therein for the physical background and mathematical derivation of the primitive equations.

We first consider the viscous primitive equations with only horizontal diffusion in the temperature equation on a flat-bottom cylindrical domain:

\begin{align}
    \dt v + (v \cdot \nbH) v + w \dz v + \nbH p - \Delta v + f_0 k \times v &= 0, \quad\label{eq:primitive-momentum}\\
    \dz p + T &= 0, \\
    \nbH \cdot v + \dz w &= 0, \\
    \dt T + (v \cdot \nbH) T + w \dz T - \Delta_H T &= 0. \quad 
\end{align}
in $\Omega:=M\times(-h,0)$, where $h>0$ is constant and $M\subset\mathbb R^2$ is a bounded domain with smooth boundary. We set $\Gamma_t=M\times\{0\}$, $\Gamma_b=M\times\{-h\}$, and $\Gamma_s=\partial M\times(-h,0)$. The unknowns are the horizontal velocity $v=(v_1,v_2)$, the vertical velocity $w$, the temperature $T$, and the pressure $p$; the constant $f_0$ denotes the Coriolis parameter. We write $\nabla_H=(\partial_x,\partial_y)$ and $\Delta_H=\partial_x^2+\partial_y^2$ for the horizontal gradient and horizontal Laplacian, respectively.  The system is supplemented with the following boundary and initial conditions:
\begin{gather}
  \dz v=0, \quad\text{on }\Gamma_t, \\
  v=0, \quad\text{on }\Gamma_b\cup\Gamma_s, \\
    w=0, \quad\text{on }\Gamma_t\cup\Gamma_b, \\
    T=0, \quad\text{on }\Gamma_s,\\
(v, T)|_{t=0} = (v_0, T_0).\label{eq:primitive-initial-data}
\end{gather}
Throughout the paper, $L^p(D)$ and $W^{m,q}(D)$ denote the standard Lebesgue and Sobolev spaces on a domain $D$, respectively. We write $H^m(D):=W^{m,2}(D)$ and omit the domain when no confusion can arise. The same notation is used for scalar-valued and vector-valued spaces when the number of components is clear from the context.

The flat-bottom problem is the basic setting in which we present the construction and the global a priori estimates. For the variable-depth case, thermal diffusion is defined along the terrain-following layers, which is related to the iso-sigma formulations used in ocean modeling;  see \cite[Section~5.1]{BarnierEtAl1998Sigma}. More precisely, for
\[
  \Omega_h:=\{(x,y,z):(x,y)\in M,\ -h(x,y)<z<0\},
  \qquad 0<h_*\leq h(x,y)\leq h^*,
\]
the thermal diffusion is a topography-adapted operator. Setting
\begin{align}
 a_i(x,y,z):=\frac{z}{h(x,y)}\partial_{x_i}h(x,y), \text{ for $i=1,2$}
  \qquad
  K_h:=
  \begin{pmatrix}
    1&0&a_1\\
    0&1&a_2\\
    a_1&a_2&|a|^2
  \end{pmatrix}, 
\label{a_i definition}
\end{align}

we define
\begin{align}
\Delta_{H,h}^{\mathrm{top}}T
:=\nabla_{x,y,z}\cdot\bigl(K_h\nabla_{x,y,z}T\bigr),
  \qquad A_{T,h}:=-\Delta_{H,h}^{\mathrm{top}}.
\label{eq:topographic-laplacian-intro}
\end{align}
The columns generating $K_h$ are tangent to the upper and bottom boundaries, and therefore we get a positive topography-adapted operator for smooth temperature with Dirichlet lateral boundary condition,
\[
  (A_{T,h}T,T)_{L^2(\Omega_h)}
  =\int_{\Omega_h}
  \left|\nabla_HT+\frac{z}{h}\nabla_Hh\,\partial_zT\right|^2\geq0.
\]
This definition is consistent with the constant-depth model: if $h(x,y)\equiv h_0$, then $\nabla_Hh=0$, $K_h=\operatorname{diag}(1,1,0)$, and hence $\Delta_{H,h}^{\mathrm{top}}=\Delta_H$.  Under the terrain-following change of variables $z=h(x,y)\sigma$, the operator becomes
\[
  \frac1h\nabla_H\cdot(h\nabla_H)
\]
on $M\times(-1,0)$ and is self-adjoint in the weighted space $L^2(h\,dx\,dy\,d\sigma)$.  This formulation preserves the energy identities and the $H^1$ temperature estimate used in the strong-solution construction.

Returning to the flat-bottom system, we use the hydrostatic relation and incompressibility to rewrite the equations as

\begin{align}
\partial_t v &- \Delta v + (v \cdot \nabla_H) v - \left( \int_{-h}^z \nabla_H \cdot v(x, y, \xi, t) d\xi \right) \partial_z v \nonumber,  \\
&+ f_0 k \times v + \nabla_H \left( p_s(x, y, t) - Q_T(x,y,z,t) \right) = 0 \label{eq:v equation},\\
\partial_t T &- \Delta_H T + v \cdot \nabla_H T - \left( \int_{-h}^z \nabla_H \cdot v(x, y, \xi, t) d\xi \right)  \partial_z T = 0 \label{eq:T equation}.
\end{align}
where
\begin{align}
Q_T(x,y,z,t):=\int_{-h}^{z}T(x,y,\xi,t)\,d\xi,
\label{eq: Q}
\end{align}
and $p_s=p_s(x,y,t)$ is the barotropic pressure, independent of $z$. For a scalar or vector field $\phi$, define
\[
\avg\phi(x,y):=\frac1h\int_{-h}^{0}\phi(x,y,z)\,dz,
\qquad u:=\partial_zv.
\]
Let $n_H$ be the outward unit normal to $\partial M$. We set
\begin{align*}
H&:=\{v\in L^2(\Omega)^2:\nabla_H\cdot\avg v=0\text{ in }M,\ \avg v\cdot n_H=0\text{ on }\partial M\},\\
V&:=\{v\in H^1(\Omega)^2:v=0\text{ on }\Gamma_b\cup\Gamma_s,\ \nabla_H\cdot\avg v=0\text{ in }M\},\\
V_T&:=\{T\in H^1(\Omega):T=0\text{ on }\Gamma_s\}.
\end{align*}
For simplicity, we also set,
\[
  \mathcal V_H:=L^2((-h,0);H_0^1(M)).
\]
Let $\mathbb P:L^2(\Omega)^2\to H$ be the hydrostatic Helmholtz projection. The hydrostatic Stokes operator is
\[
Av:=-\mathbb P\Delta v,
\qquad D(A):=\{v\in H^2(\Omega)^2\cap V:\partial_zv=0\text{ on }\Gamma_t\}.
\]
The above characterization of $D(A)$ follows from the $H^2$ regularity of the hydrostatic Stokes problem; see \cite{Ziane01111995}. $A$ is positive self-adjoint on $H$ and
\[
 D(A^{1/2})=V
\]
see \cite[Chapter~VI, Section~2]{Kato1995}. For $1\leq p\leq\infty$ and $s,r\geq0$, write $L_H^p:=L^p(M)$ and $H_H^s:=H^s(M)$. If $I$ is the vertical interval, we use
\[
 L_z^pH_H^s:=L^p(I;H^s(M)),\qquad
 H_z^rH_H^s:=H^r(I;H^s(M)).
\]
The same convention applies with $\sigma$ in place of $z$.

The systematic mathematical study of the primitive equations was initiated by Lions, Temam and Wang \cite{LionsTemamWang1,LionsTemamWang2,LionsTemamWang3}, who established global existence of weak solutions. Local well-posedness of strong solutions was subsequently proved by Guill\'en-Gonz\'alez, Masmoudi and Rodr\'iguez-Bellido \cite{Guillen}.Under a small-depth hypothesis, Hu, Temam and Ziane \cite{HuTemamZiane2003} established global existence of strong solutions for suitable initial data and forcing. A major breakthrough was achieved by Cao and Titi \cite{CaoTiti2007}, who established global existence and uniqueness of strong solutions with full viscosity and full thermal diffusivity. Alternative proofs and further developments were obtained by Kobelkov \cite{Kobelkov}. For atmospheric models incorporating moisture and phase transitions, Coti Zelati et al.~\cite{CotiZelatiEtAl2015} established global existence of strong and quasi-strong solutions to the primitive equations with vapour saturation, together with uniqueness results and maximum principles.

Anisotropic dissipation is motivated by the disparity between horizontal and vertical mixing in geophysical flows; see, for example, \cite{Vallis,Washington}. For full viscosity and only vertical thermal diffusion, Cao and Titi \cite{CaoTitiVerticalDiffusion} proved global continuation of local strong solutions, while Cao, Li and Titi \cite{CaoLiTiti2014Vertical} established local and global well-posedness for compatible $H^2$ initial data.

For full viscosity and only horizontal thermal diffusion, Cao, Li and Titi \cite{CAO20144108} obtained global strong solutions for $H^2$ initial data, and Pu and Zhou \cite{PuZhou2024} established global well-posedness for $H^1$ initial data. These results use periodic formulations with prescribed vertical symmetries.  The present work treats bounded cylindrical domains with no-slip velocity conditions on the lateral and bottom boundaries and a stress-free condition at the top. The temperature satisfies a lateral Dirichlet condition, without prescribed boundary values at the top or bottom or any vertical symmetry assumption.

Cao, Li and Titi also studied the system with only horizontal
viscosity and diffusivity in \cite{CaoLiTitiHV}.
For this system, they established local well-posedness for
$H^1$ initial data and global well-posedness for
$(v_0,T_0)\in H^1\cap L^\infty$ with
$\partial_zv_0\in L^m$, for some $m>2$, in
\cite{CaoLiTiti2017NearH1}.
For the velocity system without temperature coupling,
Hussein, Saal and Wrona \cite{HusseinSaalWrona2021}
developed a well-posedness theory with only horizontal viscosity
and lateral Dirichlet conditions, without imposing additional
velocity conditions at the top and bottom.

Saal and Slav\'ik \cite{SaalSlavik2023} established global pathwise well-posedness for stochastic primitive equations with only horizontal viscosity and thermal diffusivity on bounded cylindrical domains, under suitable assumptions on the initial data and the noise. They impose lateral Dirichlet conditions on the velocity and lateral Neumann conditions on the temperature. Their strong-solution construction also uses homogeneous Neumann conditions for the velocity at the top and bottom. For the temperature, their homogeneous formulation requires
\[
  T_0\big|_{z=-h}=T_0\big|_{z=0}=0,
\]
and the corresponding solutions preserve
\[
  T(t)\big|_{z=-h}=T(t)\big|_{z=0}=0
\]
throughout the evolution under their compatibility assumptions; see \cite[Remark~2.3]{SaalSlavik2023}. In our deterministic setting with full velocity viscosity, the temperature satisfies a homogeneous lateral Dirichlet condition, but neither the initial temperature $T_0$ nor the evolving temperature $T$ is subject to a boundary condition at the top or bottom.

For full viscosity and only horizontal thermal diffusion,
Li and Yuan \cite{10.1063/5.0065114} established global
well-posedness of $z$-weak solutions on general bounded
cylindrical domains, assuming that the initial data and their
vertical derivatives belong to $L^2$. Their solution class
includes additional spatial and temporal regularity for
positive times. They impose lateral no-slip conditions,
Neumann velocity conditions at both the top and bottom,
and a lateral Robin condition for the temperature.
The present problem instead uses a no-slip bottom and a
homogeneous lateral Dirichlet condition for the temperature.
Neither formulation prescribes temperature values at the
top or bottom.

The mixed velocity conditions in the present problem affect
the vertically averaged momentum balance:
\[
  \frac1h\int_{-h}^{0}\partial_z^2v\,dz
  =-\frac1h\partial_zv\big|_{z=-h}.
\]
This bottom shear term vanishes under homogeneous Neumann
conditions at both horizontal boundaries. To treat it here,
we combine the mixed-boundary velocity estimates of
\cite{Kukavica_2007} with anisotropic estimates for the temperature equation without vertical diffusion.

For the flat-bottom problem, we establish global existence, uniqueness, and continuous dependence of strong solutions for $(v_0,T_0)\in V\times V_T$; see Theorem~\ref{thm:global-well-posedness}. And, we prove persistence of $H^2$ regularity in Theorem~\ref{thm:higher-regularity}. The higher-order argument combines the fractional hydrostatic Stokes estimates obtained from Lemma~\ref{lem:top-stokes-higher} with anisotropic estimates of the vertical velocity to control vertical transport and close the temperature estimates.

For non-flat domains, Hu, Temam and Ziane \cite{HuTemamZiane2002} studied elliptic
regularity for the hydrostatic Stokes problem and related scalar equations in domains with varying bottom topography. Their analysis includes nonhomogeneous boundary conditions and the dependence of the estimates on the domain thickness. Kukavica and Ziane \cite{Kukavica_2007} extended the velocity estimates to a smoothly varying bottom by treating the topographic contributions as lower-order forcing and trace terms. Iida \cite{Iida2025} established global well-posedness in the isothermal setting using maximal $L^q$--$L^p$ regularity, with horizontal periodicity and Dirichlet and Neumann conditions at the bottom and top, respectively. Our variable-depth problem couples the velocity equation under lateral no-slip conditions to the temperature equation with the topography-adapted diffusion defined above.

For positive $h\in C^3(\overline M)$, we obtain global strong solutions for admissible $H^1$ initial data. Persistence of $H^2$ regularity for compatible initial data requires the additional condition $\partial_\nu h=0$ on $\partial M$, where $\nu$ denotes the outward horizontal unit normal. This orthogonal-intersection condition enters the higher hydrostatic Stokes estimates and is not required for the $H^1$ result. After flattening the domain, the thermal operator is independent of the vertical coordinate and commutes with the vertical derivative. Together with its weighted energy identity, this property allows the temperature estimates to be established by the same approach as in the flat-bottom case.

The remainder of the paper is organized as follows. Section~\ref{sec:construction-global} constructs strong solutions for $H^1$ initial data and proves global well-posedness. Section~\ref{sec:higher-regularity} establishes persistence of $H^2$ regularity. Section~\ref{sec:varying-bottom} treats variable-depth domains with topography-adapted thermal diffusion. Appendix~\ref{app:analytic-tools} collects the analytic inequalities and compactness results, while Appendix~\ref{app:auxiliary-estimates} contains the auxiliary a priori estimates used in the global argument.

\section{Construction and Global Regularity of Strong Solutions}
\label{sec:construction-global}
Throughout this section, we work on the flat-bottom cylindrical domain
\[
\Omega=M\times(-h,0),
\]
where \(h>0\) is constant. The purpose of this section is to establish the global existence and uniqueness of strong solutions for initial data \((v_0,T_0)\in V\times V_T\). We first construct a local strong solution through an iterative linearization and Galerkin approximation. We then derive uniform \(H^1\) estimates for the velocity and temperature, establish the required time regularity, and pass to the limit in the approximate system. Finally, the a priori estimates and stability inequality allow the local solution to be continued and imply uniqueness and continuous dependence on the initial data.

The anisotropic inequalities, interpolation estimates, and compactness results used in the proof are collected in the Appendix~\ref{app:analytic-tools}. The extension to a varying bottom, including the corresponding topography-adapted horizontal diffusion operator, is treated separately in Section~\ref{sec:varying-bottom}.
\begin{definition}[Strong solution]\label{def:strong-solution}
Let $\tau>0$, $v_0\in V$, and $T_0\in V_T$. A pair $(v,T)$ is a strong solution on $[0,\tau]$ if
\begin{align*}
v&\in L^\infty(0,\tau;V)\cap L^2(0,\tau;D(A)),&\partial_tv&\in L^2(0,\tau;H),\\
T&\in L^\infty(0,\tau;V_T),&\nabla_HT&\in L^2(0,\tau;H^1(\Omega)),\quad\partial_tT\in L^2(0,\tau;L^2(\Omega)),
\end{align*}
and the reformulated equations hold almost everywhere with the stated boundary and initial conditions.
\end{definition}
\begin{definition}[Maximal existence time]\label{def:maximal-time}
The maximal existence time is $\mathcal T^*:=\sup\{\tau>0:(v,T)\text{ is a strong solution on }[0,\tau]\}$.
\end{definition}
\begin{definition}[Global strong solution]\label{def:global-solution}
A strong solution is global if it exists on $[0,\tau]$ for every finite $\tau>0$.
\end{definition}
\begin{theorem}\label{thm:global-well-posedness}
Let $(v_0,T_0)\in V\times V_T$. Then the primitive equations subject to the physical boundary conditions above admit a unique global strong solution.
\end{theorem}

We now introduce the iterative system. Each iterate satisfies the same boundary and initial conditions as the original problem.
\begin{align}
\partial_t v^n + (v^{n-1} \cdot \nabla_H) v^n + (w^{n-1} \partial_z) v^n &+ \nabla_H p^n - \Delta v^n + f_0 k \times v^n = 0 \\
\partial_z p^n &+ T^n = 0 \\
\nabla_H \cdot v^n &+ \partial_z w^n = 0 \\
\partial_t T^n + (v^{n-1} \cdot \nabla_H) T^n &+ (w^{n-1} \partial_z) T^n - \Delta_H T^n = 0
\end{align}
for $n\geq1$, where the initial iterates $v^0$ and $w^0$ will be specified in the proofs of Propositions~\ref{thm:A^1/2 v} and~\ref{thm:H1 T}. For a fixed $t_0>0$, we introduce the spaces
\begin{align*}
X_{t_0}
&:=L^\infty(0,t_0;V)\cap L^2(0,t_0;D(A)),\\
T_{t_0}
&:=\left\{T\in L^\infty(0,t_0;V_T):
\nabla_HT\in L^2(0,t_0;H^1(\Omega))\right\}.
\end{align*}
Given $v^{n-1}\in X_{t_0}$ and the associated vertical velocity $w^{n-1}$, we construct $(v^n,T^n)\in X_{t_0}\times T_{t_0}$ by first solving the linear temperature equation by the Galerkin construction in Lemma~\ref{lem:linear-temperature} and then applying the Galerkin method to the velocity equation. For each fixed $n$, a priori estimates uniform in the auxiliary approximation parameters, together with compactness, justify passage to the limit; uniqueness follows from an energy estimate for the difference of two solutions.

\begin{lemma}
\label{lem:linear-velocity}
\textbf{The velocity equation} Fixed $t_0$, and $\mathbf{v}^{n-1}$ is given in $X_{t_0}$ associated with $\mathbf{w^{n-1}}$ $\&$ ${T}^*$ is given in $T_{t_0}$, then the following equation has a unique solution $(v,w)$ with $v\in$ $X_{t_0}$.
\begin{align*}
\partial_t v + (\mathbf{v}^{n-1} \cdot \nabla_H) v + (\mathbf{w}^{n-1} \partial_z) v &+ \nabla_H p - \Delta v + f_0 k \times v = 0, \\
\partial_z p &+ T^* = 0, \\
\nabla_H \cdot \mathbf{v} &+ \partial_z \mathbf{w} = 0
\end{align*}
\end{lemma}
\begin{proof}
Uniqueness follows from the $L^2$ energy estimate for the difference
of two solutions. To prove existence, let $\{V_k\}_{k\geq1}\subset D(A)$
be hydrostatic Stokes eigenfunctions satisfying $AV_k=\lambda_kV_k$.
They form an orthonormal basis of $H$ with respect to the
$L^2(\Omega)^2$ inner product and are orthogonal in $V$.
Let $P_m:H\to\operatorname{span}\{V_1,\ldots,V_m\}$ be the
$L^2$-orthogonal projection. For a fixed $m\in\mathbb N$, consider the Galerkin system,
\begin{align*}
\partial_t v^m
&+ P_m\mathbb{P}\Bigl(
    (v^{n-1}\cdot\nabla_H)v^m
    + w^{n-1}\partial_z v^m
  \Bigr)
  - \mathbb{P}\Delta v^m
  \notag\\
&= -P_m\mathbb{P}(f_0 k\times v^m)
   +P_m\mathbb{P}\left(
      \nabla_H\int_{-h}^{z}T^*\,d\xi
    \right),\\
v^m(0)&=P_m v_0.
\end{align*}

Take the inner product of the first equation above with $V_k$. Denote $(v^m, V_k) = g_k(t)$, $\beta_{kj} = \left( \mathbb{P} \left[ (v^{n-1} \cdot \nabla_H) V_j + (w^{n-1} \partial_z) V_j \right], V_k \right)$,$\quad \alpha_{j,k} = (f_0 k \times V_j, V_k), \eta_k=(\nabla_H\int_{-h}^zT^*,V_k)$, therefore, we have
\begin{align*}
g_k'(t) + \lambda_k g_k(t) + \sum g_j (\beta_{kj}+\alpha_{jk}) &= \eta_k
\end{align*}
Let $g^m=(g_1,\dots,g_m)^{T}$,$\alpha^m=(\alpha_{ij})_{1\leq i,j \leq m}$, $\beta^m=(\beta_{ij})_{1\leq i,j \leq m}$, $\eta^m=(\eta_1,\dots,\eta_m)^T$ and $\Lambda^m$ is the diagonal matrix with $j$-th diagonal element $\lambda_j$
\begin{align*}
\dt g^m(t) + [\Lambda^m+\beta^m(t)+(\alpha^m)^T] g^m(t)  &= \eta^m\\
g^m(0)&={g^m_0}
\end{align*}

\noindent where $g_0^m=
\bigl((v_0,V_1)_{L^2},\ldots,(v_0,V_m)_{L^2}
\bigr)^{\mathsf T}$,$\int_0^{t_0}|\Lambda^m+\beta^m(s)+(\alpha^m(s))^T|ds<\infty$. Therefore, because of the ODE theory, it has the unique solution $v^m$ on $[0,t_0]$.
Under classic consideration, we may show $v^m$ is uniformly bounded in $X_{t_0}$ with respect to m. Therefore, by passing to a subsequence, there exists a $v$ such that $v^m \rightarrow v$, which is the solution of,
\begin{align*}
\partial_t v + \mathbb{P} \left( (v^{n-1} \cdot \nabla_H) v + (w^{n-1} \partial_z) v \right) + Av &= - \mathbb{P} (f_0 k \times v)+\mathbb{P}(\nabla_H \int_{-h}^zT^* d\xi) \\
v(0) &= v_0
\end{align*}
which implies that $v$ is the solution of the original equation.
\end{proof}

\begin{lemma}[The linear temperature equation]
\label{lem:linear-temperature}
Let $t_0>0$ and suppose that
\[
v^{n-1}\in X_{t_0}
:=
L^\infty(0,t_0;V)
\cap
L^2(0,t_0;D(A)).
\]
\[
w^{n-1}(x,y,z,t)
:=
-\int_{-h}^{z}
\nabla_H\cdot v^{n-1}(x,y,\xi,t)\,d\xi .
\]
Then, for every $T_0\in V_T$, the problem
\begin{equation}
\label{eq:linear-temperature-original}
\left\{
\begin{aligned}
\partial_tT^n
&+
(v^{n-1}\cdot\nabla_H)T^n
+
w^{n-1}\partial_zT^n
-
\Delta_HT^n
=0
&&\text{in }\Omega\times(0,t_0),\\
T^n&=0
&&\text{on }\Gamma_s\times(0,t_0),\\
T^n(\cdot,0)&=T_0
&&\text{in }\Omega
\end{aligned}
\right.
\end{equation}
admits a unique solution satisfying
\[
T^n\in L^\infty(0,t_0;V_T),
\qquad
\nabla_HT^n\in L^2(0,t_0;H^1(\Omega)),
\qquad
\partial_tT^n\in L^2(0,t_0;L^2(\Omega)).
\]
Moreover,
\begin{equation}
\label{eq:linear-temperature-L2}
\|T^n(t)\|_{L^2}^2
+2\int_0^t\|\nabla_HT^n(s)\|_{L^2}^2\,ds
=\|T_0\|_{L^2}^2,
\qquad 0\leq t\leq t_0,
\end{equation}
and
\begin{align}
&\sup_{0\leq t\leq t_0}
\|\nabla T^n(t)\|_{L^2}^2
\nonumber\\
&\quad+
\int_0^{t_0}
\left(
\|\Delta_HT^n(t)\|_{L^2}^2
+
\|\nabla_H\partial_zT^n(t)\|_{L^2}^2
\right)\,dt
\leq C_{n-1}(t_0)\|T_0\|_{H^1}^2,
\label{eq:linear-temperature-H1}
\end{align}
where $C_{n-1}(t_0)$ depends only on $\Omega$, $t_0$, and
\[
\|v^{n-1}\|_{L^\infty(0,t_0;H^1)}
+
\|v^{n-1}\|_{L^2(0,t_0;H^2)}.
\]
\end{lemma}


\begin{proof}
To define the
temperature projection $P_N^T$, let $\{\phi_j\}_{j\geq1}$ be the
$L^2(M)$-orthonormal basis of horizontal Dirichlet problem, ordered by nondecreasing eigenvalues, i.e.
\[
 -\Delta_H\phi_j=\lambda_j\phi_j\quad\hbox{in }M,
 \qquad \phi_j=0\quad\hbox{on }\partial M.
\]
In the vertical variable, use the $L^2(-h,0)$-orthonormal cosine basis
\[
 \chi_0(z)=h^{-1/2},\qquad
 \chi_k(z)=\sqrt{\frac2h}
       \cos\!\left(\frac{k\pi(z+h)}h\right),\quad k\geq1.
\]

Set $\mu_k=(k\pi/h)^2$ for $k\geq0$, so that
$-\chi_k''=\mu_k\chi_k$ and $\chi_k'(-h)=\chi_k'(0)=0$.
The products $\phi_j(x,y)\chi_k(z)$ form an orthonormal
basis of $L^2(\Omega)$. This basis is the analog of the  basis used by Saal and Slav\'ik
\cite[Eq.~(2.17)]{SaalSlavik2023}. For $N\geq1$, define
\[
 \mathcal E_N:=\operatorname{span}
 \{\phi_j\chi_k:1\leq j\leq N,\ 0\leq k\leq N\}.
\]
Let $P_N^T:L^2(\Omega)\to\mathcal E_N$ be the $L^2$-orthogonal
projection; explicitly,
\[
 P_N^Tf
 =\sum_{j=1}^N\sum_{k=0}^N
    \left(\int_\Omega f(x,y,z)\phi_j(x,y)\chi_k(z)\,dx\,dy\,dz\right)
          \phi_j\chi_k.
\]

\[
 \|f\|_{H^1}^2
 =\sum_{j\geq1}\sum_{k\geq0}
 (1+\lambda_j+\mu_k)|(f,\phi_j\chi_k)|^2,
 \qquad f\in V_T,
\]
implies
\[
 \|P_N^Tf\|_{H^1}\leq\|f\|_{H^1},\qquad
 P_N^Tf\longrightarrow f\quad\text{in }H^1(\Omega).
\]
The Neumann condition satisfied by the basis is auxiliary:
this $H^1$ convergence imposes no vertical endpoint condition
on the limit.

We seek $T_N^n$ satisfying
\begin{equation}
 \begin{aligned}
 (\partial_tT_N^n,\psi)
 +(\nabla_HT_N^n,\nabla_H\psi)
 &+(v^{n-1}\cdot\nabla_HT_N^n
       +w^{n-1}\partial_zT_N^n,\psi)=0,
 \qquad \psi\in\mathcal E_N,\\
 &T_N^n(0)=P_N^TT_0.
 \end{aligned}
 \label{eq:linear-temperature-galerkin}
\end{equation}

This is a finite-dimensional linear system with integrable
time-dependent coefficients, and hence has a unique absolutely
continuous solution on $[0,t_0]$.
The divergence-free condition with boundary conditions gives,
\begin{equation}
 \frac12\frac{d}{dt}\|T_N^n\|_2^2
       +\|\nabla_HT_N^n\|_2^2=0.
 \label{eq:linear-temperature-approximate-L2}
\end{equation}
Testing \eqref{eq:linear-temperature-galerkin} by
$-\Delta_HT_N^n$ and $-\partial_{zz}T_N^n$, respectively, gives
\begin{align*}
 \frac12\frac{d}{dt}\|\nabla_HT_N^n\|_2^2
       +\|\Delta_HT_N^n\|_2^2
 &=\bigl(v^{n-1}\cdot\nabla_HT_N^n
          +w^{n-1}\partial_zT_N^n,\Delta_HT_N^n\bigr),\\
 \frac12\frac{d}{dt}\|\partial_zT_N^n\|_2^2
       +\|\nabla_H\partial_zT_N^n\|_2^2
 &=-\bigl(\partial_zv^{n-1}\cdot\nabla_HT_N^n,
                                      \partial_zT_N^n\bigr)\\
 &\quad+\int_\Omega(\nabla_H\cdot v^{n-1})
                                  |\partial_zT_N^n|^2.
\end{align*}
By Lemma~\ref{anisotropic inequality}, for every $\delta>0$,
\begin{align*}
 \left|\bigl(v^{n-1}\cdot\nabla_HT_N^n
          +w^{n-1}\partial_zT_N^n,\Delta_HT_N^n\bigr)\right|
 &\leq\delta\bigl(\|\Delta_HT_N^n\|_2^2
                   +\|\nabla_H\partial_zT_N^n\|_2^2\bigr)\\
 &+C_\delta\|v^{n-1}\|_{H^2}^2
                                  \|\nabla_HT_N^n\|_2^2\\
 &+C_\delta\bigl(1+\|v^{n-1}\|_{H^1}^4
             +\|v^{n-1}\|_{H^1}^2\|v^{n-1}\|_{H^2}^2\bigr)
                                  \|\partial_zT_N^n\|_2^2.
\end{align*}
And,
\begin{align*}
 &\left|\bigl(\partial_zv^{n-1}\cdot\nabla_HT_N^n,
                                  \partial_zT_N^n\bigr)\right|
 +\left|\int_\Omega(\nabla_H\cdot v^{n-1})
                                  |\partial_zT_N^n|^2\right|\\
 &\quad\leq\delta\bigl(\|\Delta_HT_N^n\|_2^2
                   +\|\nabla_H\partial_zT_N^n\|_2^2\bigr)
 +C_\delta(1+\|v^{n-1}\|_{H^2}^2)\|\nabla T_N^n\|_2^2.
\end{align*}
Choosing $\delta$ sufficiently small therefore yields
\begin{equation}
 \frac12\frac{d}{dt}\|\nabla T_N^n\|_2^2
 +\frac12\bigl(\|\Delta_HT_N^n\|_2^2
              +\|\nabla_H\partial_zT_N^n\|_2^2\bigr)
 \leq\Phi_{n-1}(t)\|\nabla T_N^n\|_2^2,
 \label{eq:linear-temperature-approximate-H1}
\end{equation}
where
\begin{equation}
 \Phi_{n-1}(t)
 :=C\left[1+\|v^{n-1}(t)\|_{H^1}^4
 +(1+\|v^{n-1}(t)\|_{H^1}^2)
                          \|v^{n-1}(t)\|_{H^2}^2\right]
 \in L^1(0,t_0).
 \label{eq:Phi-n}
\end{equation}
Gronwall's inequality gives
\begin{equation}
 \begin{aligned}
 &\sup_{0\leq t\leq t_0}\|T_N^n(t)\|_{H^1}^2
 +\int_0^{t_0}\bigl(\|\Delta_HT_N^n\|_2^2
                  +\|\nabla_H\partial_zT_N^n\|_2^2\bigr)\,dt\\
 &\hspace{30mm}\leq C_{n-1}(t_0)\|T_0\|_{H^1}^2.
 \end{aligned}
 \label{eq:linear-temperature-uniform}
\end{equation}
Here and below the constants are independent of $N$ and have
the dependence asserted in the lemma. Horizontal elliptic
regularity also bounds $\nabla_HT_N^n$ in $L^2(0,t_0;H^1)$.
Also,
\[
 \|\partial_tT_N^n\|_2
 \leq\|\Delta_HT_N^n\|_2
      +\|v^{n-1}\cdot\nabla_HT_N^n\|_2
      +\|w^{n-1}\partial_zT_N^n\|_2.
\]
Using the anisotropic inequality in Lemma~\ref{anisotropic inequality},
\begin{equation}
 \|\partial_tT_N^n\|_{L^2(0,t_0;L^2)}
 \leq C_{n-1}(t_0)\|T_0\|_{H^1}.
 \label{eq:temperature-time-uniform}
\end{equation}

By Lemma~\ref{lem:aubin-lions} in Appendix \ref{app:analytic-tools}, the preceding estimates yield a subsequence such that
\begin{align*}
 T_N^n&\longrightarrow T^n
       &&\text{in }C([0,t_0];L^2),\\
 T_N^n&\rightharpoonup^*T^n
       &&\text{in }L^\infty(0,t_0;V_T),\\
 \nabla_HT_N^n&\rightharpoonup\nabla_HT^n
       &&\text{in }L^2(0,t_0;H^1),\\
 \partial_tT_N^n&\rightharpoonup\partial_tT^n
       &&\text{in }L^2(0,t_0;L^2).
\end{align*}
In particular, $T^n(0)=T_0$, and $T^n$ has zero lateral trace.

To pass to the limit in the equation, write
\[
 G_N=v^{n-1}\cdot\nabla_HT_N^n
                  +w^{n-1}\partial_zT_N^n.
\]
The preceding estimates bound $G_N$ in $L^2(0,t_0;L^2)$.
For every smooth test function $\varphi$, integration by parts gives
\[
 \int_0^{t_0}(G_N,\varphi)\,dt
 =-\int_0^{t_0}
   (T_N^n,v^{n-1}\cdot\nabla_H\varphi
                  +w^{n-1}\partial_z\varphi)\,dt.
\]
Strong $L^2$ convergence of $T_N^n$ and
$v^{n-1},w^{n-1}\in L^2(0,t_0;L^2)$ identify the limit of this expression. Testing the Galerkin equation by $P_N^T\varphi$
and using $P_N^T\varphi\to\varphi$ in $L^2(0,t_0;H^1)$
therefore yields
\begin{align*}
 &\int_0^{t_0}\!\int_\Omega
 \bigl[-T^n\partial_t\varphi
       +\nabla_HT^n\cdot\nabla_H\varphi\\
 &\hspace{28mm}-T^n(v^{n-1}\cdot\nabla_H\varphi
                    +w^{n-1}\partial_z\varphi)\bigr]
 =\int_\Omega T_0\varphi(\cdot,0)
\end{align*}
whenever $\varphi|_{\Gamma_s}=0$ and
$\varphi(\cdot,t_0)=0$.
The values of $\varphi$ at $z=-h,0$ are unrestricted.
Thus $T^n$ solves \eqref{eq:linear-temperature-original}.

Weak lower semicontinuity yields
\eqref{eq:linear-temperature-H1}. The uniform $H^1$ bound and
$T^n\in C([0,t_0];L^2)$ give a weakly $H^1$-continuous
representative, so that the supremum can be taken over all
$t\in[0,t_0]$.
The transport bounds also hold for $T^n$, and hence every term
of its equation belongs to $L^2(0,t_0;L^2)$.
Therefore, we could multiply the equation by $T^n$, and derive the uniqueness and \eqref{eq:linear-temperature-L2}.
\end{proof}

Given $v^{n-1}\in X_{t_0}$ and the associated vertical velocity
$w^{n-1}$, Lemmas~\ref{lem:linear-velocity}
and~\ref{lem:linear-temperature} yield a unique iterate
$(v^n,T^n)\in X_{t_0}\times T_{t_0}$, by first solving the temperature
equation and then the velocity equation. The vertical velocity
$w^n$ is recovered from the incompressibility relation and its
bottom boundary condition. We next establish estimates uniform in $n$.
Since we have already shown the existence of $v_n$, we are allowed to do the estimates for $v_n$.
\begin{proposition}\label{prop:vn-L2}
Assume that $(v_0,T_0)\in V\times V_T$. Then, for every $n\geq1$ and $t>0$,
\begin{equation}
\sup_{0\leq s\leq t}\|v^n(s)\|_{L^2}^2
+\int_0^t\|\nabla v^n(s)\|_{L^2}^2\,ds
\leq 2K_1(t),
\label{eq:v-energy}
\end{equation}
where
\[
K_1(t):=
\|v_0\|_{L^2}^2+C_h t\|T_0\|_{L^2}^2,
\]
and $C_h>0$ depends only on $h$. In particular, the bound is
independent of $n$.
\end{proposition}

\begin{proof}
Multiply the velocity equation by $v^n$, the transport and surface pressure terms vanish by incompressibility, the hydrostatic divergence constraint, and the boundary conditions. The Coriolis term also vanishes since $(k\times v^n)\cdot v^n=0$. It follows by the Cauchy Schwarz inequality, and Young's inequality,
\begin{align*}
\frac12\frac{d}{dt}\|v^n\|_{L^2}^2
+\|\nabla v^n\|_{L^2}^2
&=-\int_\Omega
\left(\int_{-h}^zT^n(\xi)\,d\xi\right)
\nabla_H\cdot v^n\,dx\\
&\leq\frac12\|\nabla v^n\|_{L^2}^2
+C\|T^n\|_{L^2}^2,
\end{align*}
Using \eqref{eq:linear-temperature-L2} and integrating in time, we obtain, for every $s\in[0,t]$,
\[
\|v^n(s)\|_{L^2}^2
+\int_0^s\|\nabla v^n(\tau)\|_{L^2}^2\,d\tau
\leq
\|v_0\|_{L^2}^2+C_hs\|T_0\|_{L^2}^2
=K_1(s).
\]
Since $K_1$ is nondecreasing and both terms on the left are nonnegative, it follows that
\[
\sup_{0\leq s\leq t}\|v^n(s)\|_{L^2}^2\leq K_1(t),
\qquad
\int_0^t\|\nabla v^n(s)\|_{L^2}^2\,ds\leq K_1(t).
\]
Adding these inequalities yields \eqref{eq:v-energy}.
\end{proof}

\noindent\textit{Remark ($L^p$ estimates for the temperature).}
If, in addition, $T_0\in L^p(\Omega)$ for some $2\leq p\leq\infty$,
then
\begin{equation}
\sup_{0\leq s\leq t}\|T^n(s)\|_{L^p}
\leq\|T_0\|_{L^p},
\qquad n\geq1,\quad t>0.
\label{eq:T-Lq}
\end{equation}
Indeed, for $2\leq p<\infty$, multiply temperature equation with $|T^n|^{p-2}T^n$ and integrate to give
\[
\frac1p\frac{d}{dt}\|T^n\|_{L^p}^p
+(p-1)\int_\Omega|T^n|^{p-2}|\nabla_HT^n|^2\,dx
=0.
\]
Thus $H^1$ initial data ensures $T_0 \in L^6$, while additional $L^p$ integrability of the initial temperature gives the corresponding $L^p$ estimate for any larger exponent, the case $p=\infty$ follows by letting $p\to\infty$ when $T_0\in L^\infty(\Omega)$.

\begin{proposition}\label{thm:A^1/2 v}
Assume $(v_0,T_0)\in V\times V_T$. Then there exist a sufficiently small $t^*$ and a constant $K_2>0$, independent of $n$, such that
\begin{equation}
\sup_{0\le s\le t^*}
\|A^{1/2}v^n(s)\|_{L^2}^2
+
\int_0^{t^*}
\|Av^n(s)\|_{L^2}^2\,ds
\le K_2 .
\label{eq:uniform-vn-H1}
\end{equation}
\end{proposition}

\begin{proposition}\label{thm:H1 T}
Under the assumptions of Proposition~\ref{thm:A^1/2 v}, there exists a continuous nondecreasing function $K_3$ on $[0,t^*]$, independent of $n$, such that
\begin{equation}
\sup_{0\le s\le t^*}
\|\nabla T^n(s)\|_{L^2}^2
+
\int_0^{t^*}
\left(
\|\Delta_H T^n(s)\|_{L^2}^2
+
\|\nabla_H\partial_zT^n(s)\|_{L^2}^2
\right)\,ds
\le K_3(t^*).
\label{eq:uniform-Tn-H1}
\end{equation}
\end{proposition}

\begin{proof}
The estimates for $v^n$ and $T^n$ are coupled. More precisely, the
higher-order estimate for $v^n$ involves the $H^1$-estimate of $T^n$,
whereas the latter depends only on the estimates of $v^{n-1}$.
Therefore, we prove the two propositions simultaneously by induction
on $n$.

For convenience, define
\begin{equation}
E_v^n(t)
:=
\sup_{0\le s\le t}
\|A^{1/2}v^n(s)\|_{L^2}^2
+
\int_0^t
\|Av^n(s)\|_{L^2}^2\,ds,
\label{eq:Ev-definition}
\end{equation}
and
\begin{equation}
E_T^n(t)
:=
\sup_{0\le s\le t}
\|\nabla T^n(s)\|_{L^2}^2
+
\int_0^t
\left(
\|\Delta_H T^n(s)\|_{L^2}^2
+
\|\nabla_H\partial_zT^n(s)\|_{L^2}^2
\right)\,ds.
\label{eq:ET-definition}
\end{equation}

\medskip
\noindent\textbf{Step 1. A conditional estimate for $T^n$.}
Let $K_2>0$ be a constant to be fixed in Step~2, and suppose that
\begin{equation}
E_v^{n-1}(t)\leq K_2,
\qquad 0\leq t\leq t^*\leq1.
\label{eq:conditional-velocity-bound}
\end{equation}
The temperature estimate follows directly from the calculation in
Lemma~\ref{lem:linear-temperature}. Indeed, by the norm equivalences
for $A$ and \eqref{eq:Phi-n},
\[
\int_0^t\Phi_{n-1}(s)\,ds
\leq C\bigl[(1+K_2^2)t+(1+K_2)K_2\bigr].
\]
Applying Gronwall's inequality to
\eqref{eq:linear-temperature-approximate-H1} and passing to the limit
as in that lemma therefore gives
\begin{equation}
E_T^n(t)\leq K_3(t)
:=2\|\nabla T_0\|_{L^2}^2
\exp\!\left\{C\bigl[(1+K_2^2)t+(1+K_2)K_2\bigr]\right\},
\label{eq:conditional-temperature-bound}
\end{equation}
where $C$ depends only on $\Omega$. Here the factor $2$ accounts for
the time supremum and the integrated dissipation. In particular,
for fixed $K_2$, this bound is independent of $n$ and
$tK_3(t)\to0$ as $t\downarrow0$.
We fix $K_2$ and close the joint induction in Step~2.

\medskip
\noindent
\textbf{Step 2. The higher-order estimate for $v^n$.}
We now derive the higher-order estimate for $v^n$. A direct
higher-order energy estimate is not sufficient because of the mismatch
of indices in the approximation scheme. We therefore introduce the
decomposition
\begin{equation}
\label{eq:velocity-decomposition}
v^n=\widetilde V+\widetilde v^n,
\qquad
w^n=\widetilde W+\widetilde w^n,
\end{equation}
where $\widetilde V$ is independent of $n$ and solves
\begin{equation}
\label{eq:linear-V}
\left\{
\begin{aligned}
\partial_t\widetilde V
-\Delta\widetilde V
+\nabla_H\widetilde P_s
+f_0k\times\widetilde V
&=0
&&\text{in }\Omega\times(0,t^*),\\
\nabla_H\cdot\int_{-h}^{0}\widetilde V\,dz
&=0
&&\text{in }M,\\
\partial_z\widetilde V&=0
&&\text{on }\Gamma_t,\\
\widetilde V&=0
&&\text{on }\Gamma_b\cup\Gamma_s,\\
\widetilde V(\cdot,0)&=v_0
&&\text{in }\Omega.
\end{aligned}
\right.
\end{equation}
The corresponding vertical velocity is defined by
\begin{equation}
\label{eq:linear-W}
\widetilde W(x,y,z,t)
:=
-\int_{-h}^{z}
\nabla_H\cdot\widetilde V(x,y,\xi,t)\,d\xi .
\end{equation}
In particular,
\[
\widetilde W=0
\qquad\text{on }\Gamma_t\cup\Gamma_b.
\]

For each $n\geq0$, we also set
\begin{equation}
\label{eq:remainder-w}
\widetilde w^n(x,y,z,t)
:=
-\int_{-h}^{z}
\nabla_H\cdot\widetilde v^n(x,y,\xi,t)\,d\xi .
\end{equation}
Then
\[
w^n=\widetilde W+\widetilde w^n.
\]

We choose the initial iterate to be the linear solution:
\begin{equation}
\label{eq:initial-iterate}
v^0:=\widetilde V,
\qquad
w^0:=\widetilde W.
\end{equation}
Consequently,
\begin{equation}
\label{eq:initial-remainder}
\widetilde v^0=0,
\qquad
\widetilde w^0=0.
\end{equation}
Notice that every iterate has the same initial value:
\[
v^n(\cdot,0)=v_0,
\qquad n\geq0,
\]
and hence
\begin{equation}
\label{eq:remainder-zero-initial}
\widetilde v^n(\cdot,0)=0,
\qquad n\geq0.
\end{equation}

Testing the projected form of \eqref{eq:linear-V} with
$\widetilde V$ and $A\widetilde V$, respectively, gives
\begin{align*}
\frac12\frac{d}{dt}\|\widetilde V\|_{L^2}^2
+\|A^{1/2}\widetilde V\|_{L^2}^2
&=0,\\
\frac12\frac{d}{dt}\|A^{1/2}\widetilde V\|_{L^2}^2
+\|A\widetilde V\|_{L^2}^2
&=-f_0(k\times\widetilde V,A\widetilde V)_{L^2}\\
&\leq\frac12\|A\widetilde V\|_{L^2}^2
+\frac{f_0^2}{2}\|\widetilde V\|_{L^2}^2.
\end{align*}
The first identity implies
$\|\widetilde V(t)\|_{L^2}\leq\|v_0\|_{L^2}$.
Integrating the second inequality therefore yields, for
$0\leq t\leq t^*\leq1$,
\begin{align}
E_{\widetilde V}(t)
&:=\sup_{0\leq s\leq t}\|A^{1/2}\widetilde V(s)\|_{L^2}^2
+\int_0^t\|A\widetilde V(s)\|_{L^2}^2\,ds
\nonumber\\
&\leq2\|A^{1/2}v_0\|_{L^2}^2
+2f_0^2\|v_0\|_{L^2}^2=:C_V.
\label{eq:CV-definition}
\end{align}
Moreover, since
$\partial_t\widetilde V=-A\widetilde V
-\mathbb{P}(f_0k\times\widetilde V)$, we have
\begin{equation}
\int_0^t\|\partial_t\widetilde V(s)\|_{L^2}^2\,ds
\leq2\int_0^t\|A\widetilde V(s)\|_{L^2}^2\,ds
+2f_0^2t\|v_0\|_{L^2}^2.
\label{eq:linear-V-time}
\end{equation}

We now derive the equation satisfied by the remainder.
Using
\[
v^{n-1}
=
\widetilde V+\widetilde v^{n-1},
\qquad
w^{n-1}
=
\widetilde W+\widetilde w^{n-1},
\]
and subtracting \eqref{eq:linear-V} from the $n$-th velocity
equation, we obtain
\begin{equation}
\label{eq:remainder-equation}
\left\{
\begin{aligned}
\partial_t\widetilde v^n
-\Delta\widetilde v^n
+\nabla_H\widetilde p_s^n
+f_0k\times\widetilde v^n
={}&
-\bigl(
  (\widetilde V+\widetilde v^{n-1})\cdot\nabla_H
 \bigr)
 (\widetilde V+\widetilde v^n)
\\
&-
(\widetilde W+\widetilde w^{n-1})
\partial_z(\widetilde V+\widetilde v^n)
+\nabla_H Q_T^n
&&\text{in }\Omega\times(0,t^*),\\
\nabla_H\cdot\int_{-h}^{0}\widetilde v^n\,dz
={}&0
&&\text{in }M,\\
\partial_z\widetilde v^n={}&0
&&\text{on }\Gamma_t,\\
\widetilde v^n={}&0
&&\text{on }\Gamma_b\cup\Gamma_s,\\
\widetilde v^n(\cdot,0)={}&0
&&\text{in }\Omega,
\end{aligned}
\right.
\end{equation}
where
\[
\widetilde p_s^n:=p_s^n-\widetilde P_s.
\]

The essential point is that the higher-order energy of
$\widetilde v^n$ vanishes initially. We therefore define
\begin{equation}
\label{eq:remainder-energy}
E_{\widetilde v}^n(t)
:=
\sup_{0\leq s\leq t}
\|A^{1/2}\widetilde v^n(s)\|_{L^2}^2
+
\int_0^t
\|A\widetilde v^n(s)\|_{L^2}^2\,ds.
\end{equation}

Fix
\[
K_2:=2(C_V+1).
\]
We prove by induction that, for a sufficiently small $t^*>0$
independent of $n$,
\begin{equation}
E_{\widetilde v}^n(t^*)\leq1,
\qquad n\geq0.
\label{eq:remainder-induction-goal}
\end{equation}
The case $n=0$ follows from \eqref{eq:initial-remainder}.
Assume the bound holds for $n-1$. Then, for $0\leq t\leq t^*$,
\[
E_v^{n-1}(t)
\leq2E_{\widetilde V}(t)+2E_{\widetilde v}^{n-1}(t)
\leq K_2,
\]
so Step~1 gives $E_T^n(t)\leq K_3(t)$.

For $a,b\in D(A)$, let
$w(a):=-\int_{-h}^z\nabla_H\cdot a(\xi)\,d\xi$.
The Sobolev embeddings and the anisotropic product estimate,
together with the norm equivalences for $A$, give
\[
\|(a\cdot\nabla_H)b+w(a)\partial_zb\|_{L^2}^2
\leq C\|A^{1/2}a\|_{L^2}\|Aa\|_{L^2}
       \|A^{1/2}b\|_{L^2}\|Ab\|_{L^2}.
\]
Testing \eqref{eq:remainder-equation} by $A\widetilde v^n$,
expanding the transport terms, and applying this estimate and
Young's inequality, we obtain
\begin{align*}
&\frac{d}{dt}\|A^{1/2}\widetilde v^n\|_{L^2}^2
+\|A\widetilde v^n\|_{L^2}^2\\
&\quad\leq C\Bigl(
1+\|A^{1/2}\widetilde V\|_{L^2}^2\|A\widetilde V\|_{L^2}^2
+\|A^{1/2}\widetilde v^{n-1}\|_{L^2}^2
 \|A\widetilde v^{n-1}\|_{L^2}^2\Bigr)
\|A^{1/2}\widetilde v^n\|_{L^2}^2\\
&\qquad+C\Bigl[
\|A^{1/2}\widetilde V\|_{L^2}^2\|A\widetilde V\|_{L^2}^2
+\|A^{1/2}\widetilde v^{n-1}\|_{L^2}
 \|A\widetilde v^{n-1}\|_{L^2}
 \|A^{1/2}\widetilde V\|_{L^2}\|A\widetilde V\|_{L^2}
+\|\nabla_HT^n\|_{L^2}^2\Bigr].
\end{align*}
Here and below $C$ depends only on $\Omega$ and $f_0$.
Set
\[
\rho(t):=\int_0^t\|A\widetilde V(s)\|_{L^2}^2\,ds.
\]
Since $\widetilde V$ is fixed, $\rho(t)\to0$ as $t\downarrow0$.
The induction hypothesis and the bound for $\widetilde V$ imply
\begin{align*}
\int_0^t\|A^{1/2}\widetilde v^{n-1}\|_{L^2}^2
             \|A\widetilde v^{n-1}\|_{L^2}^2\,ds
&\leq1,\\
\int_0^t\|A^{1/2}\widetilde V\|_{L^2}^2
             \|A\widetilde V\|_{L^2}^2\,ds
&\leq C_V\rho(t),\\
\int_0^t\|A^{1/2}\widetilde v^{n-1}\|_{L^2}
             \|A\widetilde v^{n-1}\|_{L^2}
             \|A^{1/2}\widetilde V\|_{L^2}
             \|A\widetilde V\|_{L^2}\,ds
&\leq\sqrt{C_V\rho(t)}.
\end{align*}
Consequently, Gronwall's inequality and the zero initial value of
$\widetilde v^n$ yield
\[
E_{\widetilde v}^n(t)
\leq C\exp\!\bigl\{C[t+C_V\rho(t)+1]\bigr\}
\Bigl[C_V\rho(t)+\sqrt{C_V\rho(t)}+tK_3(t)\Bigr].
\]
The right-hand side is independent of $n$ and tends to zero as
$t\downarrow0$. We may therefore choose $t^*\in(0,1]$ so that it
is at most $1$ at $t=t^*$. This closes the induction. Finally,
\[
E_v^n(t^*)\leq2(C_V+1)=K_2,
\qquad
E_T^n(t^*)\leq K_3(t^*),
\]
which proves both propositions.
\end{proof}

\begin{proposition}[Time regularity]
\label{prop:iteration-time-regularity}
Let $(v^n,T^n)$ be the iterates constructed above, and let
$0<t^*\leq1$ be such that \eqref{eq:uniform-vn-H1} and
\eqref{eq:uniform-Tn-H1} hold. Then there exists a constant $K_4$,
independent of $n$, such that
\[
 \int_0^{t^*}\left(
       \|\partial_tv^n(s)\|_{L^2}^2
       +\|\partial_tT^n(s)\|_{L^2}^2\right)\,ds
 \leq K_4.
\]
\end{proposition}

\begin{proof}
The temperature estimate follows from the time-derivative argument
in Lemma~\ref{lem:linear-temperature}, leading to
\eqref{eq:temperature-time-uniform}. Indeed, inserting
\eqref{eq:uniform-vn-H1} and \eqref{eq:uniform-Tn-H1} into that
argument gives, for the limiting iterate,
\[
 \int_0^{t^*}\|\partial_tT^n\|_{L^2}^2\,ds
 \leq C(1+K_2)K_3(t^*).
\]
Here the velocity bound also holds for the initial iterate $v^0$
by its construction, so the constant is uniform for all $n\geq1$.

Applying the hydrostatic Helmholtz projection to the velocity
equation gives
\begin{align*}
 \partial_tv^n+Av^n
 ={}&-\mathbb P\bigl[(v^{n-1}\cdot\nabla_H)v^n
              +w^{n-1}\partial_zv^n+f_0k\times v^n\bigr]\\
 &+\mathbb P\left(\int_{-h}^z\nabla_HT^n\,d\xi\right).
\end{align*}
The transport estimate used in the joint proof of
Propositions~\ref{thm:A^1/2 v} and~\ref{thm:H1 T}, together with
Cauchy--Schwarz in time and \eqref{eq:uniform-vn-H1}, yields
\[
 \int_0^{t^*}
 \|(v^{n-1}\cdot\nabla_H)v^n
                +w^{n-1}\partial_zv^n\|_{L^2}^2\,ds
 \leq CK_2^2.
\]
Since $\mathbb P$ is an $L^2$ contraction, the projected equation
and the uniform spatial estimates imply
\[
 \int_0^{t^*}\|\partial_tv^n\|_{L^2}^2\,ds
 \leq C\bigl[(1+t^*)K_2+K_2^2+t^*K_3(t^*)\bigr].
\]
Combining the two estimates proves the assertion with
\[
 K_4=C(1+K_2)\bigl[K_2+K_3(t^*)\bigr],
\]
where $C$ depends only on the fixed parameters.
\end{proof}

\begin{proposition}[Strong convergence of the iteration]
\label{prop:iteration-convergence}
After decreasing the common existence time $t^*>0$ if necessary,
the entire sequence $(v^n,T^n)$ converges strongly in
\[
 \bigl(C([0,t^*];V)\cap L^2(0,t^*;D(A))\bigr)
 \times
 \bigl(C([0,t^*];L^2(\Omega))
       \cap L^2(0,t^*;L^2((-h,0);H_0^1(M)))\bigr).
\]
In particular, $v^{n-1}$ and $v^n$ converge to the same velocity.
\end{proposition}

\begin{proof}
For $n\geq2$, denote $U^n:=v^n-v^{n-1},\Theta^n:=T^n-T^{n-1}, w(U^n):=-\int_{-h}^z\nabla_H\cdot U^n(\xi)\,d\xi.$
Starting at $n=2$ avoids introducing a temperature iterate $T^0$.
Write $B(a,f):=a\cdot\nabla_Hf+w(a)\partial_zf$. Subtracting consecutive projected velocity equations and temperature equations gives
\begin{equation}
 \label{eq:joint-iteration-difference-system}
 \begin{aligned}
 \partial_tU^{n+1}+AU^{n+1}
       +\mathbb P B(v^n,U^{n+1})
       +\mathbb P(f_0k\times U^{n+1})
 &=-\mathbb P B(U^n,v^n)
       +\mathbb P\left(\nabla_H\int_{-h}^z\Theta^{n+1}(\xi)\,d\xi\right),\\
 \partial_t\Theta^{n+1}-\Delta_H\Theta^{n+1}
       +B(v^n,\Theta^{n+1})
 &=-B(U^n,T^n),
 \end{aligned}
\end{equation}
with $(U^{n+1},\Theta^{n+1})(0)=0$ and the inherited homogeneous boundary conditions. We multiply the velocity equation and temperature equation by $AU^{n+1}$ and $\Theta^{n+1}$ respectively, and let 
\[
 M_T:=\sup_{n\geq1}\|T^n\|_{L^\infty(0,t^*;H^1)}<\infty,
\]
In the temperature equation, For every $\eta>0$,
\[
 \begin{aligned}
 |(U^n\cdot\nabla_HT^n,\Theta^{n+1})|
 &\leq CM_T\|AU^n\|_2\|\Theta^{n+1}\|_2\\
 &\leq\eta\|AU^n\|_2^2+C_\eta M_T^2\|\Theta^{n+1}\|_2^2,
 \end{aligned}
\]
using anisotropic estimates, we obtain,
\[
 \begin{aligned}
 |(w(U^n)\partial_zT^n,\Theta^{n+1})|
 &\leq CM_T\|A^{1/2}U^n\|_2^{1/2}\|AU^n\|_2^{1/2}
           \|\Theta^{n+1}\|_2^{1/2}
           \|\nabla_H\Theta^{n+1}\|_2^{1/2}\\
 &\leq\tfrac14\|\nabla_H\Theta^{n+1}\|_2^2
       +\eta\|AU^n\|_2^2
       +C_\eta M_T^2\bigl(\|A^{1/2}U^n\|_2^2
                              +\|\Theta^{n+1}\|_2^2\bigr).
 \end{aligned}
\]
Consequently,
$\eta>0$,
\begin{equation}
 \label{eq:joint-iteration-temperature-energy}
 \frac{d}{dt}\|\Theta^{n+1}\|_2^2
       +\|\nabla_H\Theta^{n+1}\|_2^2
 \leq\eta\|AU^n\|_2^2
       +C_\eta M_T^2\bigl(\|A^{1/2}U^n\|_2^2
                              +\|\Theta^{n+1}\|_2^2\bigr).
\end{equation}

For the velocity equation, use the estimate from the joint proof of Propositions~\ref{thm:A^1/2 v} and~\ref{thm:H1 T},
\begin{equation}
 \label{eq:joint-iteration-velocity-energy}
 \begin{aligned}
 \frac{d}{dt}\|A^{1/2}U^{n+1}\|_2^2+\|AU^{n+1}\|_2^2
 &\leq
 C\bigl(1+\|A^{1/2}v^n\|_2^2\|Av^n\|_2^2\bigr)
                         \|A^{1/2}U^{n+1}\|_2^2\\
 &+C_\eta\|A^{1/2}v^n\|_2^2\|Av^n\|_2^2
                         \|A^{1/2}U^n\|_2^2
       +\eta\|AU^n\|_2^2
       +C_0\|\nabla_H\Theta^{n+1}\|_2^2.
 \end{aligned}
\end{equation}
Here $C_0$ is fixed independently of $\eta$ and $n$.
Choose a fixed $\kappa\geq\max\{1,2C_0\}$ and define
\[
 Y_n:=\|A^{1/2}U^n\|_2^2+\kappa\|\Theta^n\|_2^2,
 \qquad
 D_n:=\|AU^n\|_2^2+\frac\kappa2\|\nabla_H\Theta^n\|_2^2.
\]
Multiplying \eqref{eq:joint-iteration-temperature-energy} by
$\kappa$, adding \eqref{eq:joint-iteration-velocity-energy}, and
choosing $\eta$ in terms of $\varepsilon>0$, we obtain
\begin{equation}
 \label{eq:joint-iteration-combined-energy}
 Y_{n+1}'+D_{n+1}
 \leq G_{n,\varepsilon}(t)(Y_{n+1}+Y_n)+\varepsilon D_n,
 \qquad
 G_{n,\varepsilon}
 :=C_\varepsilon\bigl(1+M_T^2+
             \|A^{1/2}v^n\|_2^2\|Av^n\|_2^2\bigr).
\end{equation}
The decomposition $v^n=\widetilde V+\widetilde v^n$ implies
\begin{equation}
 \label{eq:joint-iteration-small-dissipation}
 r(t):=\sup_{n\geq0}\int_0^t\|Av^n(s)\|_2^2\,ds
 \leq2\int_0^t\|A\widetilde V(s)\|_2^2\,ds+2\sup_{n\geq0}E_{\widetilde v}^n(t)
 \longrightarrow0.
\end{equation}
For $0<\tau\leq t^*$, set
\[
 E_{\mathrm{diff}}^n(\tau)
 :=\sup_{0\leq s\leq\tau}Y_n(s)+\int_0^\tau D_n(s)\,ds,
 \qquad
 \mu_\varepsilon(\tau)
 :=C_\varepsilon\bigl[(1+M_T^2)\tau+K_2r(\tau)\bigr].
\]
By \eqref{eq:uniform-vn-H1},
$\int_0^\tau G_{n,\varepsilon}(s)\,ds
\leq\mu_\varepsilon(\tau)$ uniformly in $n$. Apply Gronwall's inequality to Equation~\eqref{eq:joint-iteration-combined-energy},
\[
 E_{\mathrm{diff}}^{n+1}(\tau)
 \leq2e^{\mu_\varepsilon(\tau)}
       \bigl[\varepsilon+\mu_\varepsilon(\tau)\bigr]
       E_{\mathrm{diff}}^n(\tau).
\]
Choosing \(\varepsilon>0\) sufficiently small and then \(\tau\in(0,t^*]\) sufficiently small, independently of \(n\), we obtain
\[
E_{\mathrm{diff}}^{n+1}(\tau)\le \frac{1}{4}\,E_{\mathrm{diff}}^n(\tau),
\qquad n\ge2,
\]
Thus the entire sequences are Cauchy in the stated spaces, with time continuity following from the time-derivative bounds. Replacing \(t^*\) by \(\tau\) proves the assertion. In particular, \(v^n\) and \(v^{n-1}\) converge strongly to the same limit.
\end{proof}

\begin{theorem}
Let $v_0\in V$ and $T_0\in V_T(\Omega)$, then system~\eqref{eq:primitive-momentum}--\eqref{eq:primitive-initial-data} has a strong solution $(v,T)$
in $\Omega\times(0,t^*)$, where $t^*$ is the same positive time stated in the previous propositions.
\end{theorem}

\begin{proof}
We use the common time $t^*$ after the possible reduction in
Proposition~\ref{prop:iteration-convergence}. All preceding
uniform estimates remain valid on this interval.
By Propositions~\ref{thm:A^1/2 v}, \ref{thm:H1 T},
and~\ref{prop:iteration-time-regularity}, for all $n$, 
\begin{equation}
\begin{aligned}
&\sup_{0\leq t\leq t^*}
\left(
\|v^n\|_{H^1}^2
+\|T^n\|_{H^1}^2
\right)
+\int_0^{t^*}
\left(
\| v^n\|_{H^2}^2
+\|\nabla_H T^n\|_{H^1}^2
\right)\,dt
\leq C ,
\end{aligned}
\end{equation}
and
\begin{equation}
\int_0^{t^*}
\left(
\|\partial_t v^n\|_{L^2}^2
+
\|\partial_t T^n\|_{L^2}^2
\right)\,dt
\leq C ,
\end{equation}
where $C$ is a constant independent of $n$.

In view of the above uniform estimates, Lemma~\ref{lem:aubin-lions} yields a subsequence,
still denoted by $\{(v^n,T^n)\}$, and a pair $(v,T)$ such that
\begin{align}
(v^n,T^n) &\to (v,T)
&& \text{in } C([0,t^*];L^2(\Omega)), \label{eq:iteration-strong-C-L2}\\
(\nabla v^n,\nabla_H T^n) &\to (\nabla v,\nabla_H T)
&& \text{in } L^2(0,t^*;L^2(\Omega)), \label{eq:iteration-strong-gradient-L2}\\
(v^n,T^n) &\overset{*}{\rightharpoonup} (v,T)
&& \text{in } L^\infty(0,t^*;H^1(\Omega)), \label{eq:iteration-weak-star-H1}\\
(\nabla v^n,\nabla_H T^n) &\rightharpoonup
(\nabla v,\nabla_H T)
&& \text{in } L^2(0,t^*;H^1(\Omega)), \label{eq:iteration-weak-gradient-H1}\\
(\partial_t v^n,\partial_t T^n) &\rightharpoonup
(\partial_t v,\partial_t T)
&& \text{in } L^2(\Omega\times(0,t^*)). \label{eq:iteration-weak-time-derivatives}
\end{align}
Here, $\rightharpoonup$ and $\overset{*}{\rightharpoonup}$ denote the weak
and weak-* convergences, respectively.
Moreover, Proposition~\ref{prop:iteration-convergence} shows that
the entire velocity sequence converges in $X_{t^*}$. Hence the
preceding iterates also satisfy
\[
 v^{n-1}\longrightarrow v\qquad\text{strongly in }X_{t^*}.
\]
Thus the transporting velocities and the velocity iterates
have the same limit. By virtue of
\eqref{eq:iteration-strong-C-L2}--\eqref{eq:iteration-weak-time-derivatives}, we may pass to the limit in the approximate system and
conclude that $(v,T)$ is a strong solution to the original system on
$(0,t^*)$.

By the local construction and the uniqueness estimate in
Proposition~\ref{prop:strong-solution-stability}, the solution extends
to a unique maximal strong solution on $[0,\mathcal T^*)$.
Suppose that $\mathcal T^*<\infty$.
Propositions~\ref{prop:appendix-basic-energy},
\ref{prop:appendix-velocity-L6-vertical},
\ref{prop:appendix-velocity-H1}, and
\ref{prop:appendix-temperature-H1} give bounds that remain finite
as $t\uparrow\mathcal T^*$. Consequently,
\[
\begin{gathered}
v\in L^\infty(0,\mathcal T^*;V)
       \cap L^2(0,\mathcal T^*;D(A)),\\
T\in L^\infty(0,\mathcal T^*;V_T),\qquad
\nabla_HT\in L^2(0,\mathcal T^*;H^1(\Omega)).
\end{gathered}
\]
Applying to the original equations the same product estimates used in Proposition~\ref{prop:iteration-time-regularity}, we obtain
\[
v_t\in L^2(0,\mathcal T^*;H),\qquad
T_t\in L^2(0,\mathcal T^*;L^2(\Omega)).
\]
The Hilbert-space trace theorem
\cite{LionsMagenes1972}, together with
$D(A^{1/2})=V$, gives
\[
 L^2(0,\mathcal T^*;D(A))
 \cap H^1(0,\mathcal T^*;H)
 \hookrightarrow
 C([0,\mathcal T^*];D(A^{1/2}))
 =C([0,\mathcal T^*];V).
\]
Together with
\[
 T\in H^1(0,\mathcal T^*;L^2(\Omega))
 \hookrightarrow C([0,\mathcal T^*];L^2(\Omega)),
\]
this yields the limits
\[
v(t)\longrightarrow v_*\quad\text{in }V,\qquad
T(t)\longrightarrow T_*\quad\text{in }L^2(\Omega)
\quad\text{as }t\uparrow\mathcal T^*.
\]
The uniform $H^1$ bound for $T$ and the weak closedness of $V_T$
imply $T_*\in V_T$.

Applying the local construction at time $\mathcal T^*$ with initial
data $(v_*,T_*)\in V\times V_T$ gives a strong solution on
$[\mathcal T^*,\mathcal T^*+\delta]$ for some $\delta>0$.
The matching $L^2$ time traces allow this solution to be joined to
the preceding one; the joined solution satisfies the equations
distributionally and retains all the regularity required in
Definition~\ref{def:strong-solution}. It is therefore a strong solution
beyond $\mathcal T^*$, contradicting maximality.
Thus $\mathcal T^*=\infty$, which, together with uniqueness, proves
Theorem~\ref{thm:global-well-posedness}.
\end{proof}

At the end of this section, we also state some regularity results and the continuous dependence on initial data.

\begin{proposition}
In $\Omega \times [0,t_0)$, with $v_0 \in V$ and $T_0 \in V_T$, there exists the following continuous functions $M_1, M_2, M_3$, such that, 
\begin{align*}
\sup_{0 \leq s \leq t} [\norm{v}_6^2+\norm{\partial_zv}_2^2] + \int_0^t \norm{\nabla u}_2^2\leq M_1(t) \\
\sup_{0 \leq s \leq t} \norm{A^{1/2} v}_2^2 + \int_0^t \norm{Av}_H^2 \leq M_2(t) \\
\sup_{0 \le s \le t} \left(  \|\nabla T\|_2^2 \right) + \int_0^t  \|\nbH \nabla T\|_2^2  \le M_3(t)
\end{align*}
\end{proposition}
\begin{proof}
The three bounds are proved in
Propositions~\ref{prop:appendix-velocity-L6-vertical},
\ref{prop:appendix-velocity-H1}, and~\ref{prop:appendix-temperature-H1},
respectively, in Appendix~\ref{app:auxiliary-estimates}.
\end{proof}
\begin{proposition}[Uniqueness and continuous dependence]
\label{prop:strong-solution-stability}
Let $t_0>0$, and let $(v^{(1)},T^{(1)})$ and
$(v^{(2)},T^{(2)})$ be two strong solutions on $[0,t_0]$ in the
sense of Definition~\ref{def:strong-solution}, with initial data
$(v_0^{(1)},T_0^{(1)})$ and $(v_0^{(2)},T_0^{(2)})$ in
$V\times V_T$, respectively. Define
\begin{equation}
\label{eq:coupled-stability-coefficient}
\begin{split}
G(t):=C\Bigl[&1+
 \bigl(1+\|v^{(2)}(t)\|_{H^1}^2\bigr)
          \|v^{(2)}(t)\|_{H^2}^2
 +\|\nabla_HT^{(2)}(t)\|_{H^1}^2\\
 &+\|\partial_zT^{(2)}(t)\|_2^2
 \bigl(\|\partial_zT^{(2)}(t)\|_2^2
       +\|\nabla_H\partial_zT^{(2)}(t)\|_2^2\bigr)\Bigr]
\end{split}
\end{equation}
where $C$ depends only on $\Omega$. Then $G\in L^1(0,t_0)$ and, for every $0\leq t\leq t_0$, we have 
\begin{equation}
\label{eq:coupled-strong-stability}
\begin{split}
 &\|v^{(1)}(t)-v^{(2)}(t)\|_2^2
       +\|T^{(1)}(t)-T^{(2)}(t)\|_2^2\\
 &\quad+\int_0^t\Bigl(
       \|\nabla(v^{(1)}-v^{(2)})\|_2^2
       +\|\nabla_H(T^{(1)}-T^{(2)})\|_2^2\Bigr)\,ds\\
 &\qquad\leq
 \Bigl(\|v_0^{(1)}-v_0^{(2)}\|_2^2
          +\|T_0^{(1)}-T_0^{(2)}\|_2^2\Bigr)
       \exp\!\left(\int_0^tG(s)\,ds\right).
\end{split}
\end{equation}
In particular, the coupled strong solution is unique and depends
continuously on its initial data in the $L^2$ energy norm.
\end{proposition}

\begin{proof}
For $i=1,2$, recall the vertical velocity is given by
\[
 w^{(i)}(z)=-\int_{-h}^z\nabla_H\cdot v^{(i)}(\xi)\,d\xi.
\]
The differences inherit all the boundary conditions. The estimates below are rigorously justified at level of theorem \ref{thm:global-well-posedness} with the regularity in Definition~\ref{def:strong-solution}. Indeed, both differences
belong to $H^1(0,t_0;L^2)$, so their squared $L^2$ norms are absolutely continuous and satisfy the Hilbert-space chain rule.
Subtract the two equations, retain $(v^{(1)},w^{(1)})$ as the
transporting velocity, and multiply by $v^{(1)}-v^{(2)}$ and $T^{(1)}-T^{(2)}$, respectively. Similarly as the $L^2$ energy estimates, it follows that
\begin{align*}
 \frac12\frac{d}{dt}\Bigl(
       \|v^{(1)}-v^{(2)}\|_2^2
       +\|T^{(1)}-T^{(2)}\|_2^2\Bigr)
 &+\|\nabla(v^{(1)}-v^{(2)})\|_2^2
       +\|\nabla_H(T^{(1)}-T^{(2)})\|_2^2\\
 &=-\int_\Omega
       \bigl((v^{(1)}-v^{(2)})\cdot\nabla_H\bigr)v^{(2)}
                           \cdot(v^{(1)}-v^{(2)})\\
 &\quad-\int_\Omega(w^{(1)}-w^{(2)})\partial_zv^{(2)}
                           \cdot(v^{(1)}-v^{(2)})\\
 &\quad+\int_\Omega\left(
       \int_{-h}^z\nabla_H(T^{(1)}-T^{(2)})(\xi)\,d\xi\right)
                           \cdot(v^{(1)}-v^{(2)})\\
 &\quad-\int_\Omega
       \bigl((v^{(1)}-v^{(2)})\cdot\nabla_HT^{(2)}\bigr)
                           (T^{(1)}-T^{(2)})\\
 &\quad-\int_\Omega(w^{(1)}-w^{(2)})\partial_zT^{(2)}
                           (T^{(1)}-T^{(2)}).
\end{align*}
For any $\eta>0$, H\"older's inequality, Sobolev embedding, and Poincar\'e's
inequality give
\begin{align*}
 &\left|\int_\Omega
       \bigl((v^{(1)}-v^{(2)})\cdot\nabla_H\bigr)v^{(2)}
                           \cdot(v^{(1)}-v^{(2)})\right|\\
 &\quad\leq\|\nabla_Hv^{(2)}\|_3
                           \|v^{(1)}-v^{(2)}\|_3^2\\
 &\quad\leq\eta\|\nabla(v^{(1)}-v^{(2)})\|_2^2
       +C_\eta\|v^{(2)}\|_{H^2}^2
                           \|v^{(1)}-v^{(2)}\|_2^2,
\end{align*}
and
\begin{align*}
 &\left|\int_\Omega
       \bigl((v^{(1)}-v^{(2)})\cdot\nabla_HT^{(2)}\bigr)
                           (T^{(1)}-T^{(2)})\right|\\
 &\quad\leq\|v^{(1)}-v^{(2)}\|_6
       \|\nabla_HT^{(2)}\|_3\|T^{(1)}-T^{(2)}\|_2\\
 &\quad\leq\eta\|\nabla(v^{(1)}-v^{(2)})\|_2^2
       +C_\eta\|\nabla_HT^{(2)}\|_{H^1}^2
                           \|T^{(1)}-T^{(2)}\|_2^2.
\end{align*}
\begin{align*}
 &\left|\int_\Omega\left(
       \int_{-h}^z\nabla_H(T^{(1)}-T^{(2)})(\xi)\,d\xi\right)
                           \cdot(v^{(1)}-v^{(2)})\right|\\
 &\quad\leq C_h\|T^{(1)}-T^{(2)}\|_2
                           \|\nabla(v^{(1)}-v^{(2)})\|_2\\
 &\quad\leq\eta\|\nabla(v^{(1)}-v^{(2)})\|_2^2
       +C_\eta\|T^{(1)}-T^{(2)}\|_2^2.
\end{align*}

To estimate the two vertical transport terms, use anisotropic inequality Lemma~\ref{anisotropic inequality},
\begin{align*}
 &\left|\int_\Omega(w^{(1)}-w^{(2)})\partial_zv^{(2)}
                           \cdot(v^{(1)}-v^{(2)})\right|\\
 &\quad\leq C\|v^{(2)}\|_{H^1}^{1/2}
              \|v^{(2)}\|_{H^2}^{1/2}
              \|v^{(1)}-v^{(2)}\|_2^{1/2}
              \|\nabla(v^{(1)}-v^{(2)})\|_2^{3/2}\\
 &\quad\leq\eta\|\nabla(v^{(1)}-v^{(2)})\|_2^2
       +C_\eta\|v^{(2)}\|_{H^1}^2\|v^{(2)}\|_{H^2}^2
                           \|v^{(1)}-v^{(2)}\|_2^2.
\end{align*}
For the temperature term, the anisotropic inequality yields,
\begin{align*}
 &\left|\int_\Omega(w^{(1)}-w^{(2)})\partial_zT^{(2)}
                           (T^{(1)}-T^{(2)})\right|\\
 &\quad\leq C\|\nabla(v^{(1)}-v^{(2)})\|_2
          \|\partial_zT^{(2)}\|_2^{1/2}
          \Bigl(\|\partial_zT^{(2)}\|_2^{1/2}
                +\|\nabla_H\partial_zT^{(2)}\|_2^{1/2}\Bigr)\\
 &\hspace{30mm}\cdot\|T^{(1)}-T^{(2)}\|_2^{1/2}
          \|\nabla_H(T^{(1)}-T^{(2)})\|_2^{1/2}\\
 &\quad\leq\eta\Bigl(
          \|\nabla(v^{(1)}-v^{(2)})\|_2^2
          +\|\nabla_H(T^{(1)}-T^{(2)})\|_2^2\Bigr)\\
 &\qquad+C_\eta\|\partial_zT^{(2)}\|_2^2
          \Bigl(\|\partial_zT^{(2)}\|_2^2
                +\|\nabla_H\partial_zT^{(2)}\|_2^2\Bigr)
                           \|T^{(1)}-T^{(2)}\|_2^2.
\end{align*}
Choose $\eta$ sufficiently small so that all the gradient terms could be absorbed, we obtain
\begin{equation}
\label{eq:coupled-stability-differential}
\begin{split}
 \frac{d}{dt}\Bigl(
       \|v^{(1)}-v^{(2)}\|_2^2
       +\|T^{(1)}-T^{(2)}\|_2^2\Bigr)
 &+\|\nabla(v^{(1)}-v^{(2)})\|_2^2
       +\|\nabla_H(T^{(1)}-T^{(2)})\|_2^2\\
 &\leq G(t)\Bigl(
       \|v^{(1)}-v^{(2)}\|_2^2
       +\|T^{(1)}-T^{(2)}\|_2^2\Bigr).
\end{split}
\end{equation}
The coefficient is integrable using only
\[
 v^{(2)}\in L^\infty_tH^1\cap L^2_tH^2,\qquad
 \partial_zT^{(2)}\in L^\infty_tL^2,\qquad
 \nabla_HT^{(2)}\in L^2_tH^1.
\]
Gronwall's inequality followed by integration of
\eqref{eq:coupled-stability-differential} proves
\eqref{eq:coupled-strong-stability}. The $L^2$ time continuity of both solutions guarantees the initial and terminal values in the estimate.
\end{proof}

\section{The Persistence of $H^2$ Regularity}
\label{sec:higher-regularity}
In this section, we establish persistence of higher regularity for initial data $v_0\in D(A)$ and $T_0\in H^2(\Omega)\cap V_T$. We work with the unique global strong solution furnished by Theorem~\ref{thm:global-well-posedness}. Recall that $\mathcal T^*$ denotes its maximal existence time, as defined in Definition~\ref{def:maximal-time}; the global existence result gives $\mathcal T^*=\infty$. All estimates below are established on arbitrary finite time intervals $[0,t_0]$, with constants allowed to depend on $t_0$.
\begin{theorem}\label{thm:higher-regularity}
Given a positive number $t_0$. Let $v_0 \in D(A) $ and $T_0 \in H^2(\Omega)\cap V_T$ be two functions, then there exists a unique strong solution to system~\eqref{eq:primitive-momentum}--\eqref{eq:primitive-initial-data} on $\Omega \times (0,t_0)$ with the following regularity,  for every $\epsilon\in(0,1)$,
\[
v \in L^\infty(0,t_0;H^2(\Omega))
\cap C([0,t_0];H^1(\Omega)),
\quad  v \in L^2(0,t_0;H^{3-\epsilon}(\Omega)),
\]
\[
T \in L^\infty(0,t_0;H^2(\Omega))
\cap C([0,t_0];H^1(\Omega)),
\quad
\nabla_H T \in L^2(0,t_0;H^{2}(\Omega)),
\]
\[
\partial_t v \in L^2(0,t_0;H^1(\Omega)),
\qquad
\partial_t T \in L^2(0,t_0;H^1(\Omega));
\]
\end{theorem}
\begin{remark}
On nonsmooth domains, solutions of the Stokes system generally do not enjoy the same regularity as on smooth domains because of corner singularities, which prevent us from obtaining higher-order regularity such as $v\in H^3(\Omega)$ and $p\in H^2(\Omega)$, even when the forcing term is sufficiently smooth. Higher-order regularity depends on the corner and edge geometry and on the boundary conditions, and may require additional compatibility conditions on the data; see \cite{MazyaRossmann2007}. Since no compatibility assumptions of this type are imposed in the present work, we do not discuss $H^3$-regularity of the solution.
\end{remark}

Theorem~\ref{thm:global-well-posedness} provides a unique global strong
solution for the initial data in Theorem~\ref{thm:higher-regularity}.
We establish the additional regularity in the following propositions
and complete the proof of Theorem~\ref{thm:higher-regularity} at the end
of this section.
\begin{proposition}
\label{prop:time-derivative-L2}
Under the initial-data assumptions of this section, there exists a continuous nondecreasing function $M_4$, finite on compact subintervals
of $[0,\mathcal T^*)$, such that
\[
 \mathop{\mathrm{ess\,sup}}_{0<s<t}
 \left(\|v_t(s)\|_{L^2}^2+\|T_t(s)\|_{L^2}^2\right)
 +\int_0^t\left(\|\nabla v_t(s)\|_{L^2}^2
                    +\|\nabla_HT_t(s)\|_{L^2}^2\right)\,ds
 \leq M_4(t).
\]
\end{proposition}

\begin{proof}
Fix $0<t_0<\mathcal T^*$ and choose $\eta_0>0$ such that
$t_0+\eta_0<\mathcal T^*$. We use the existing strong solution
from Section~\ref{sec:construction-global}. For $0<\eta<\eta_0$, set
\[
 \delta_\eta v(t):=\frac{v(t+\eta)-v(t)}{\eta},
 \qquad
 \delta_\eta T(t):=\frac{T(t+\eta)-T(t)}{\eta}.
\]
We first bound the initial difference quotients.

Let $w_0=-\int_{-h}^z\nabla_H\cdot v_0(\xi)\,d\xi$.
The anisotropic estimates give,
\[
 \|w_0\partial_zv_0\|_2+\|w_0\partial_zT_0\|_2
 \leq C\|v_0\|_{H^2}
             \bigl(\|v_0\|_{H^2}+\|T_0\|_{H^2}\bigr).
\]
Hence $(F_0^v,F_0^T)$ defined following belong to $L^2(\Omega)$:
\[
\begin{aligned}
 &F_0^v:=Av_0+\mathbb P\left[
       (v_0\cdot\nabla_H)v_0+w_0\partial_zv_0+f_0k\times v_0
       -\int_{-h}^z\nabla_HT_0(\xi)\,d\xi\right],\\
 &F_0^T=-\Delta_HT_0+v_0\cdot\nabla_HT_0+w_0\partial_zT_0.
\end{aligned}
\]
Choose sufficiently large $C_0 \geq \max \{1,\norm{F_0^v}_{L^2},\norm{F_0^T}_{L^2},e^{G(0)|_{(v_0,T_0)}}\}$, which only depends on $\norm{v_0}_{H^2}$ and $\norm{T_0}_{H^2}$. $G(0)|_{(v_0,T_0)}$ is the value in \eqref{eq:coupled-stability-coefficient} evaluated at $(v_0,T_0)$.

Subtract the $F_0^v$ and $F_0^T$ from the velocity and temperature equation,
\begin{align*}
\partial_t(v-v_0)+A(v-v_0)
+\mathbb P\Bigg[
  &v\cdot\nabla_H(v-v_0)+w\partial_z(v-v_0)\\
  &+(v-v_0)\cdot\nabla_Hv_0
    +(w-w_0)\partial_zv_0\\
  &+f_0k\times(v-v_0)
    -\int_{-h}^{z}\nabla_H(T-T_0)(\xi)\,d\xi
\Bigg]
&=-F_0^v,\\[1mm]
\partial_t(T-T_0)-\Delta_H(T-T_0)
  &+v\cdot\nabla_H(T-T_0)+w\partial_z(T-T_0)\\
  &+(v-v_0)\cdot\nabla_HT_0
    +(w-w_0)\partial_zT_0
=- F_0^T.
\end{align*}
And repeat the energy calculation in the Uniqueness proof in Proposition~\ref{prop:strong-solution-stability}, we obtain,
\[
\begin{aligned}
 \frac{d}{dt}\bigl(\|v(t)-v_0\|_2^2+\|T(t)-T_0\|_2^2\bigr)
 &\leq
 C_0\bigl(\|v(t)-v_0\|_2^2+\|T(t)-T_0\|_2^2\bigr)\\
 &+2C_0
       \bigl(\|v(t)-v_0\|_2^2+\|T(t)-T_0\|_2^2\bigr)^{1/2}.
\end{aligned}
\]
Set
\[
Y(t):=\|v(t)-v_0\|_2^2+\|T(t)-T_0\|_2^2.
\]
Then we have the differential form $Y'(t)\leq C_0 Y+2 C_0 Y^{\frac{1}{2}}$.

Since $v,T\in H^1(0,t_0+\eta_0;L^2(\Omega))$,
\[
Y'(t)
=2\bigl(v(t)-v_0,v_t(t)\bigr)_{L^2}
 +2\bigl(T(t)-T_0,T_t(t)\bigr)_{L^2}
\quad\text{for a.e. }t.
\]
The right-hand side belongs to $L^1(0,t_0+\eta_0)$, so $Y$ is absolutely continuous. The initial traces give $Y(0)=0$.
Applying Lemma~\ref{lem:sqrt-gronwall} to the differential inequality $Y'(t)\leq C_0 Y+2 C_0 Y^{\frac{1}{2}}$, we obtain
\[
\sqrt{Y(t)}\leq C_0t\,e^{C_0t/2}.
\]
Consequently, for $0<\eta<\eta_0$,
\[
\begin{aligned}
\|\delta_\eta v(0)\|_2^2+\|\delta_\eta T(0)\|_2^2
=\frac{Y(\eta)}{\eta^2}\leq C_0^2e^{C_0\eta}.
\end{aligned}
\]
Now apply Uniqueness Proposition~\ref{prop:strong-solution-stability} to
$(v(t+\eta),T(t+\eta))$ and $(v(t),T(t))$ on $[0,t_0]$.
Let $G(t)$ be the expression in
\eqref{eq:coupled-stability-coefficient} with
$(v^{(2)},T^{(2)})=(v(t),T(t))$, which is independent of $\eta$ and belongs to $L^1(0,t_0)$ by theorem~\ref{thm:global-well-posedness}.
Dividing \eqref{eq:coupled-strong-stability} by $\eta^2$ and
controlling the supremum and the dissipation integral separately,
we obtain, for $0\leq t\leq t_0$,
\[
\begin{aligned}
 \sup_{0\leq s\leq t}
       \bigl(\|\delta_\eta v(s)\|_2^2
                         +\|\delta_\eta T(s)\|_2^2\bigr)&+\int_0^t
       \bigl(\|\nabla\delta_\eta v\|_2^2
                         +\|\nabla_H\delta_\eta T\|_2^2\bigr)\,ds\\
 &\leq 2C_0^2\exp\!\left(C_0\eta+\int_0^tG(s)\,ds\right).
\end{aligned}
\]
Since $v,T\in H^1(0,t_0+\eta_0;L^2)$, their time
difference quotients are the Steklov averages of $v_t,T_t$.
By the strong convergence of Steklov averages
(see \cite[Chapter~I, Lemma~3.2]{DiBenedetto1993}
and \cite[Lemma~2.5]{AmbrosioCiani2026}), we obtain
\[
(\delta_\eta v,\delta_\eta T)\longrightarrow(v_t,T_t)
\quad\text{strongly in }L^2(0,t_0;L^2)
\quad\text{as }\eta\downarrow0.
\]
The uniform bounds above identify the weak limits of the spatial
gradients with $\nabla v_t$ and $\nabla_HT_t$, and the weak-star
limits in $L^\infty(0,t_0;L^2)$ with $v_t,T_t$.
Lower semicontinuity therefore gives
\[
\begin{aligned}
 \mathop{\mathrm{ess\,sup}}_{0<s<t}
       \bigl(\|v_t(s)\|_2^2+\|T_t(s)\|_2^2\bigr) +\int_0^t
       \bigl(\|\nabla v_t\|_2^2+\|\nabla_HT_t\|_2^2\bigr)\,ds
 \leq2C_0^2\exp\!\left(\int_0^tG(s)\,ds\right).
\end{aligned}
\]
The right-hand side defines a continuous nondecreasing $M_4$ on
$[0,\mathcal T^*)$.
\end{proof}

\begin{proposition}\label{prop:velocity-H2}
For every $0<t_0<\mathcal T^*$, there is a continuous nondecreasing
function $M_5$ on $[0,t_0]$ such that
\[
 \sup_{0\leq s\leq t}\|v(s)\|_{H^2}^2\leq M_5(t),
 \qquad 0\leq t\leq t_0.
\]
\end{proposition}
\begin{proof}
Recall $v\in L^2(0,t_0;D(A))$ and $v_t\in L^2(0,t_0;V)$, then
\[
 \frac12\frac{d}{dt}\|A^{3/4}v\|_2^2
   =(A^{1/2}v_t,Av)_{L^2}.
\]
It gives,
\[
 \sup_{0\leq s\leq t}\|A^{3/4}v(s)\|_2^2
 \leq\|A^{3/4}v_0\|_2^2
       +2\|A^{1/2}v_t\|_{L^2(0,t;L^2)}\|Av\|_{L^2(0,t;L^2)}.
\]
The initial term is finite because $v_0\in D(A)$. By interpolation, and
Lemma~\ref{lemma:Gagliardo}\textup{(ii)},
\[
 \begin{aligned}
 D(A^{3/4})
   &=[D(A^{1/2}),D(A)]_{1/2}\\
   &\hookrightarrow[H^1(\Omega)^2,H^2(\Omega)^2]_{1/2}
     =H^{3/2}(\Omega)^2.
 \end{aligned}
\]
Thus $v\in L^\infty(0,t_0;H^{3/2}(\Omega))$.
Also,
\[
 \|w\|_{L^\infty_zH^{1/2}(M)}
       +\|\partial_zv\|_{L^2_zH^{1/2}(M)}
 \leq C\|v\|_{H^{3/2}(\Omega)}.
\]
Using $H^{1/2}(M)\hookrightarrow L^4(M)$,
$H^1(\Omega)\hookrightarrow L^6(\Omega)$,
$H^{1/2}(\Omega)\hookrightarrow L^3(\Omega)$, we obtain
\[
 \begin{aligned}
 \|(v\cdot\nabla_H)v+w\partial_zv\|_2
 &\leq \|v\|_6\|\nabla_Hv\|_3
       +\|w\|_{L^\infty_zL^4(M)}\|\partial_zv\|_{L^2_zL^4(M)}\\
 &\leq C\|v\|_{H^{3/2}}^2.
 \end{aligned}
\]
Apply Lemma~\ref{lem:top-stokes-H2} in the constant-depth case to
the momentum equation. Vertical integration of the thermal forcing
then yields, for almost every time,
\[
 \|v\|_{H^2}
 \leq C\bigl(\|v_t\|_2+\|v\|_{H^{3/2}}^2
                         +\|v\|_2+\|\nabla_HT\|_2\bigr).
\]
Every term on the right is uniformly bounded by the preceding
estimates. The known $H^1$ continuity of $v$ and weak lower
semicontinuity extend this bound to every time.
\end{proof}

\begin{proposition}[Higher spatial regularity of the velocity]
\label{prop:velocity-higher-spatial}
Let \(0<t<\mathcal T^*\), then for every \(\varepsilon\in(0,1)\),
\[
v\in L^2(0,t;H^{3-\varepsilon}(\Omega)).
\]
More precisely, there exists a continuous nondecreasing function
\(M_{v,\varepsilon}\) such that
\[
\int_0^t
\left(\|v\|_{H^{3-\varepsilon}}^2
\right)\,ds
\leq M_{v,\varepsilon}(t)
\]
\end{proposition}

\begin{proof}
\noindent
The Sobolev embeddings and the product rule give
\begin{align}
\|(v\cdot\nabla_H)v\|_{H^1}
\leq
C\left(
 \|v\|_\infty\|\nabla_Hv\|_{H^1}
 +\|\nabla v\|_4\|\nabla_Hv\|_4
\right)
\nonumber\leq C\|v\|_{H^2}^2.
\label{eq:horizontal-convection-H1}
\end{align}
Consequently,
\[
(v\cdot\nabla_H)v
\in L^\infty(0,t_0;H^1(\Omega))
\subset L^2(0,t_0;H^1(\Omega)).
\]

Since $u=\partial_zv$ and $w=\int_{-h}^z-\nabla_H\cdot v$, we also have
\[
\|u\|_{H^1}\leq C\|v\|_{H^2}, \qquad \|w\|_{H^1}
\leq C\|v\|_{H^2}.
\]
By the product rule and the embeddings,
\begin{equation}
\begin{aligned}
\|w\partial_zv\|_{H^{2/5}}
&\leq C\|wu\|_{W^{1,3/2}}\\
&\leq C\bigl(\|w\|_6\|u\|_2
 +\|\nabla w\|_2\|u\|_6
 +\|w\|_6\|\nabla u\|_2\bigr)\\
&\leq C\|w\|_{H^1}\|u\|_{H^1}
 \leq C\|v\|_{H^2}^2.
\end{aligned}
\label{eq:vertical-convection-H25}
\end{equation}
By the definition of $Q_T$ in \eqref{eq: Q}, $\|\nabla_H Q_T\|_{H^1}
\leq C\|\nabla_HT\|_{H^1}$,
Hence
\begin{align}
\nabla_H Q_T\in L^2(0,t_0;H^1(\Omega)).
\label{eq:Q-H1-estimate}
\end{align}

Rewrite the momentum equation as
\begin{equation}
-\Delta v+\nabla_Hp_s=F,
\label{eq:velocity-Stokes-equation}
\end{equation}
where
\[
F:=
-v_t-(v\cdot\nabla_H)v-w\partial_zv
-f_0k\times v+\nabla_H Q_T .
\]
By Proposition~\ref{prop:time-derivative-L2},
\eqref{eq:horizontal-convection-H1},
\eqref{eq:vertical-convection-H25},
\eqref{eq:Q-H1-estimate}, and Proposition~\ref{prop:velocity-H2},
\begin{equation}
F\in L^2(0,t_0;H^{2/5}(\Omega)).
\label{eq:F-H25}
\end{equation}
More precisely,
\begin{equation}
\|F\|_{H^{2/5}}
\leq
C\left(
 \|v_t\|_{H^1}
 +\|\nabla_HT\|_{H^1}
 +\|v\|_{H^2}
 +\|v\|_{H^2}^2
\right).
\label{eq:F-H25-bound}
\end{equation}

Fix \(0<\delta_0<1/6\), and let \(\mathcal S\) denote the velocity
component of the hydrostatic Stokes solution operator with the
prescribed boundary conditions. Lemmas~\ref{lem:top-stokes-H2}
and~\ref{lem:top-stokes-higher} give
\[
\mathcal S:L^2(\Omega)\longrightarrow H^2(\Omega),
\qquad
\mathcal S:H^1(\Omega)\longrightarrow H^{3-\delta_0}(\Omega).
\]
Complex interpolation yields
\[
\mathcal S:
\bigl[L^2(\Omega),H^1(\Omega)\bigr]_{2/5}
\longrightarrow
\bigl[H^2(\Omega),H^{3-\delta_0}(\Omega)\bigr]_{2/5}.
\]
By Lemma~\ref{lemma:Gagliardo}\textup{(ii)},
\[
\begin{aligned}
\bigl[L^2(\Omega),H^1(\Omega)\bigr]_{2/5}
&=H^{2/5}(\Omega),\\
\bigl[H^2(\Omega),H^{3-\delta_0}(\Omega)\bigr]_{2/5}
&=H^{s_*}(\Omega)\hookrightarrow H^{7/3}(\Omega),
\end{aligned}
\]
where \(s_*:=2+\frac25(1-\delta_0)>\frac73\). Therefore,
\[
\|v\|_{H^{7/3}}
\leq C\|v\|_{H^{s_*}}
\leq C\|F\|_{H^{2/5}}.
\]
Using \eqref{eq:F-H25}, we conclude that
\begin{equation}
v\in L^2(0,t_0;H^{7/3}(\Omega)).
\label{eq:v-H73}
\end{equation}
Using \(w=-\int_{-h}^z\nabla_H\cdot v(\xi)\,d\xi\),
\(\partial_zw=-\nabla_H\cdot v\), and the embedding
\(H^{4/3}(M)\hookrightarrow L^\infty(M)\cap W^{1,3}(M)\), we obtain
\begin{equation}
\begin{aligned}
\|w\|_\infty+\|\nabla w\|_3
&\leq C\Bigl(
 \|\nabla_H\cdot v\|_{L^2_zH^{4/3}(M)}
 +\|\nabla_H\cdot v\|_3\Bigr)\\
&\leq C\|v\|_{H^{7/3}}.
\end{aligned}
\label{eq:w-anisotropic-H73}
\end{equation}
Using the product rule,
\begin{align}
\|w\partial_zv\|_{H^1}
&\leq
C\left(
 \|w\|_\infty\|u\|_{H^1}
 +\|\nabla w\|_3\|u\|_6
\right)
\nonumber\\
&\leq
C\|v\|_{H^{7/3}}\|v\|_{H^2}.
\label{eq:vertical-convection-final-H1}
\end{align}
Then,
\begin{equation}
w\partial_zv\in L^2(0,t_0;H^1(\Omega)).
\label{eq:vertical-convection-L2H1}
\end{equation}
Returning to the definition of \(F\), we now have
\[
F\in L^2(0,t_0;H^1(\Omega)).
\]
with
\begin{align}
\|F\|_{H^1}
\leq C\Big(
 &\|v_t\|_{H^1}
 +\|\nabla_HT\|_{H^1}
 +\|v\|_{H^2}
 \nonumber+\|v\|_{H^2}\|v\|_{H^{7/3}}
\Big),
\label{eq:F-final-H1}
\end{align}
and every term on the right-hand side is $L^2$ in time

For every fixed \(\varepsilon\in(0,1)\), the higher hydrostatic
Stokes estimate in Lemma~\ref{lem:top-stokes-higher}, applied to
\eqref{eq:velocity-Stokes-equation} yields, for almost every time,
\begin{equation}
\begin{split}
\|v\|_{H^{3-\varepsilon}}
+\|\partial_zv\|_{H^{2-\varepsilon}}
+\|\nabla_Hp_s\|_{H^{1-\varepsilon}(M)}
\leq C_\varepsilon\|F\|_{H^1}.
\end{split}
\label{eq:final-higher-Stokes}
\end{equation}
Squaring \eqref{eq:final-higher-Stokes} and integrating in time gives
\[
v\in L^2(0,t_0;H^{3-\varepsilon}(\Omega)).
\]
\end{proof}
\begin{proposition}
\label{lem:flat-vertical-temperature}
\label{prop:Tzz-estimate}
Let $v_0\in D(A)$ and $T_0\in H^2(\Omega)\cap V_T$, and let $(v,T)$ be the unique global strong solution given by Theorem~\ref{thm:global-well-posedness}. Recall, for every $0<t_0<\mathcal T^*$, we have 
\[
\begin{aligned}
T
&\in L^\infty(0,t_0;V_T),
&
\nabla_HT
&\in L^2(0,t_0;H^1(\Omega)),
\\
T_t
&\in L^\infty(0,t_0;L^2(\Omega)),
&
\nabla_HT_t
&\in L^2(0,t_0;L^2(\Omega)).
\end{aligned}
\]
Then the same solution has the additional regularity
\[
\partial_z^2T\in L^\infty(0,t_0;L^2(\Omega)),
\qquad
\nabla_H\partial_z^2T
\in L^2(0,t_0;L^2(\Omega)).
\]
More precisely, there exists a continuous nondecreasing
function $M_7$ on $[0,\mathcal T^*)$, finite on subintervals, such that
\begin{equation}
\sup_{0\leq s\leq t}
\|\partial_z^2T(s)\|_{L^2(\Omega)}^2
+
\int_0^t
\|\nabla_H\partial_z^2T(s)\|_{L^2(\Omega)}^2\,ds
\leq M_7(t)
\label{eq:Tzz-bound}
\end{equation}
for every $0\leq t<\mathcal T^*$.
\end{proposition}

\begin{proof}
Fix $0<t_0<\mathcal T^*$. We use the temperature Galerkin spaces and projection $P_N^T$ from Lemma~\ref{lem:linear-temperature},
and write $\mathcal E_N:=\operatorname{ran}P_N^T$. Recall,
\[
 \partial_z w=-\nabla_H\cdot v,\qquad
 \partial_{zz}w=-\nabla_H\cdot\partial_zv,\qquad
 w|_{z=-h,0}=0.
\]
By Sobolev embedding and trace theorems, we have $H^{3/4}(\Omega)\hookrightarrow L^\infty_zH^{1/4}(M)
\hookrightarrow L^\infty_zL^{8/3}(M)$, which implies,
\begin{equation}
 \begin{split}
 &\|v\|_\infty+\|w\|_\infty+\|\partial_zv\|_\infty
       +\|\partial_zw\|_\infty+\|\partial_{zz}v\|_3\\
 &\qquad+\|\partial_{zz}w\|_{L^\infty_zL^{8/3}_H}
       \leq C\|v\|_{H^{11/4}}.
 \end{split}
 \label{eq:flat-vertical-coefficients}
\end{equation}
Differentiating the temperature equation \eqref{eq:T equation} in $z$ yields
\[
 (\partial_zT)_t-\Delta_H\partial_zT
 +v\cdot\nabla_H\partial_zT+\partial_z(w\partial_zT)
 =-\partial_zv\cdot\nabla_HT
\]
in the sense of distributions. The existing $\partial_zT$ belongs to
$L^\infty_tL^2\cap L^2_tL^2_zH_0^1(M)$; Its initial trace is $\partial_zT_0$ weakly in $L^2$, since $T\in C([0,t_0];L^2)$ and is uniformly bounded in $H^1$. Recall the temperature projection defined in lemma~\ref{lem:linear-temperature}, with the given $v,w,T$ fixed, solve the finite-dimensional linear system
\begin{equation}
 \begin{split}
 &\bigl(\partial_t(\partial_zT)_N,\psi\bigr)
   +\bigl(\nabla_H(\partial_zT)_N,\nabla_H\psi\bigr)\\
 &\quad+\bigl(v\cdot\nabla_H(\partial_zT)_N
                  +\partial_z(w(\partial_zT)_N),\psi\bigr)
   =-\bigl(\partial_zv\cdot\nabla_HT,\psi\bigr),\\
 &(\partial_zT)_N(0)=P_N^T(\partial_zT_0),
 \end{split}
 \label{eq:flat-vertical-regularized}
\end{equation}
for every $\psi\in\mathcal E_N$. Parentheses denote the
$L^2(\Omega)$ inner product. For each $N$, the matrix coefficients
are integrable in time and the forcing belongs to $L^2$ in time;
the linear system therefore has a unique absolutely continuous
solution on $[0,t_0]$. $(\partial_zT)_N$ denotes the solution of this auxiliary system.
\[
 \begin{gathered}
 P_N^T(\partial_zT_0)\longrightarrow\partial_zT_0
                         \quad\hbox{in }H^1(\Omega),\\
 \|\partial_zP_N^T(\partial_zT_0)\|_2
                              \leq\|\partial_{zz}T_0\|_2,\qquad
 \partial_zP_N^T(\partial_zT_0)\longrightarrow\partial_{zz}T_0
                         \quad\hbox{in }L^2(\Omega).
 \end{gathered}
\]

Mutiply the Galerkin system with $(\partial_zT)_N$ and $-\partial_{zz}(\partial_zT)_N$. Moreover,
$\partial_z(\partial_zT)_N=0$ at $z=-h,0$. Integration by parts gives
\[
 \begin{aligned}
 &\frac12\frac{d}{dt}\|(\partial_zT)_N\|_2^2
       +\|\nabla_H(\partial_zT)_N\|_2^2\\
 &\quad=-\bigl(\partial_zv\cdot\nabla_HT,(\partial_zT)_N\bigr)
       -\int_\Omega\partial_zw\,|(\partial_zT)_N|^2
 \end{aligned}
\]
and
\[
 \begin{aligned}
 &\frac12\frac{d}{dt}\|\partial_z(\partial_zT)_N\|_2^2
       +\|\nabla_H\partial_z(\partial_zT)_N\|_2^2\\
 &\quad=-\bigl(\partial_z(\partial_zv\cdot\nabla_HT),
                                      \partial_z(\partial_zT)_N\bigr)\\
 &\qquad-\bigl(\partial_zv\cdot\nabla_H(\partial_zT)_N,
                                      \partial_z(\partial_zT)_N\bigr)\\
 &\qquad-2\int_\Omega\partial_zw\,
                                      |\partial_z(\partial_zT)_N|^2\\
 &\qquad-\bigl(\partial_{zz}w\,(\partial_zT)_N,
                                      \partial_z(\partial_zT)_N\bigr).
 \end{aligned}
\]
By \eqref{eq:flat-vertical-coefficients}, $H^1(\Omega)\hookrightarrow
L^6(\Omega)$, and $H^1(M)\hookrightarrow L^8(M)$,
\[
 \begin{aligned}
 &\|\partial_zv\cdot\nabla_HT\|_2
       +\|\partial_z(\partial_zv\cdot\nabla_HT)\|_2
       \leq C\|v\|_{H^{11/4}}\|\nabla_HT\|_{H^1},\\
 &\|\partial_{zz}w\,(\partial_zT)_N\|_2
       \leq C\|v\|_{H^{11/4}}
          \bigl(\|(\partial_zT)_N\|_2
                    +\|\nabla_H(\partial_zT)_N\|_2\bigr).
 \end{aligned}
\]
Adding the two energy estimates and absorbing the terms containing $\|v\|_{H^{11/4}}\|\nabla_H(\partial_zT)_N\|_2 \|\partial_z(\partial_zT)_N\|_2$ yields
\[
 \begin{aligned}
 &\frac{d}{dt}\bigl(1+\|(\partial_zT)_N\|_2^2
                          +\|\partial_z(\partial_zT)_N\|_2^2\bigr)\\
 &\quad+\|\nabla_H(\partial_zT)_N\|_2^2
             +\|\nabla_H\partial_z(\partial_zT)_N\|_2^2\\
 &\leq C\bigl(1+\|v\|_{H^{11/4}}^2
                       +\|v\|_{H^{11/4}}\|\nabla_HT\|_{H^1}\bigr)\\
 &\qquad\cdot\bigl(1+\|(\partial_zT)_N\|_2^2
                          +\|\partial_z(\partial_zT)_N\|_2^2\bigr).
 \end{aligned}
\]
The coefficient is integrable by Section~\ref{sec:construction-global}
and Proposition~\ref{prop:velocity-higher-spatial} with
$\varepsilon=1/4$, and the initial energy is bounded by
$1+\|\partial_zT_0\|_2^2+\|\partial_{zz}T_0\|_2^2$.
Gronwall's inequality therefore gives uniform bounds for
$(\partial_zT)_N$ and $\partial_z(\partial_zT)_N$ in
$L^\infty_tL^2$, and for their horizontal gradients in $L^2_tL^2$.
Equation~\eqref{eq:flat-vertical-regularized}, written in divergence
form, also bounds $\partial_t(\partial_zT)_N$ in $L^2(0,t_0;V_T')$.
Here $V_T'$ is the dual of the scalar temperature space $V_T$,
with pivot $L^2(\Omega)$.
Indeed, the projection is bounded on that space,
$v,w,\partial_zw\in L^2_tL^\infty$, and
$\partial_zv\cdot\nabla_HT\in L^2_tL^2$ by
$T\in L^\infty_tH^1$ and \eqref{eq:flat-vertical-coefficients}.
Aubin--Lions and weak compactness give, along a subsequence,
\[
 (\partial_zT)_N\longrightarrow(\partial_zT)_*
       \quad\hbox{strongly in }L^2_tL^2
       \quad\hbox{and in }C([0,t_0];L^2\hbox{-weak}).
\]
The derivative bounds pass weakly to this limit. It remains to identify this limit with $\partial_zT$. Let
\[
d=(\partial_zT)_*-\partial_zT,
\qquad
P(x,y,z,t)=\int_{-h}^{z}d(x,y,\xi,t)\,d\xi.
\]
Subtracting the two equations gives,
\[
d_t-\Delta_Hd+v\cdot\nabla_Hd+\partial_z(wd)=0,
\]
with zero initial data in the weak $L^2$ sense. The preceding
bounds imply
\[
P\in L^\infty(0,t_0;L^2(\Omega))
       \cap L^2(0,t_0;H^1(\Omega)),
\qquad
\partial_zP=d\in L^\infty(0,t_0;L^2(\Omega)),
\]
and $P$ has zero lateral boundary condition. Let $\chi_a$ be a vertical cutoff of width
$a$, with $|\chi_a'|\leq C/a$. Since
\[
|w(z)|\leq
\operatorname{dist}(z,\{-h,0\})\|\partial_zw\|_\infty,
\]
for any smooth test function $\varphi$,
\[
\left|
\int_0^{t_0}\int_\Omega wd\,\chi_a'\varphi
\right|
\leq C_\varphi a^{1/2}
\|\partial_zw\|_{L^1_tL^\infty}
\|d\|_{L^\infty_tL^2}
\longrightarrow0.
\]
Testing with the backward vertical primitive of a smooth test
function and applying Fubini's theorem therefore yields
\[
P_t-\Delta_HP+v\cdot\nabla_HP+w\partial_zP=\mathcal K_v P,
\qquad
\mathcal K_v P(z)
:=\int_{-h}^{z}
\partial_zv(\xi)\cdot\nabla_HP(\xi)\,d\xi.
\]

Horizontal integration by parts, the embedding $H_0^1(M)\hookrightarrow L^8(M)$, and the identity $\nabla_H\cdot\partial_zv=-\partial_{zz}w$ give
\[
\begin{aligned}
\|\mathcal K_v P\|_{\mathcal V_H'}
&\leq C\left(
\|\partial_zv\|_\infty
+\|\partial_{zz}w\|_{L^\infty_zL^{8/3}(M)}
\right)\|P\|_2\\
&\leq C\|v\|_{H^{11/4}}\|P\|_2.
\end{aligned}
\]
Consequently, $\mathcal K_v P\in L^2(0,t_0;\mathcal V_H')$. The equation and the preceding bounds then imply
\[
P_t\in L^2(0,t_0;\mathcal V_H').
\]
Since $P\in L^2(0,t_0;\mathcal V_H)$, it follows that $P\in C([0,t_0];L^2)$ and that its $L^2$ energy identity is valid. The weak initial condition gives $P(0)=0$.

Testing the equation with $P$, the transport terms cancel by $\nabla_H\cdot v+\partial_zw=0$ and the boundary conditions.
Moreover,
\[
\|\mathcal K_v P\|_2
\leq C_h\|\partial_zv\|_\infty\|\nabla_HP\|_2.
\]
Hence
\[
\begin{aligned}
\frac12\frac{d}{dt}\|P\|_2^2+\|\nabla_HP\|_2^2
&=(\mathcal K_v P,P)\\
&\leq
C_h\|\partial_zv\|_\infty\|\nabla_HP\|_2\|P\|_2\\
&\leq
\frac12\|\nabla_HP\|_2^2
+C\|\partial_zv\|_\infty^2\|P\|_2^2.
\end{aligned}
\]
Since $\partial_zv\in L^2(0,t_0;L^\infty)$ and $P(0)=0$, Gronwall's inequality gives $P=0$. Thus
\[
d=\partial_zP=0,
\qquad
(\partial_zT)_*=\partial_zT.
\]
The uniform derivative bounds therefore pass to the original solution, yielding
\[
\partial_{zz}T\in L^\infty(0,t_0;L^2),
\qquad
\nabla_H\partial_{zz}T\in L^2(0,t_0;L^2).
\]
The identified limit belongs to $L^2_tH^1$ and has time derivative
in $L^2$ in time with values in the dual of the lateral-zero $H^1$
space. Its $L^2$ energy identity is therefore valid. Compare it
with the first Galerkin energy identity. The forcing terms converge
by strong $L^2_tL^2$ convergence. The terms containing
$\partial_zw\,|(\partial_zT)_N|^2$ converge because
$\partial_zw\in L^2_tL^\infty$ and the approximations are uniformly
bounded in $L^\infty_tL^2$. Weak convergence at the terminal time
and lower semicontinuity then give
\[
 \nabla_H(\partial_zT)_N\longrightarrow\nabla_H\partial_zT
                    \quad\hbox{strongly in }L^2_tL^2.
\]
Returning to the second Galerkin energy identity, without absorbing the horizontal gradient of $(\partial_zT)_N$, gives
\begin{equation}
 \begin{split}
 &\frac{d}{dt}\bigl(1+\|\partial_z(\partial_zT)_N\|_2^2\bigr)
       +2\|\nabla_H\partial_z(\partial_zT)_N\|_2^2\\
 &\quad\leq C\|v\|_{H^{11/4}}
       \bigl(1+\|\nabla_HT\|_{H^1}+\|(\partial_zT)_N\|_2
                         +\|\nabla_H(\partial_zT)_N\|_2\bigr)\\
 &\hspace{24mm}\cdot
       \bigl(1+\|\partial_z(\partial_zT)_N\|_2^2\bigr).
 \end{split}
 \label{eq:flat-vertical-energy}
\end{equation}
By the two strong convergences above and Cauchy--Schwarz in time,
the coefficient on the right converges in $L^1(0,t_0)$ to
\[
 C\|v\|_{H^{11/4}}
       \bigl(1+\|\nabla_HT\|_{H^1}+\|\partial_zT\|_2
                                 +\|\nabla_H\partial_zT\|_2\bigr).
\]
This is bounded by
\[
 C\|v\|_{H^{11/4}}
       \bigl(1+\|T\|_{L^\infty(0,t_0;H^1)}
                                 +\|\nabla_HT\|_{H^1}\bigr).
\]
Also,
$\partial_zP_N^T(\partial_zT_0)\to\partial_{zz}T_0$ in $L^2$.
Applying Gronwall's inequality to \eqref{eq:flat-vertical-energy}, integrate, and passing to the limit proves \eqref{eq:flat-vertical-temperature-bound}. 
\begin{equation}
 \begin{split}
 &\sup_{0\leq s\leq t}
       \|\partial_{zz}T(s)\|_2^2
       +\int_0^t\|\nabla_H\partial_{zz}T\|_2^2\,ds\\
 &\qquad\leq C\bigl(1+\|\partial_{zz}T_0\|_2^2\bigr)
        \exp\!\Bigg(C\int_0^t\|v(s)\|_{H^{11/4}}\\
 &\hspace{16mm}\cdot\bigl(1+\|T\|_{L^\infty(0,t_0;H^1)}
             +\|\nabla_HT(s)\|_{H^1}\bigr)\,ds\Bigg),
 \end{split}
 \label{eq:flat-vertical-temperature-bound}
\end{equation}
where $C$ depends only on $M$ and $h$.
\end{proof}

\begin{proposition}
\label{prop:horizontal-temperature-L2}
There exists a continuous nondecreasing function $M_6$ on
$[0,\mathcal T^*)$ such that
\[
\|\Delta_HT\|_{L^\infty(0,t;L^2(\Omega))}\leq M_6(t)
\qquad\text{for every }0<t<\mathcal T^*.
\]
\end{proposition}
\begin{proof}
Fix $0<t<\mathcal T^*$. Theorem \ref{thm:global-well-posedness}
and Propositions~\ref{prop:time-derivative-L2},
\ref{prop:velocity-H2}, and~\ref{prop:Tzz-estimate} give
\[
\begin{gathered}
T\in L^\infty(0,t;H^1),\qquad
T_t\in L^\infty(0,t;L^2),\\
v\in L^\infty(0,t;H^2),\qquad
\partial_z^2T\in L^\infty(0,t;L^2).
\end{gathered}
\]
Moreover, $\Delta_HT\in L^2(0,t;L^2)$ is already known. Hence horizontal Dirichlet elliptic regularity and Lemma~\ref{lem:anisotropic-isotropic} with $s=2$ imply, for almost every time,
\[
\begin{aligned}
\|\nabla_H\partial_zT\|_2
&\leq \|T\|_{H^2}\\
&\leq C\bigl(
\|\Delta_HT\|_2+\|T\|_{H^1}+\|\partial_z^2T\|_2
\bigr).
\end{aligned}
\]

The anisotropic Lemma~\ref{anisotropic inequality} gives
\[
\begin{aligned}
\|w\partial_zT\|_2^2
&\leq C\|v\|_{H^1}\|v\|_{H^2}
\|\partial_zT\|_2
\bigl(\|\partial_zT\|_2+\|\nabla_H\partial_zT\|_2\bigr)\\
&\leq C\bigl(1+\|\Delta_HT\|_2\bigr),
\end{aligned}
\]
where $C$ after the second inequality depends only on the preceding bounds and the domain, which is independent of the time. Therefore, the temperature equation, Sobolev embedding, and Young's inequality give
\[
\begin{aligned}
\|\Delta_HT\|_2
&\leq \|T_t\|_2
+\|v\|_\infty\|\nabla_HT\|_2
+\|w\partial_zT\|_2\\
&\leq C\bigl(1+\|\Delta_HT\|_2^{1/2}\bigr)\\
&\leq \frac12\|\Delta_HT\|_2+C,
\end{aligned}
\]
for almost every time in $(0,t)$. Absorbing the first term on the right and taking the essential
supremum yields
\[
\|\Delta_HT\|_{L^\infty(0,t;L^2)}\leq 2C.
\]
The preceding estimates admit continuous nondecreasing majorants,
and the constants above depend only on these majorants and the
fixed domain. They therefore provide the required function
$M_6(t)$.
\end{proof}

\begin{proposition}[Time regularity of the vertical temperature derivative]
\label{prop:vertical-temperature-time}
Let $(v,T)$ be the strong solution with the initial data assumed in
this section. For every $0<t_0<\mathcal T^*$,
\[
\partial_t\partial_zT\in L^2(0,t_0;L^2(\Omega)),
\qquad
\nabla_H\partial_zT\in L^\infty(0,t_0;L^2(\Omega)).
\]
More precisely, there exists a continuous nondecreasing function
$M_{T,z}$ on $[0,\mathcal T^*)$ such that
\[
\sup_{0\leq s\leq t}\|\nabla_H\partial_zT(s)\|_{L^2}^2
+\int_0^t\|\partial_t\partial_zT(s)\|_{L^2}^2\,ds
\leq M_{T,z}(t)
\]
for every $0\leq t<\mathcal T^*$.
\end{proposition}

\begin{proof}
Fix $0<t_0<\mathcal T^*$. Proposition~\ref{prop:velocity-higher-spatial},
with $\varepsilon=1/4$, and the embedding
$H^{7/4}(\Omega)\hookrightarrow L^\infty(\Omega)$ give
\[
\begin{aligned}
 &\|\partial_zv\|_{L^2(0,t_0;L^\infty)}
  +\|\nabla_H\cdot v\|_{L^2(0,t_0;L^\infty)}
 \leq C\|v\|_{L^2(0,t_0;H^{11/4})}<\infty.
\end{aligned}
\]
Also, \eqref{eq:w-anisotropic-H73} implies
\[
 \|w\|_{L^2(0,t_0;L^\infty)}
 \leq C\|v\|_{L^2(0,t_0;H^{7/3})}<\infty.
\]
Section~\ref{sec:construction-global} and
Propositions~\ref{prop:velocity-H2} and~\ref{prop:Tzz-estimate}
provide
\[
\begin{gathered}
 v\in L^\infty(0,t_0;L^\infty),\qquad
 \partial_zT,\ \nabla_HT\in L^\infty(0,t_0;L^2),\\
 \nabla_H\partial_zT\in L^2(0,t_0;L^2),\qquad
 \partial_z^2T\in L^\infty(0,t_0;L^2).
\end{gathered}
\]

Differentiating the temperature equation in $z$ and using
$\partial_zw=-\nabla_H\cdot v$, we obtain
\[
\begin{aligned}
 \partial_t\partial_zT-\Delta_H\partial_zT
 ={}&-\partial_zv\cdot\nabla_HT
     -v\cdot\nabla_H\partial_zT\\
    &+(\nabla_H\cdot v)\partial_zT
     -w\partial_z^2T.
\end{aligned}
\]
By H\"older's inequality,
\begin{align}
 \|\partial_zv\cdot\nabla_HT\|_{L^2}
 &\leq
 \|\partial_zv\|_{L^\infty}\|\nabla_HT\|_{L^2},
 \label{eq:Tz-1}\\
 \|v\cdot\nabla_H\partial_zT\|_{L^2}
 &\leq
 \|v\|_{L^\infty}\|\nabla_H\partial_zT\|_{L^2},
 \label{eq:Tz-2}\\
 \|(\nabla_H\cdot v)\partial_zT\|_{L^2}
 &\leq
 \|\nabla_H\cdot v\|_{L^\infty}\|\partial_zT\|_{L^2},
 \label{eq:Tz-3}\\
 \|w\partial_z^2T\|_{L^2}
 &\leq
 \|w\|_{L^\infty}\|\partial_z^2T\|_{L^2}.
 \label{eq:Tz-4}
\end{align}
Thus all four terms belong to $L^2(0,t_0;L^2(\Omega))$.

Multiply the differentiated equation with $\partial_t\partial_zT$, integrate it over $\Omega$, and applying Young's inequality give,
\begin{align}
 &\frac{d}{dt}\|\nabla_H\partial_zT\|_{L^2}^2
   +\|\partial_t\partial_zT\|_{L^2}^2
 \nonumber\\
 &\quad\leq C\Bigl(
   \|\partial_zv\|_{L^\infty}^2\|\nabla_HT\|_{L^2}^2
   +\|v\|_{L^\infty}^2
      \|\nabla_H\partial_zT\|_{L^2}^2
 \nonumber\\
 &\hspace{27mm}
   +\|\nabla_H\cdot v\|_{L^\infty}^2
      \|\partial_zT\|_{L^2}^2
   +\|w\|_{L^\infty}^2
      \|\partial_z^2T\|_{L^2}^2
   \Bigr).
 \label{eq:Tz-energy}
\end{align}
Integrating and we obtain,
\begin{align*}
 &\sup_{0\leq s\leq t}
      \|\nabla_H\partial_zT(s)\|_{L^2}^2
   +\int_0^t
      \|\partial_t\partial_zT(s)\|_{L^2}^2\,ds\\
 &\quad\leq 2\|\nabla_H\partial_zT_0\|_{L^2}^2+C\int_0^t\Bigl(
   \|\partial_zv\|_{L^\infty}^2\|\nabla_HT\|_{L^2}^2
   +\|v\|_{L^\infty}^2
      \|\nabla_H\partial_zT\|_{L^2}^2\\
 &\hspace{27mm}
   +\|\nabla_H\cdot v\|_{L^\infty}^2
      \|\partial_zT\|_{L^2}^2
   +\|w\|_{L^\infty}^2
      \|\partial_z^2T\|_{L^2}^2
   \Bigr)\,ds.
\end{align*}
Since $t_0<\mathcal T^*$ is arbitrary, the right-hand side defines a continuous nondecreasing function $M_{T,z}$ on $[0,\mathcal T^*)$, as required.
\end{proof}
\begin{remark}
\label{rem:vertical-temperature-time-justification}
The preceding calculation is justified using the auxiliary Galerkin solutions $(\partial_zT)_N$ from Proposition~\ref{prop:Tzz-estimate}. Testing their equations with $\partial_t(\partial_zT)_N$, the uniform bounds established there and the coefficient bounds above give a uniform $L^2(0,t_0;L^2)$ bound for these time derivatives. Weak compactness and the previously identified limit yield $\partial_t\partial_zT\in L^2(0,t_0;L^2)$; the differentiated equation then gives $\Delta_H\partial_zT\in L^2(0,t_0;L^2)$. The time-trace theorem for the horizontal Dirichlet Laplacian therefore yields $\partial_zT\in C([0,t_0];\mathcal V_H)$ with initial value $\partial_zT_0$, and the corresponding energy identity justifies \eqref{eq:Tz-energy} and the resulting estimate at every time.
\end{remark}

\begin{proposition}
\label{prop:horizontal-T-higher-reg}
Let $(v,T)$ be the strong solution of Theorem~\ref{thm:global-well-posedness}, with $v_0\in D(A)$ and $T_0\in H^2(\Omega)\cap V_T$. For every $0<t<\mathcal T^*$,
\begin{equation}
 \nabla_HT\in L^2(0,t;H^2(\Omega)).
 \label{eq:gradHT-H2eps}
\end{equation}
Moreover, there exists a continuous nondecreasing function $M_8$, finite on compact subintervals of $[0,\mathcal T^*)$, such that
\begin{equation}
 \int_0^t\|\nabla_HT(s)\|_{H^2(\Omega)}^2\,ds
 \leq M_8(t).
 \label{eq:gradHT-H2eps-est}
\end{equation}
\end{proposition}

\begin{proof}
Fix $0<t<\mathcal T^*$. Propositions~\ref{prop:horizontal-temperature-L2} and \ref{prop:Tzz-estimate} give $\Delta_HT,\partial_z^2T\in L^\infty(0,t;L^2)$. Horizontal Dirichlet elliptic regularity and Lemma~\ref{lem:anisotropic-isotropic} with $s=2$, applied to $T$, therefore yield
\[
 \|T\|_{H^2}
 \leq C\bigl(\|\Delta_HT\|_2+\|T\|_{H^1}
                            +\|\partial_z^2T\|_2\bigr).
\]
this proves $T\in L^\infty(0,t;H^2)$. Proposition~\ref{prop:velocity-higher-spatial}, with $\varepsilon=1/2$, also gives $v\in L^2(0,t;H^{5/2})$.

On the two-dimensional domain $M$, for $0<\varepsilon<1$ we have
\begin{equation}
 \begin{aligned}
 \|fg\|_{H^{1-\varepsilon}(M)}
 &\leq C_\varepsilon\|f\|_{H^1(M)}\|g\|_{H^1(M)},\\
 \|fg\|_{H^1(M)}
 &\leq C\|f\|_{H^{3/2}(M)}\|g\|_{H^1(M)}.
 \end{aligned}
 \label{eq:2d-product}
\end{equation}
The second inequality follows from the product rule and the embeddings $H^{3/2}(M)\hookrightarrow L^\infty(M)\cap W^{1,4}(M)$ and $H^1(M)\hookrightarrow L^4(M)$. The embedding $H^{5/2}(\Omega)\hookrightarrow H^1_zH^{3/2}_H$, the one-dimensional Sobolev embedding in $z$, 
\[
 \|v\|_{L^\infty_zH^{3/2}_H}
 +\|w\|_{L^\infty_zH^{3/2}_H}
 \leq C\|v\|_{H^{5/2}}.
\]
Using the second inequality in \eqref{eq:2d-product} on each
horizontal slice, we obtain
\[
 \|v\cdot\nabla_HT+w\partial_zT\|_{L^2_zH^1_H}
 \leq C\|v\|_{H^{5/2}}\|\nabla T\|_{L^2_zH^1_H}
 \leq C\|v\|_{H^{5/2}}\|T\|_{H^2}.
\]
The right-hand side belongs to $L^2(0,t)$.
Proposition~\ref{prop:time-derivative-L2} gives
$T_t\in L^2(0,t;L^2_zH^1_H)$. The temperature equation thus implies
\[
 \|\Delta_HT\|_{L^2_zH^1_H}
 \leq\|T_t\|_{L^2_zH^1_H}
        +C\|v\|_{H^{5/2}}\|T\|_{H^2}
 \in L^2(0,t).
\]
Applying horizontal Dirichlet elliptic regularity to $T$ yields
\[
 \|\nabla_HT\|_{L^2_zH^2_H}
 \leq C\bigl(\|\Delta_HT\|_{L^2_zH^1_H}+\|T\|_2\bigr)
 \in L^2(0,t).
\]
Only the original lateral condition on $T$ enters this elliptic step.
Finally, Lemma~\ref{lem:anisotropic-isotropic} with $s=2$, applied
componentwise to $\nabla_HT$, gives
\[
 \|\nabla_HT\|_{H^2}
 \leq C\bigl(\|\nabla_HT\|_{L^2_zH^2_H}
             +\|T\|_{H^2}+\|\nabla_H\partial_z^2T\|_2\bigr).
\]
Consequently,
\begin{align*}
 \int_0^t\|\nabla_HT\|_{H^2}^2\,ds
 \leq C\Bigg[&\int_0^t\|T_t\|_{L^2_zH^1_H}^2\,ds\\
 &+\|T\|_{L^\infty(0,t;H^2)}^2
       \left(t+\int_0^t\|v\|_{H^{5/2}}^2\,ds\right)\\
 &+\int_0^t\|\nabla_H\partial_z^2T\|_2^2\,ds\Bigg].
\end{align*}
All terms on the right are controlled by the preceding estimates;
the last one is bounded by Proposition~\ref{prop:Tzz-estimate}.
Their continuous nondecreasing majorants supply $M_8$ and prove
\eqref{eq:gradHT-H2eps-est}.
\end{proof}

\begin{proof}[\textbf{Proof of Theorem~\ref{thm:higher-regularity}}]
Let $(v,T)$ be the unique global strong solution provided by
Theorem~\ref{thm:global-well-posedness}, and fix an arbitrary $t_0>0$.
Propositions~\ref{prop:velocity-H2} and~\ref{prop:velocity-higher-spatial}
give
\[
 v\in L^\infty(0,t_0;H^2(\Omega))
       \cap L^2(0,t_0;H^{3-\varepsilon}(\Omega))
 \qquad\text{for every }\varepsilon\in(0,1).
\]
Propositions~\ref{prop:Tzz-estimate} and
\ref{prop:horizontal-temperature-L2}, together with horizontal
Dirichlet elliptic regularity and Lemma~\ref{lem:anisotropic-isotropic},
yield $T\in L^\infty(0,t_0;H^2(\Omega))$, as shown in the proof of
Proposition~\ref{prop:horizontal-T-higher-reg}. That proposition also
gives $\nabla_HT\in L^2(0,t_0;H^2(\Omega))$.

Proposition~\ref{prop:time-derivative-L2} implies
\[
 v_t\in L^2(0,t_0;H^1(\Omega)),\qquad
 T_t,\ \nabla_HT_t\in L^2(0,t_0;L^2(\Omega)).
\]
Proposition~\ref{prop:vertical-temperature-time} supplies the remaining vertical derivative and hence $T_t\in L^2(0,t_0;H^1(\Omega))$.
Together with the spatial bounds, these estimates show that
\[
 v,T\in H^1(0,t_0;H^1(\Omega))
       \hookrightarrow C([0,t_0];H^1(\Omega)).
\]
Thus all the asserted regularities hold on every finite time interval.
Existence and uniqueness follow from
Theorem~\ref{thm:global-well-posedness}, completing the proof.
\end{proof}

\section{Extension to a varying bottom topography}
\label{sec:varying-bottom}
In this section we extend the preceding results to a fixed variable-depth domain with thermal diffusion adapted to the bottom topography. The velocity estimates follow the varying-bottom treatment in \cite{Kukavica_2007}; the thermal operator preserves positivity without any additional temperature condition.

For brevity, we present most of the energy calculations below as formal a priori estimates for sufficiently regular solutions and omit repeated approximation and limit passages. The required modifications of the flat-bottom arguments are specified in Subsection~\ref{subsec:topographic-notation}; a complete justification of the time-derivative estimate is given in Proposition~\ref{prop:top-time-derivative-L2}. The conclusions of Theorem~\ref{prop:varying-bottom-main} concern the constructed strong solutions: part~\textup{(i)} follows from the weighted local construction and continuation, and part~\textup{(ii)} improves the regularity of those same solutions. Smoothness used in the displayed calculations is not an additional hypothesis of that theorem.
Let $h\in C^3(\overline M)$ satisfy
\begin{equation}
  0<h_*\leq h(x,y)\leq h^*<\infty,
  \qquad (x,y)\in \overline M,
  \label{eq:h-assumptions}
\end{equation}
and set
\[
  \Omega_h:=\{(x,y,z):(x,y)\in M,\ -h(x,y)<z<0\}.
\]
The upper, bottom, and lateral boundaries are denoted respectively by
\[
  \Gamma_t=M\times\{0\},\qquad
  \Gamma_b=\{(x,y,-h(x,y)):(x,y)\in M\},\qquad
  \Gamma_s=\partial\Omega_h\setminus(\Gamma_t\cup\Gamma_b).
\]
We retain the no-slip condition for the horizontal velocity on $\Gamma_b\cup\Gamma_s$, the stress-free condition on $\Gamma_t$, and the impermeability condition $(v,w)\cdot n=0$ on $\Gamma_t\cup\Gamma_b$.  In particular, under the no-slip condition one has $v=w=0$ on $\Gamma_b$.  The temperature continues to satisfy only the lateral condition $T=0$ on $\Gamma_s$.
\subsection{Motivation for topography-adapted thermal diffusion}

The Cartesian horizontal Laplacian does not retain the flat-bottom
energy identity under the stated temperature boundary conditions.
Indeed, for a sufficiently regular temperature with $T=0$ on
$\Gamma_s$, integration by parts gives
\[
  \int_{\Omega_h}(-\Delta_HT)T
  =\int_{\Omega_h}|\nabla_HT|^2
   -\int_{\Gamma_b}T(\nabla_HT,0)\cdot n_b,
\]
where $n_b$ is the outward unit normal to the bottom.
The bottom term has no prescribed sign and does not vanish in general. Therefore, we introduce the diffusion directions tangent to the topographic layers $z/h(x,y)=\mathrm{const}$. Diffusion along such layers is related to the iso-sigma formulations used in ocean modelling; see \cite[Section~5.1]{BarnierEtAl1998Sigma}. This is a choice of thermal diffusion operator for the varying-bottom model. We introduce the notation first, and then establish its energy properties.

\subsection{Notation, Formulation, and Justification}
\label{subsec:topographic-notation}

\paragraph{Coordinates and weighted spaces.}
Set
\[
  Q=M\times(-1,0),\qquad
  \Phi(x,y,\sigma)=(x,y,h(x,y)\sigma).
\]
Write $v^{\mathrm{phys}}$ for the horizontal velocity on $\Omega_h$,
and use
\[
  v=v^{\mathrm{phys}}\circ\Phi,\qquad
  \theta=T\circ\Phi,\qquad
  \omega=\frac{w\circ\Phi-\sigma v\cdot\nabla_Hh}{h}
\]
for the transformed fields on $Q$. Horizontal derivatives of fields on $Q$ are taken at fixed $\sigma$. Since $\det D\Phi=h$, the physical $L^2$ inner product becomes the weighted inner product on $L_h^2(Q):=L^2(Q;h\,dx\,dy\,d\sigma)$:
\[
  (f,g)_h=\int_Qhfg,\qquad \|f\|_{2,h}^2=(f,f)_h.
\]
We also write $(\cdot,\cdot)_{L_h^2}$ and $\|\cdot\|_{L_h^2}$ for these same quantities. For vector fields, products in the inner product are Euclidean dot products.

\paragraph{Thermal operators and form.}
Recall from \eqref{a_i definition} that
\[
  a_i(x,y,z)=\frac{z}{h(x,y)}\partial_{x_i}h(x,y),
  \qquad i=1,2,
\]
where $x_1=x$ and $x_2=y$, and write $a=(a_1,a_2)$.
On the physical domain, define
\begin{equation}
  X_i^h:=\partial_{x_i}+a_i\partial_z,
  \qquad i=1,2,
  \label{eq:topographic-vector-fields}
\end{equation}
and denote the topographic horizontal gradient by
\[
  X^hf:=(X_1^hf,X_2^hf).
\]
With $\nabla_3=(\nabla_H,\partial_z)$, set
\begin{equation}
  B_h:=\begin{pmatrix}1&0\\0&1\\a_1&a_2\end{pmatrix},
  \qquad
  K_h:=B_hB_h^T
  =\begin{pmatrix}
    1&0&a_1\\
    0&1&a_2\\
    a_1&a_2&|a|^2
  \end{pmatrix}.
  \label{eq:topographic-diffusion-matrix}
\end{equation}
Thus $B_h^T\nabla_3T=(X_1^hT,X_2^hT)^T$. Define
\begin{equation}
  \Delta_{H,h}^{\mathrm{top}}T
  :=\nabla_3\cdot(K_h\nabla_3T),
  \qquad A_{T,h}:=-\Delta_{H,h}^{\mathrm{top}}.
  \label{eq:topographic-laplacian}
\end{equation}
The change of variables gives
\[
  (X_i^hT)\circ\Phi=\partial_{x_i}\theta,\qquad
  (\Delta_{H,h}^{\mathrm{top}}T)\circ\Phi
  =\frac1h\nabla_H\cdot(h\nabla_H\theta).
\]
Accordingly, on $Q$ we use
\begin{equation}
  A_h\theta:=-\frac1h\nabla_H\cdot(h\nabla_H\theta),
  \label{eq:fixed-positive-operator}
\end{equation}
so that $(A_{T,h}T)\circ\Phi=A_h\theta$, with associated form
\begin{align}
  \mathfrak a_h(\theta,\varphi)
  :=\int_Qh\nabla_H\theta\cdot\nabla_H\varphi.
\end{align}

Incompressibility and impermeability become
\begin{equation}
  \nabla_H\cdot(hv)+\partial_\sigma(h\omega)=0,
  \qquad \omega|_{\sigma=-1,0}=0.
  \label{eq:weighted-incompressibility}
\end{equation}
\vspace{-5mm}
\begin{equation}
  \omega(x,y,\sigma,t)
  =-\frac1{h(x,y)}\int_{-1}^{\sigma}
  \nabla_H\cdot(hv)(x,y,\eta,t)\,d\eta.
  \label{eq:omega-representation}
\end{equation}
The transport operator is
\[
  L_v:=v\cdot\nabla_H+\omega\partial_\sigma .
\]
Introduce
\[
  \begin{aligned}
  \mathcal D_{H,h}
  &=\nabla_H-\sigma\frac{\nabla_Hh}{h}\,\partial_\sigma ,\\
  \mathcal L_h f
  &=\bigl[\Delta(f\circ\Phi^{-1})\bigr]\circ\Phi.
  \end{aligned}
\]
These are the pullbacks of the physical horizontal gradient and
Laplacian, respectively. The velocity energy form is
\begin{equation}
  \mathfrak a_{v,h}(f,f)
  =\int_Qh\left(
       |\nabla_Hf-\sigma\frac{\nabla_Hh}{h}\,\partial_\sigma f|^2
       +h^{-2}|\partial_\sigma f|^2\right).
  \label{eq:top-time-velocity-form}
\end{equation}

\paragraph{Averaging, function spaces, and projections.}
Similar as the flat case, we keep the following definitions on $\Omega_h$,
\[
  M_hf(x,y):=\frac1{h(x,y)}\int_{-h(x,y)}^0f(x,y,z)\,dz.
\]
The physical velocity and temperature spaces are
\[
  V_h:=\left\{v\in H^1(\Omega_h)^2:
  v=0\text{ on }\Gamma_b\cup\Gamma_s,\quad
  \nabla_H\cdot\int_{-h(x,y)}^0v(x,y,z)\,dz=0\right\}
\]
and
\[
  V_{T,h}:=\{T\in H^1(\Omega_h):T=0\text{ on }\Gamma_s\}.
\]
Let $P_h$ be the orthogonal projection in $L_h^2(Q)^2$ onto
\[
  \overline{\{v\circ\Phi:v\in V_h\}}^{\,L_h^2(Q)^2},
\]
and let $\mathbb P_h^{\mathrm{phys}}$ be the orthogonal projection in $L^2(\Omega_h)^2$ onto $\overline{V_h}^{\,L^2(\Omega_h)^2}$.
\[
  P_hf=\bigl[\mathbb P_h^{\mathrm{phys}}(f\circ\Phi^{-1})\bigr]\circ\Phi.
\]
For $v=v^{\mathrm{phys}}\circ\Phi$ with $v^{\mathrm{phys}}\in V_h$,
the depth-integrated constraint is $\nabla_H\cdot(h\int_{-1}^0v\,d\sigma)=0$, and $v$ has zero lateral trace. Recall the decomposition for pressure,
\[
  p(x,y,z,t)=p_s(x,y,t)-\int_0^zT(x,y,\xi,t)\,d\xi
\]
and define
\begin{equation}
  \begin{split}
  \mathcal B_hf
  &=\mathcal D_{H,h}\left(h\int_0^\sigma f\,d\xi\right)\\
  &=h\int_0^\sigma\nabla_Hf\,d\xi
    +\nabla_Hh\left(\int_0^\sigma f\,d\xi-\sigma f\right).
  \end{split}
  \label{eq:top-time-buoyancy}
\end{equation}
The calligraphic operator $\mathcal B_h$ is distinct from the
$3\times2$ matrix $B_h$ in \eqref{eq:topographic-diffusion-matrix}.

\paragraph{Positivity and constant-depth consistency.}
Since $a=0$ on $\Gamma_t$ and $a=-\nabla_Hh$ on $\Gamma_b$,
the columns of $B_h$ are tangent to both boundaries. Hence
\[
  B_h^Tn=0,\qquad (K_h\nabla_3T)\cdot n=0
  \quad\text{on }\Gamma_t\cup\Gamma_b.
\]
Together with $T=0$ on $\Gamma_s$, this yields
\begin{equation}
  (A_{T,h}T,T)_{L^2(\Omega_h)}
  =\int_{\Omega_h}|B_h^T\nabla_3T|^2
  =\int_{\Omega_h}
    \left|\nabla_HT+\frac{z}{h}\nabla_Hh\,\partial_zT\right|^2
  \geq0.
  \label{eq:topographic-positivity}
\end{equation}
Changing variables and applying the horizontal Poincar\'e inequality
give
\[
  (A_{T,h}T,T)_{L^2(\Omega_h)}
  =\int_Qh|\nabla_H\theta|^2
  \geq\lambda_1(M)\frac{h_*}{h^*}\|T\|_{L^2(\Omega_h)}^2,
\]
where $\lambda_1(M)$ is the first Dirichlet eigenvalue on $M$.
\begin{remark}
For constant depth $h\equiv h_0$, we have $a=0$ and $K_h=\operatorname{diag}(1,1,0)$, so
\begin{equation}
  \Delta_{H,h}^{\mathrm{top}}=\Delta_H.
  \label{eq:constant-depth-consistency}
\end{equation}
\end{remark}

\paragraph{The coupled equations on the fixed cylinder.}
On $\Omega_h$, the temperature equation is taken to be
\begin{equation}
  \partial_tT+v^{\mathrm{phys}}\cdot\nabla_HT+w\partial_zT
  -\Delta_{H,h}^{\mathrm{top}}T=0.
  \label{eq:temperature-varying-bottom}
\end{equation}
With the notation, we get topographic-coupled system,
\begin{equation}
  \begin{cases}
  v_t-\mathcal L_hv+L_vv+\nabla_Hp_s+f_0k\times v
      =\mathcal B_h\theta,\\
  \theta_t+A_h\theta+L_v\theta=0,\\
  \nabla_H\cdot(hv)+\partial_\sigma(h\omega)=0.
  \end{cases}
  \label{eq:topographic-coupled-system}
\end{equation}
Here $p_s=p_s(x,y,t)$ is independent of $\sigma$, and $\omega$ is
given by \eqref{eq:omega-representation}. The boundary and initial conditions are those of the varying-bottom problem, expressed through $\Phi$. The temperature equation can be written explicitly as
\begin{equation}
  \partial_t\theta+v\cdot\nabla_H\theta
  +\omega\partial_\sigma\theta
  -\frac1h\nabla_H\cdot(h\nabla_H\theta)=0,
  \label{eq:temperature-flattened}
\end{equation}
or, using the positive operator $A_h$, as
\begin{equation}
  \partial_t\theta+v\cdot\nabla_H\theta
  +\omega\partial_\sigma\theta+A_h\theta=0.
  \label{eq:temperature-positive-form}
\end{equation}

\paragraph{Transport cancellation and temperature energy identities.}
Weighted incompressibility and the velocity boundary conditions give
\begin{equation}
  (L_vf,f)_h
  =\frac12\int_{\partial Q}hf^2(v,\omega)\cdot n=0.
  \label{eq:topographic-transport-skew}
\end{equation}
Testing \eqref{eq:temperature-flattened} with
$|\theta|^{q-2}\theta$ in the weighted measure therefore yields, for
$2\leq q<\infty$,
\begin{equation}
  \frac1q\frac{d}{dt}\int_Qh|\theta|^q
  +(q-1)\int_Qh|\theta|^{q-2}|\nabla_H\theta|^2=0.
  \label{eq:weighted-Lq-temperature}
\end{equation}
This gives the basic $L^2$ estimate and, by letting $q\to\infty$ for bounded initial temperature, the maximum principle. The velocity diffusion remains $\mathcal L_h$; its elliptic estimates are those in Lemmas~\ref{lem:top-stokes-H2} and~\ref{lem:top-stokes-higher}, with their respective geometric hypotheses.

\paragraph{Justification of the a priori estimates.}
The computations below are presented as formal a priori estimates for smooth solutions. Their application to strong solutions follows the approximation and limiting arguments of the constant-depth case, adapted to the weighted setting. We outline this adaptation here and give a representative time-difference quotient justification for the coupled estimate of $(v_t,\theta_t)$ in Proposition~\ref{prop:top-time-derivative-L2}. The corresponding approximation and limiting details are omitted for the remaining estimates.

The depth bounds \eqref{eq:h-assumptions} ensure equivalence of the weighted and unweighted Sobolev norms on $Q$. Since $h$ is independent of time, the flattening map commutes with time differentiation. Moreover,
\[
 A_hf=-\Delta_Hf-\nabla_H\log h\cdot\nabla_Hf,
\]
and the additional term in \eqref{eq:omega-representation} is $\nabla_H\log h\cdot v$. These lower-order terms are controlled by the weighted product estimates, while the velocity operator is handled by the weighted Stokes estimates.

Let $\{\varphi_j^h\}_{j\geq1}$ be an orthonormal Dirichlet
eigenbasis of $A_h$ in $L^2(M;h\,dx\,dy)$, and set
\[
 c_0=1,\qquad
 c_k(\sigma)=\sqrt{2}\cos(k\pi(\sigma+1)),\quad k\geq1.
\]
Denote by $P_N^{T,h}$ the $L_h^2(Q)$-orthogonal projection onto
\[
 \mathcal E_N^h
 =\operatorname{span}
 \{\varphi_j^h c_k:1\leq j\leq N,\ 0\leq k\leq N\}.
\]
For $f\in H^1(Q)$ with zero lateral boundary condition, define
\[
 \|f\|_{H_h^1}^2
 :=\int_Q h\bigl(|f|^2+|\nabla_Hf|^2
                         +|\partial_\sigma f|^2\bigr).
\]
with
\begin{equation}
 \|P_N^{T,h}f\|_{H_h^1}\leq\|f\|_{H_h^1},
 \qquad
 P_N^{T,h}f\longrightarrow f
 \quad\text{in }H_h^1.
 \label{eq:weighted-projection-stability}
\end{equation}
 Both $A_h\theta_N$ and $-\partial_{\sigma\sigma}\theta_N$ belong to $\mathcal E_N^h$ and are therefore admissible tests, as in Lemma~\ref{lem:linear-temperature}.

For velocity, we use the spectral projections of the positive self-adjoint operator $\mathcal A_{v,h}$ associated with $\mathfrak a_{v,h}$ on $\operatorname{ran}P_h$, where $\mathcal A_{v,h}=-P_h\mathcal L_h$ on compatible $H^2$ fields. These projections converge in the form norm and, on $D(\mathcal A_{v,h})$, in the graph norm.

Using \eqref{eq:topographic-anisotropic-product}, \eqref{eq:topographic-basic-buoyancy-bound}, and Lemma~\ref{lem:top-stokes-H2}, the local construction of Section~\ref{sec:construction-global} carries over, including the decomposition into a fixed linear velocity and a remainder with zero initial data, and the iteration-difference estimates. The uniform full $H^1(Q)$ bounds and $L^2$ time-derivative bounds give compactness; the estimates pass by weak lower semicontinuity, and uniqueness follows from the weighted difference estimate below. The bounds \eqref{eq:topographic-KZ-velocity-bound} and \eqref{eq:topographic-H1-bound}, together with the time-derivative bounds, give continuation as in Section~\ref{sec:construction-global}.

For higher regularity, the time-derivative estimate is obtained first. The higher velocity bounds then follow from Lemma~\ref{lem:top-stokes-higher}, under \eqref{eq:top-stokes-orthogonal}. Once \eqref{eq:topographic-higher-velocity-assumption} holds, Lemma~\ref{lem:weighted-vertical-temperature} applies with $\rho=h$. Its auxiliary equation uses the existing strong solution as its coefficients and source, and the limit is identified by the vertical-primitive argument of Proposition~\ref{prop:Tzz-estimate}. Endpoint cutoffs are removed using $\omega|_{\sigma=-1,0}=0$ and $\partial_\sigma\omega\in L_t^2L^\infty$.

After recovering $H^2$ regularity, the right-hand side of
\eqref{eq:topographic-q-equation} belongs to $L_t^2L_h^2$.
The weighted version of the Galerkin justification following
Proposition~\ref{prop:vertical-temperature-time} then
justifies testing by $\partial_\sigma\theta_t$.
The initial trace is recovered in the same way, since $\partial_\sigma\theta_0$ belongs to the horizontal form domain by tangential differentiation of the lateral trace. The remaining elliptic estimates apply almost everywhere in time, as in Proposition~\ref{prop:horizontal-T-higher-reg}.

\begin{proposition}[Weighted stability and uniqueness]
\label{prop:topographic-strong-stability}
Let $(v^{(i)},\theta^{(i)})$, $i=1,2$, be solutions of
\eqref{eq:topographic-coupled-system} on $[0,t_0]$, satisfying the
prescribed velocity and lateral temperature boundary conditions
and having the Sobolev regularity listed in
Theorem~\ref{prop:varying-bottom-main}\textup{(i)} on this interval.
For a sufficiently large geometric constant $C_h$, let
\begin{equation}
\begin{split}
 \Lambda_h[v^{(2)},\theta^{(2)}]=C_h\bigl[&
 1+\|v^{(2)}\|_{H^1}^4+\|\theta^{(2)}\|_{H^1}^4\\
 &+(1+\|v^{(2)}\|_{H^1}^2)\|v^{(2)}\|_{H^2}^2\\
 &+(1+\|\theta^{(2)}\|_{H^1}^2)
                         \|\nabla_H\theta^{(2)}\|_{H^1}^2\bigr].
\end{split}
\label{eq:topographic-stability-coefficient}
\end{equation}
Then $\Lambda_h[v^{(2)},\theta^{(2)}]\in L^1(0,t_0)$ and,
for every $0\leq t\leq t_0$,
\begin{equation}
\begin{split}
 &\|v^{(1)}(t)-v^{(2)}(t)\|_{2,h}^2
       +\|\theta^{(1)}(t)-\theta^{(2)}(t)\|_{2,h}^2\\
 &\quad+\int_0^t\mathfrak a_{v,h}
       \bigl(v^{(1)}(s)-v^{(2)}(s),
             v^{(1)}(s)-v^{(2)}(s)\bigr)\,ds\\
 &\quad+\int_0^t
       \|\nabla_H(\theta^{(1)}(s)-\theta^{(2)}(s))\|_{2,h}^2\,ds\\
 &\leq\Bigl(\|v^{(1)}(0)-v^{(2)}(0)\|_{2,h}^2
       +\|\theta^{(1)}(0)-\theta^{(2)}(0)\|_{2,h}^2\Bigr)\\
 &\qquad\times\exp\!\left(\int_0^t
                  \Lambda_h[v^{(2)},\theta^{(2)}](s)\,ds\right).
\end{split}
 \label{eq:topographic-strong-stability}
\end{equation}
In particular, the solution is unique in this class and depends
continuously on its initial data in the weighted $L^2$ energy norm.
\end{proposition}

\begin{proof}
We follow the proof of Proposition~\ref{prop:strong-solution-stability} and record the modifications caused by the fixed depth $h$. The stated time regularity justifies testing the difference equations
by $v^{(1)}-v^{(2)}$ and $\theta^{(1)}-\theta^{(2)}$.
With $(v^{(1)},\omega^{(1)})$ as the transporting field,
\eqref{eq:topographic-transport-skew} gives transport cancellation
in the weighted inner product. The hydrostatic constraint eliminates
pressure, and the Coriolis term vanishes as before.

By \eqref{eq:omega-representation}, the additional term is $\nabla_H\log h\cdot(v^{(1)}-v^{(2)})$, and
\[
 |\omega^{(1)}-\omega^{(2)}|
 \leq C_h\int_{-1}^0
       \bigl(|\nabla_H(v^{(1)}-v^{(2)})|
                         +|v^{(1)}-v^{(2)}|\bigr)\,d\xi.
\]
Consequently, Lemma~\ref{anisotropic inequality} gives the same bounds for the two vertical transport terms as in Proposition~\ref{prop:strong-solution-stability}, with constants depending on $h$. The horizontal transport estimates are unchanged by weighted norm equivalence.

For the temperature term, integration by parts in physical coordinates gives
\begin{align}
 &\bigl(\mathcal B_h(\theta^{(1)}-\theta^{(2)}),
                 v^{(1)}-v^{(2)}\bigr)_h\\
 &\quad=-\bigl(\theta^{(1)}-\theta^{(2)},
                         h(\omega^{(1)}-\omega^{(2)})\bigr)_h\\
 &\qquad\phantom{=}-\bigl(\theta^{(1)}-\theta^{(2)},
              \sigma(v^{(1)}-v^{(2)})\cdot\nabla_Hh\bigr)_h.
\end{align}
Here $h(\omega^{(1)}-\omega^{(2)}) +\sigma(v^{(1)}-v^{(2)})\cdot\nabla_Hh$ is the pullback of the difference of the physical vertical velocities. The preceding diagnostic bound therefore yields

\begin{align}
 &\bigl|\bigl(\mathcal B_h(\theta^{(1)}-\theta^{(2)}),
                 v^{(1)}-v^{(2)}\bigr)_h\bigr|\\
 &\qquad\leq C_h\|\theta^{(1)}-\theta^{(2)}\|_{2,h}
                        \|v^{(1)}-v^{(2)}\|_{H^1}.
\end{align}

The remaining Sobolev and Young estimates are those of Proposition~\ref{prop:strong-solution-stability}. Using coercivity of $\mathfrak a_{v,h}$ to absorb the dissipation terms, and bounding the resulting coefficients by \eqref{eq:topographic-stability-coefficient}, we obtain
\begin{equation}
\begin{split}
 &\frac{d}{dt}\Bigl(\|v^{(1)}-v^{(2)}\|_{2,h}^2
                  +\|\theta^{(1)}-\theta^{(2)}\|_{2,h}^2\Bigr)\\
 &\quad+\mathfrak a_{v,h}
                  (v^{(1)}-v^{(2)},v^{(1)}-v^{(2)})\\
 &\quad+\|\nabla_H(\theta^{(1)}-\theta^{(2)})\|_{2,h}^2\\
 &\leq\Lambda_h[v^{(2)},\theta^{(2)}]
       \Bigl(\|v^{(1)}-v^{(2)}\|_{2,h}^2
                  +\|\theta^{(1)}-\theta^{(2)}\|_{2,h}^2\Bigr).
\end{split}
\label{eq:topographic-stability-energy}
\end{equation}
The strong-solution bounds imply
$\Lambda_h[v^{(2)},\theta^{(2)}]\in L^1(0,t_0)$.
Multiplication by the integrating factor and integration give
\eqref{eq:topographic-strong-stability}.
Time continuity in $L_h^2$ yields the estimate for every
$t\in[0,t_0]$, and hence uniqueness and continuous dependence.
\end{proof}

\subsection{Basic temperature bounds and the velocity estimates}
\label{subsec:basic-velocity-varying-bottom}
We display the a priori calculations under the justification convention specified in Subsection~\ref{subsec:topographic-notation}. Write $v_0=v_0^{\mathrm{phys}}\circ\Phi$ and $\theta_0=T_0\circ\Phi$.
The depth bounds give
\begin{equation}
 h_*\|f\|_{L^2(Q)}^2\leq\|f\|_{L_h^2}^2
                  \leq h^*\|f\|_{L^2(Q)}^2.
 \label{eq:weighted-norm-equivalence}
\end{equation}
The case $q=2$ of \eqref{eq:weighted-Lq-temperature} yields
\begin{equation}
 \|\theta(t)\|_{L_h^2}^2
 +2\int_0^t\|\nabla_H\theta(s)\|_{L_h^2}^2\,ds
 =\|\theta_0\|_{L_h^2}^2.
 \label{eq:topographic-basic-temperature-energy}
\end{equation}
Then
\begin{equation}
 \theta\in L^\infty(0,t_0;L^2(Q)),\qquad
 \nabla_H\theta\in L^2(0,t_0;L^2(Q)).
 \label{eq:topographic-basic-temperature-regularity}
\end{equation}
Thus the explicit formula \eqref{eq:top-time-buoyancy} gives
\begin{align}
 \|\mathcal B_h\theta\|_{L_h^2}^2
 &\leq C_h\bigl(\|\nabla_H\theta\|_{L_h^2}^2
                         +\|\theta\|_{L_h^2}^2\bigr),
 \label{eq:topographic-basic-buoyancy-bound}\\
 \int_0^{t_0}\|\mathcal B_h\theta\|_{L_h^2}^2\,ds
 &\leq C_h(1+t_0)\|\theta_0\|_{L_h^2}^2.
 \label{eq:topographic-buoyancy-time-bound}
\end{align}
For the velocity, the physical Dirichlet conditions and Poincar\'e
inequality give
$\mathfrak a_{v,h}(v,v)=\|\nabla v^{\mathrm{phys}}\|_2^2
\geq c_h\|v\|_{H^1(Q)}^2$.
The pressure cancels since
\[
 (\nabla_Hp_s,v)_h
 =-\int_Mp_s\nabla_H\cdot\left(h\int_{-1}^0v\,d\sigma\right)=0.
\]
Similarly, weighted incompressibility cancels transport terms, and the Coriolis term is orthogonal to $v$. Hence the energy argument of Proposition~\ref{prop:vn-L2}, with the weighted velocity form, holds,
\begin{align}
 \frac12\frac{d}{dt}\|v\|_{L_h^2}^2+\mathfrak a_{v,h}(v,v)
 =(\mathcal B_h\theta,v)_h\notag
 &\leq\frac12\mathfrak a_{v,h}(v,v)
                +C_h\|\mathcal B_h\theta\|_{L_h^2}^2,\\
 \sup_{0\leq s\leq t_0}\|v(s)\|_{L_h^2}^2
 +\int_0^{t_0}\|v(s)\|_{H^1(Q)}^2\,ds
 &\leq C_h\bigl(\|v_0\|_{L_h^2}^2
                  +(1+t_0)\|\theta_0\|_{L_h^2}^2\bigr).
 \label{eq:topographic-basic-velocity-bound}
\end{align}

In physical coordinates, we treat the momentum equation as a forced primitive equation with  force $F=(\mathcal B_h\theta-f_0k\times v)\circ\Phi^{-1}$,
\begin{equation}
 \int_0^{t_0}\|F(s)\|_{L^2(\Omega_h)}^2\,ds
 \leq C_{h,f_0}\bigl(\|v_0\|_{L_h^2}^2
                  +(1+t_0)\|\theta_0\|_{L_h^2}^2\bigr).
 \label{eq:topographic-KZ-forcing-bound}
\end{equation}
We apply the result of Kukavica and Ziane \cite{Kukavica_2007}, including its varying-bottom case; the
$L_t^2L_x^2$ forcing assumption is also explicit in
\cite[Theorem~2.1]{KukavicaZiane1}.
After composition with $\Phi$, the resulting bound is,
\begin{equation}
 \sup_{0\leq s\leq t_0}\|v(s)\|_{H^1(Q)}^2
 +\int_0^{t_0}\|v(s)\|_{H^2(Q)}^2\,ds
 \leq K_{v,h}(t_0)<\infty.
 \label{eq:topographic-KZ-velocity-bound}
\end{equation}
\subsection{The strong solution with $H^1$ initial temperature}
\label{subsec:H1-varying-bottom}

The diffeomorphism $\Phi$ gives
\begin{equation}
 C_h^{-1}\|T\|_{H^1(\Omega_h)}\leq\|\theta\|_{H^1(Q)}
                         \leq C_h\|T\|_{H^1(\Omega_h)}.
 \label{eq:H1-coordinate-equivalence}
\end{equation}
We adapt the $H^1$ calculation of Lemma~\ref{lem:linear-temperature}
using the positive self-adjoint Dirichlet realization of $A_h$ in
$L_h^2(Q)$. Its horizontal form satisfies
\begin{equation}
 (A_h\theta,\varphi)_h=\mathfrak a_h(\theta,\varphi),
 \label{eq:Ah-form-identity}
\end{equation}
and horizontal Dirichlet elliptic regularity gives
\begin{equation}
 \|\nabla_H^2\theta\|_{2,h}
 \leq C_h\bigl(\|A_h\theta\|_{2,h}+\|\nabla_H\theta\|_{2,h}\bigr).
 \label{eq:weighted-horizontal-elliptic-estimate}
\end{equation}
Multiply and integrate by $A_h\theta$ and differentiating in $\sigma$ therefore gives,
\begin{align}
 \frac12\frac{d}{dt}\|\nabla_H\theta\|_{2,h}^2
       +\|A_h\theta\|_{2,h}^2
 & =-(L_v\theta,A_h\theta)_h,
 \label{eq:horizontal-H1-energy}\\
 \frac12\frac{d}{dt}\|\partial_\sigma\theta\|_{2,h}^2
       +\|\nabla_H\partial_\sigma\theta\|_{2,h}^2
 & =-(\partial_\sigma v\cdot\nabla_H\theta,
                         \partial_\sigma\theta)_h-(\partial_\sigma\omega\,\partial_\sigma\theta,
                         \partial_\sigma\theta)_h.
 \label{eq:vertical-H1-energy}
\end{align}
Here $\partial_\sigma\theta$ has zero lateral trace because $\sigma$ is tangential to the side. 

The anisotropic estimate remains similar to that, which is used in Lemma~\ref{lem:linear-temperature}
and \eqref{eq:omega-representation},
\begin{equation}
 \|\omega f\|_2^2
 \leq C_h\|v\|_{H^1}(\|v\|_{H^1}+\|v\|_{H^2})
            \|f\|_2(\|f\|_2+\|\nabla_Hf\|_2).
 \label{eq:topographic-anisotropic-product}
\end{equation}
The coefficients depends on $h$. Applying the Sobolev and Young inequalities to \eqref{eq:horizontal-H1-energy} and \eqref{eq:vertical-H1-energy}, with the $L^2$ estimate before,
\begin{equation}
 \frac{d}{dt}E_{T,h} +\|A_h\theta\|_{2,h}^2+\|\nabla_H\partial_\sigma\theta\|_{2,h}^2
 \leq\Psi_h E_{T,h},
 \label{eq:topographic-H1-differential-inequality}
\end{equation}
where
\[
 E_{T,h}=\|\theta\|_{2,h}^2+\|\nabla_H\theta\|_{2,h}^2
                       +\|\partial_\sigma\theta\|_{2,h}^2
\]
and
\begin{equation}
 \Psi_h=C_h\bigl[1+\|v\|_{H^1}^4
                       +(1+\|v\|_{H^1}^2)\|v\|_{H^2}^2\bigr].
 \label{eq:Psi-h}
\end{equation}
By \eqref{eq:topographic-KZ-velocity-bound}, $\Psi_h\in L^1(0,t_0)$.
Gronwall's inequality gives
\begin{align}
 \sup_{0\leq s\leq t_0}E_{T,h}(s)
 &+\int_0^{t_0}\bigl(\|A_h\theta\|_{2,h}^2
                  +\|\nabla_H\partial_\sigma\theta\|_{2,h}^2\bigr)\,ds
 \notag\\
 &\leq2E_{T,h}(0)\exp\left(\int_0^{t_0}\Psi_h(s)\,ds\right).
 \label{eq:topographic-H1-bound}
\end{align}
Together with \eqref{eq:weighted-horizontal-elliptic-estimate}, this proves
\begin{equation}
 \theta\in L^\infty(0,t_0;H^1(Q)),\qquad
 \nabla_H\theta\in L^2(0,t_0;H^1(Q)).
 \label{eq:topographic-temperature-spatial-regularity}
\end{equation}

\paragraph{Time regularity}
By \eqref{eq:topographic-H1-bound},
\[
 \frac1h\nabla_H\cdot(h\nabla_H\theta)
 \in L^2(0,t_0;L_h^2(Q)).
\]
The velocity bound \eqref{eq:topographic-KZ-velocity-bound}, and anisotropic estimate
\eqref{eq:topographic-anisotropic-product} give
\[
 v\cdot\nabla_H\theta+\omega\partial_\sigma\theta
 \in L^2(0,t_0;L_h^2(Q)).
\]
Thus the temperature equation implies
$\partial_t\theta\in L^2(0,t_0;L_h^2(Q))$.

For the velocity, the same product estimates yield
\[
 \|v\cdot\nabla_Hv+\omega\partial_\sigma v\|_{2,h}^2
 \leq C_h\|v\|_{H^1}^2
              \bigl(\|v\|_{H^1}+\|v\|_{H^2}\bigr)^2.
\]
Hence the velocity transport term belongs to
$L^2(0,t_0;L_h^2(Q)^2)$.

Since $P_hv=v$ and $P_h$ is independent of time, the projected
momentum equation gives, in time distributions,
\[
 \partial_t v
 =P_h\bigl[\mathcal L_hv-L_vv-f_0k\times v+\mathcal B_h\theta\bigr].
\]
Therefore gives $\partial_t v\in L^2(0,t_0;L_h^2(Q)^2)$.

\paragraph{Continuation.}
The local existence and uniqueness argument of
Section~\ref{sec:construction-global} applies with the weighted forms
and the bounds above. If its maximal time $\mathcal T^*$ were finite,
the basic energies would give
$F\in L^2(0,\mathcal T^*;L^2(\Omega_h)^2)$. Extend this fixed force by
zero beyond $\mathcal T^*$. The forced velocity bound and uniqueness
then give \eqref{eq:topographic-KZ-velocity-bound} up to $\mathcal T^*$.
Thus $\Psi_h\in L^1(0,\mathcal T^*)$, and
\eqref{eq:topographic-H1-bound} controls the temperature $H^1$ norm.
Together with the preceding time-derivative estimates, these are the
bounds needed to restart the local solution at $\mathcal T^*$, a
contradiction.

We collect the global $H^1$ result and its higher-regularity extension
in the following theorem.

\begin{theorem}[Varying bottom topography]
\label{prop:varying-bottom-main}
Assume \eqref{eq:h-assumptions} and $h\in C^3(\overline M)$.
Consider the varying-bottom problem with the thermal diffusion operator
$A_{T,h}$ defined in \eqref{eq:topographic-laplacian} and the original
velocity and lateral temperature boundary conditions.\\
\smallskip
\noindent\textup{(i) Global strong solutions.}
For every $(v_0^{\mathrm{phys}},T_0)\in V_h\times V_{T,h}$,
there exists a unique global strong solution. For every finite $t_0>0$,
its pullbacks $v=v^{\mathrm{phys}}\circ\Phi$ and $\theta=T\circ\Phi$
satisfy
\[
 \begin{gathered}
 v\in L^\infty(0,t_0;H^1(Q)^2)\cap L^2(0,t_0;H^2(Q)^2),
 \qquad v_t\in L^2(0,t_0;L^2(Q)^2),\\
 \theta\in L^\infty(0,t_0;H^1(Q)),\qquad
 \nabla_H\theta\in L^2(0,t_0;H^1(Q)),\\
 \theta_t\in L^2(0,t_0;L^2(Q)).
 \end{gathered}
\]

\smallskip
\noindent\textup{(ii) Persistence of higher regularity.}
Suppose, in addition, that $M$ has $C^3$ boundary and
$\partial_\nu h=0$ on $\partial M$, where $\nu$ is the horizontal
outward unit normal; this is condition \eqref{eq:top-stokes-orthogonal}.
Assume also that
\[
 \begin{gathered}
 v_0^{\mathrm{phys}}\in H^2(\Omega_h)^2\cap V_h,
 \qquad T_0\in H^2(\Omega_h)\cap V_{T,h},\\
 \partial_zv_0^{\mathrm{phys}}=0\quad\text{on }\Gamma_t.
 \end{gathered}
\]
Then the temperature satisfies
\eqref{eq:topographic-physical-H2-conclusion}. For every finite $t_0>0$
and every $0<\varepsilon<1$, the velocity satisfies
\[
 \begin{gathered}
 v^{\mathrm{phys}}\in L^\infty(0,t_0;H^2(\Omega_h)^2)
                 \cap C([0,t_0];H^1(\Omega_h)^2),\\
  v^{\mathrm{phys}}\in L^2(0,t_0;H^{3-\varepsilon}(\Omega_h)^2),
 \qquad v_t^{\mathrm{phys}}\in L^2(0,t_0;H^1(\Omega_h)^2).
 \end{gathered}
\]
No top or bottom temperature boundary condition is imposed.
The constants may depend on the fixed domain, $h_*$, $h^*$,
and $\|h\|_{C^3}$, in addition to the initial norms and the time interval.
\end{theorem}

Part~\textup{(i)} follows from the weighted local construction and the
preceding continuation argument. In
Subsection~\ref{subsec:H2-varying-bottom}, the time-difference-quotient
and auxiliary-equation arguments described above transfer the
persistence estimates to that same strong solution, proving
part~\textup{(ii)} under its additional hypotheses.

\subsection{Persistence of $H^2$ regularity}
\label{subsec:H2-varying-bottom}

\subsubsection{Estimates for the time derivatives}
\label{subsubsec:topographic-time-derivatives}

We first establish the varying-bottom counterpart of Proposition~\ref{prop:time-derivative-L2}. We use the pullback notation, operators, and velocity energy form fixed in Subsection~\ref{subsec:topographic-notation}. Unsubscripted Sobolev norms in this subsection are taken on $Q$. For vector fields satisfying the homogeneous velocity boundary conditions, integration by parts and Poincare's inequality give
\begin{equation}
 (-\mathcal L_hf,f)_h=\mathfrak a_{v,h}(f,f),\qquad
 c_h\|f\|_{H^1}^2\leq\mathfrak a_{v,h}(f,f)
 \leq C_h\|f\|_{H^1}^2.
 \label{eq:top-time-coercivity}
\end{equation}

\begin{proposition}
\label{prop:top-time-derivative-L2}
Assume the hypotheses of the $H^1$ theory in Subsection~\ref{subsec:H1-varying-bottom}, and suppose in addition that $v^{\mathrm{phys}}_0\in H^2(\Omega_h)^2\cap V_h$ satisfies the homogeneous top stress condition, and that $T_0\in H^2(\Omega_h)\cap V_{T,h}$. Let $v_0=v^{\mathrm{phys}}_0\circ\Phi$ and $\theta_0=T_0\circ\Phi$. Then there exists a continuous nondecreasing function $M_{4,h}$, finite on every finite time interval, such that
\begin{equation}
 \begin{split}
 &\mathop{\mathrm{ess\,sup}}_{0<s<t}
   \left(\|v_t(s)\|_{2,h}^2+\|\theta_t(s)\|_{2,h}^2\right)\\
 &\quad+\int_0^t
   \left(\mathfrak a_{v,h}(v_t,v_t)
              +\|\nabla_H\theta_t\|_{2,h}^2\right)\,ds
 \leq M_{4,h}(t).
 \end{split}
 \label{eq:top-time-main}
\end{equation}
Equivalently, in physical coordinates,
\begin{equation}
 \begin{split}
 &\mathop{\mathrm{ess\,sup}}_{0<s<t}
   \left(\|\partial_t v^{\mathrm{phys}}(s)\|_{L^2(\Omega_h)}^2
                 +\|T_t(s)\|_{L^2(\Omega_h)}^2\right)\\
 &\quad+\int_0^t
   \left(\|\nabla\partial_t v^{\mathrm{phys}}\|_{L^2(\Omega_h)}^2
                 +\|X^hT_t\|_{L^2(\Omega_h)}^2\right)\,ds
 \leq M_{4,h}(t),
 \end{split}
 \label{eq:top-time-physical}
\end{equation}
where $X^h$ is the topographic horizontal gradient defined in
Subsection~\ref{subsec:topographic-notation}.
\end{proposition}
\begin{proof}
We follow the difference-quotient argument of
Proposition~\ref{prop:time-derivative-L2}, using
Proposition~\ref{prop:topographic-strong-stability}
for the weighted stability estimate.

Fix $t_0,\eta_0>0$ and consider the existing strong solution
on $[0,t_0+\eta_0]$. For $0<\eta<\eta_0$, set
\[
 \delta_\eta v(t):=\frac{v(t+\eta)-v(t)}{\eta},
 \qquad
 \delta_\eta\theta(t):=
             \frac{\theta(t+\eta)-\theta(t)}{\eta}.
\]
The initial regularity and boundary conditions imply
$v_0\in D(\mathcal A_{v,h})$ and $\theta_0\in D(A_h)$.
Define
\[
\begin{aligned}
 R_v&:=P_h\bigl[
       \mathcal L_hv_0-L_{v_0}v_0-f_0k\times v_0
                         +\mathcal B_h\theta_0\bigr],\\
 R_\theta&:=-A_h\theta_0-L_{v_0}\theta_0,
 \qquad
 E_0:=\|R_v\|_{2,h}^2+\|R_\theta\|_{2,h}^2.
\end{aligned}
\]
By \eqref{eq:topographic-anisotropic-product},
\eqref{eq:topographic-basic-buoyancy-bound}, and
$H^2(Q)\hookrightarrow L^\infty(Q)$, both residuals
belong to $L_h^2$, with
\begin{equation}
 E_0\leq C_{h,f_0}
       \bigl(1+\|v_0\|_{H^2}^2\bigr)
       \bigl(\|v_0\|_{H^2}^2+\|\theta_0\|_{H^2}^2\bigr).
 \label{eq:top-time-initial}
\end{equation}

We first estimate the initial difference quotients.
The equations for $v-v_0$ and $\theta-\theta_0$
have the same structure as the difference equations in
Proposition~\ref{prop:topographic-strong-stability},
with reference pair $(v_0,\theta_0)$ and additional
forcing $(R_v,R_\theta)$.
Repeating the energy calculation in its proof and retaining
these forcing terms gives
\[
\begin{aligned}
 &\frac{d}{dt}
       \bigl(\|v-v_0\|_{2,h}^2
                         +\|\theta-\theta_0\|_{2,h}^2\bigr)\\
 &\quad\leq
 C_{0,h}\bigl(\|v-v_0\|_{2,h}^2
                         +\|\theta-\theta_0\|_{2,h}^2\bigr)\\
 &\qquad\quad
 +2\sqrt{E_0}\,
       \bigl(\|v-v_0\|_{2,h}^2
                         +\|\theta-\theta_0\|_{2,h}^2\bigr)^{1/2},
\end{aligned}
\]
where $C_{0,h}$ depends only on the initial $H^2$ norms
and the fixed coefficients.
The known $H^1_tL_h^2$ regularity makes the energy
absolutely continuous, and its initial value is zero.
Lemma~\ref{lem:sqrt-gronwall} therefore yields
\[
 \bigl(\|v(s)-v_0\|_{2,h}^2
                    +\|\theta(s)-\theta_0\|_{2,h}^2\bigr)^{1/2}
 \leq s\sqrt{E_0}\,e^{C_{0,h}s/2}.
\]
Consequently,
\begin{equation}
 \|\delta_\eta v(0)\|_{2,h}^2
       +\|\delta_\eta\theta(0)\|_{2,h}^2
 \leq E_0e^{C_{0,h}\eta}.
 \label{eq:top-time-initial-quotients}
\end{equation}
No initial trace of $(v_t,\theta_t)$ is used.

Since the system is autonomous, both
$(v(\,\cdot+\eta),\theta(\,\cdot+\eta))$ and
$(v,\theta)$ solve the same equations on $[0,t_0]$.
Apply Proposition~\ref{prop:topographic-strong-stability}
with these as the first and second solutions, respectively.
Set
\begin{equation}
 \Lambda_h(s):=\Lambda_h[v,\theta](s),
 \label{eq:top-time-coefficient}
\end{equation}
where the expression on the right is defined in
\eqref{eq:topographic-stability-coefficient}.
This coefficient is independent of $\eta$ and belongs to
$L^1(0,t_0)$ by the previously established $H^1$ bounds.
Dividing \eqref{eq:topographic-strong-stability} by $\eta^2$
and using \eqref{eq:top-time-initial-quotients}, we obtain
\[
\begin{aligned}
 &\sup_{0\leq s\leq t}
       \bigl(\|\delta_\eta v(s)\|_{2,h}^2
                   +\|\delta_\eta\theta(s)\|_{2,h}^2\bigr)\\
 &\quad+\int_0^t
       \Bigl(\mathfrak a_{v,h}(\delta_\eta v,\delta_\eta v)
                   +\|\nabla_H\delta_\eta\theta\|_{2,h}^2\Bigr)\,ds\\
 &\qquad\leq
 2E_0\exp\!\left(C_{0,h}\eta+\int_0^t\Lambda_h(s)\,ds\right),
 \qquad 0\leq t\leq t_0.
\end{aligned}
\]
Here the factor $2$ results from estimating the supremum
and the dissipation integral separately.

Finally, the known $H^1_tL_h^2$ regularity gives
\[
 \delta_\eta v(t)=\frac1\eta\int_t^{t+\eta}v_t(s)\,ds,
 \qquad
 \delta_\eta\theta(t)=
                 \frac1\eta\int_t^{t+\eta}\theta_t(s)\,ds.
\]
As in Proposition~\ref{prop:time-derivative-L2},
the strong convergence of these Steklov averages yields
\[
 (\delta_\eta v,\delta_\eta\theta)
 \longrightarrow(v_t,\theta_t)
 \quad\text{strongly in }L^2(0,t_0;L_h^2)
 \quad\text{as }\eta\downarrow0.
\]
The uniform estimate also gives, along a subsequence,
weak convergence in the velocity and horizontal temperature
form spaces and weak-star convergence in
$L^\infty(0,t_0;L_h^2)$.
The strong convergence identifies these limits with the
time derivatives, and the form-domain constraints pass
to the limit by weak closedness.
Lower semicontinuity therefore gives
\begin{equation}
\begin{split}
 &\mathop{\mathrm{ess\,sup}}_{0<s<t}
       \bigl(\|v_t(s)\|_{2,h}^2+\|\theta_t(s)\|_{2,h}^2\bigr)\\
 &\quad+\int_0^t
       \Bigl(\mathfrak a_{v,h}(v_t,v_t)
                         +\|\nabla_H\theta_t\|_{2,h}^2\Bigr)\,ds\\
 &\qquad\leq
 2E_0\exp\!\left(\int_0^t\Lambda_h(s)\,ds\right).
\end{split}
\label{eq:top-time-gronwall}
\end{equation}
Since $t_0$ is arbitrary, we may take
\[
 M_{4,h}(t):=
 2(1+E_0)\exp\!\left(\int_0^t\Lambda_h(s)\,ds\right),
\]
which is continuous, nondecreasing, and finite on every
finite time interval.
The physical-coordinate estimate follows from
\eqref{eq:top-time-velocity-form} and
$(X^hT_t)\circ\Phi=\nabla_H\theta_t$.
\end{proof}
\paragraph{Remark.}
The temperature dissipation in physical coordinates is $\|X^hT_t\|_2^2$.
Indeed,
\[
 (\nabla_H^{\mathrm{phys}}T_t)\circ\Phi
 =\nabla_H\theta_t-\sigma\frac{\nabla_H h}{h}\,\partial_\sigma\theta_t,
\]
and the present estimate does not control
$\partial_\sigma\theta_t$. When $h$ is constant, $X^h=\nabla_H$,
and the proposition reduces to Proposition~\ref{prop:time-derivative-L2}.
The $H^2$ assumption on the initial data is used to control the
initial residual $E_0$ in \eqref{eq:top-time-initial};
all time-dependent coefficients in the proof are controlled at the
previously established $H^1$ level.

\subsubsection{Persistence of spatial $H^2$ regularity}
\label{subsubsec:topographic-spatial-H2}
We now derive the spatial priori estimates for part~\textup{(ii)} of
Theorem~\ref{prop:varying-bottom-main}, using the time-derivative
bounds from Subsubsection~\ref{subsubsec:topographic-time-derivatives}.
Throughout this subsubsection, $M$ has $C^3$ boundary and
$h\in C^3(\overline M)$ satisfies \eqref{eq:h-assumptions}.
For the higher-order estimates we impose the geometric condition
\begin{equation}
 \partial_\nu h=0\quad\hbox{on }\partial M,
 \label{eq:top-stokes-orthogonal}
\end{equation}
where $\nu$ is the horizontal outward unit normal. Indeed, the outward
normals to the sidewall and bottom are $(\nu,0)$ and
$(-\nabla_Hh,-1)/(1+|\nabla_Hh|^2)^{1/2}$, respectively; thus
\eqref{eq:top-stokes-orthogonal} means that these faces meet orthogonally.

\paragraph{Persistence estimates.}
Fix a finite $t_0>0$, and assume the initial-data hypotheses of
Proposition~\ref{prop:top-time-derivative-L2} and the orthogonality condition
\eqref{eq:top-stokes-orthogonal}. Unless a domain is indicated,
Sobolev norms below are taken on $Q$. We start from
\begin{equation}
 \begin{gathered}
 v\in L^\infty(0,t_0;H^1)\cap L^2(0,t_0;H^2),
 \qquad v_t\in L^\infty(0,t_0;L^2)\cap L^2(0,t_0;H^1),\\
 \theta\in L^\infty(0,t_0;H^1),\qquad
 \nabla_H\theta\in L^2(0,t_0;H^1),\\
 \theta_t\in L^\infty(0,t_0;L^2),\qquad
 \nabla_H\theta_t\in L^2(0,t_0;L^2).
 \end{gathered}
 \label{eq:top-spatial-known-bounds}
\end{equation}
We first estimate the velocity, then the second vertical temperature
derivative. The temperature equation recovers the remaining $H^2$
bounds, followed by the mixed and higher horizontal estimates.

\emph{Step 1: velocity regularity.}
We adapt the proofs of Propositions~\ref{prop:velocity-H2}
and~\ref{prop:velocity-higher-spatial}. Product estimates are taken on $Q$, while the Stokes estimates are applied on $\Omega_h$ and transferred by $\Phi$.
\[
 \omega(\sigma)=-\int_{-1}^{\sigma}
   \bigl(\nabla_H\cdot v+\nabla_H\log h\cdot v\bigr)(s)\,ds.
\]
The additional term has lower order with coefficient $\nabla_H\log h\in C^2(\overline M)$. In particular, whenever the right-hand sides are finite,
\[
 \|\omega\|_{H^1}\leq C_h\|v\|_{H^2},\qquad
 \|\omega\|_\infty+\|\nabla\omega\|_3
       \leq C_h\|v\|_{H^{7/3}}.
\]
\eqref{eq:top-time-buoyancy} gives $\|\mathcal B_h\theta\|_2\leq C_h\|\theta\|_{H^1}$, together with
\[
 \partial_\sigma(\mathcal B_h\theta)
       =h\nabla_H\theta-\sigma\nabla_Hh\,\partial_\sigma\theta,
\]
yields
\begin{equation}
 \|\mathcal B_h\theta\|_{H^1}
 \leq C_h\bigl(\|\nabla_H\theta\|_{H^1}+\|\theta\|_{H^1}\bigr).
 \label{eq:top-spatial-buoyancy-H1}
\end{equation}
Hence $\mathcal B_h\theta\in L^2(0,t_0;H^1)$ by
\eqref{eq:top-spatial-known-bounds}.

The interpolation argument of Proposition~\ref{prop:velocity-H2}, applied to the positive self-adjoint hydrostatic Stokes operator associated with $\mathfrak a_{v,h}$ in the weighted space, gives $v\in L^\infty(0,t_0;H^{3/2})$ by \eqref{eq:top-spatial-known-bounds} and Lemma~\ref{lem:top-stokes-H2}. The corresponding slice product estimate yields
\begin{equation}
 \|L_vv\|_2\leq C_h\|v\|_{H^{3/2}}^2.
 \label{eq:top-spatial-nonlinearity-L2}
\end{equation}
The physical momentum equation is \eqref{eq:top-stokes-system} with
\begin{equation}
 f=\bigl[-v_t-L_vv-f_0k\times v+\mathcal B_h\theta\bigr]
                                      \circ\Phi^{-1}.
 \label{eq:top-spatial-stokes-forcing}
\end{equation}
Thus $f\in L^\infty(0,t_0;L^2(\Omega_h))$, and Lemma~\ref{lem:top-stokes-H2} gives
\begin{equation}
 v\in L^\infty(0,t_0;H^2(Q)).
 \label{eq:top-spatial-velocity-H2}
\end{equation}

Using the interpolations similar as Proposition~\ref{prop:velocity-higher-spatial},
\[
 \begin{aligned}
 f\in L^2(0,t_0;H^{2/5}(\Omega_h))
 &\ \Longrightarrow\ v\in L^2(0,t_0;H^{7/3}(Q))\\
 &\ \Longrightarrow\ f\in L^2(0,t_0;H^1(\Omega_h)).
 \end{aligned}
\]
Here the interpolated Stokes estimate follows from
Lemmas~\ref{lem:top-stokes-H2} and~\ref{lem:top-stokes-higher}
under \eqref{eq:top-stokes-orthogonal}. Applying the latter lemma
once more yields, for every $0<\varepsilon<1$,
\begin{equation}
 \begin{gathered}
 v\in L^\infty(0,t_0;H^2(Q))\cap L^2(0,t_0;H^{3-\varepsilon}(Q)),\\
 \nabla_Hp_s\in L^2(0,t_0;H^{1-\varepsilon}(M)).
 \end{gathered}
 \label{eq:topographic-higher-velocity-assumption}
\end{equation}
\emph{Step 2: the second vertical derivative and $H^2$ recovery.}
Taking $\varepsilon=1/4$ in
\eqref{eq:topographic-higher-velocity-assumption}, we obtain
\[
 v\in L^\infty(0,t_0;H^2(Q))
       \cap L^2(0,t_0;H^{11/4}(Q)).
\]
Together with \eqref{eq:top-spatial-known-bounds} and the initial
and boundary conditions, this verifies the hypotheses of
Lemma~\ref{lem:weighted-vertical-temperature} with $\ell=1$ and
$\rho=h$. Hence
\begin{equation}
 \mathop{\mathrm{ess\,sup}}_{0<t<t_0}
       \|\partial_{\sigma\sigma}\theta(t)\|_{2,h}^2
 +\int_0^{t_0}
       \|\nabla_H\partial_{\sigma\sigma}\theta\|_{2,h}^2\,dt
 \leq C_{h,t_0}.
 \label{eq:topographic-H2-bound}
\end{equation}
Combined with
$\nabla_H\theta\in L^2(0,t_0;H^1)$, this yields
$\theta\in L^2(0,t_0;H^2(Q))$.

Recall,
\begin{equation}
 \partial_\sigma\omega
 =-\nabla_H\cdot v-\nabla_H\log h\cdot v.
 \label{eq:topographic-omega-derivatives}
\end{equation}

Since $\theta$ has zero lateral trace, the $s=2$ case of
Lemma~\ref{lem:anisotropic-isotropic} and horizontal Dirichlet
elliptic regularity give
\begin{equation}
 \|\theta\|_{H^2(Q)}
 \leq C_h\bigl(\|A_h\theta\|_{2,h}
          +\|\partial_{\sigma\sigma}\theta\|_{2,h}+\|\theta\|_{H^1}\bigr).
 \label{eq:topographic-H2-reconstruction}
\end{equation}
Moreover,\eqref{eq:topographic-anisotropic-product}, applied to $\partial_\sigma\theta$, implies
\begin{equation}
 \|\omega\partial_\sigma\theta\|_2
 \leq C_h\|v\|_{H^2}\|\partial_\sigma\theta\|_2^{1/2}
       \bigl(\|\partial_\sigma\theta\|_2
                    +\|\nabla_H\partial_\sigma\theta\|_2\bigr)^{1/2}.
 \label{eq:top-spatial-vertical-product}
\end{equation}
At almost every time, the temperature equation and these estimates
yield
\[
 \begin{aligned}
 \|A_h\theta\|_{2,h}
 &\leq\|\theta_t\|_{2,h}
       +\|v\cdot\nabla_H\theta\|_{2,h}
       +\|\omega\partial_\sigma\theta\|_{2,h}\\
 &\leq C_{h,t_0}\bigl(1+\|\nabla_H\partial_\sigma\theta\|_2^{1/2}\bigr)\\
 &\leq C_{h,t_0}\bigl(1+\|A_h\theta\|_{2,h}^{1/2}
                 +\|\partial_{\sigma\sigma}\theta\|_{2,h}^{1/2}\bigr).
 \end{aligned}
\]
Absorbing the square root of $\|A_h\theta\|_{2,h}$ gives
\begin{equation}
 \begin{aligned}
 \|A_h\theta\|_{2,h}
    &\leq C_{h,t_0}\bigl(1+\|\partial_{\sigma\sigma}\theta\|_{2,h}^{1/2}\bigr),\\
 \|\theta\|_{H^2}^2
    &\leq C_{h,t_0}\bigl(1+\|\partial_{\sigma\sigma}\theta\|_{2,h}^2\bigr).
 \end{aligned}
 \label{eq:top-spatial-temperature-recovery}
\end{equation}
Hence
\begin{equation}
 \theta\in L^\infty(0,t_0;H^2(Q)).
 \label{eq:topographic-H2-persistence}
\end{equation}

\emph{Step 3: mixed derivatives and the temperature time derivative.}
Differentiation in $\sigma$ gives
\begin{equation}
 \begin{split}
 \partial_\sigma\theta_t+A_h\partial_\sigma\theta
 ={}&-v\cdot\nabla_H\partial_\sigma\theta
      -\omega\partial_{\sigma\sigma}\theta\\
    &-\partial_\sigma v\cdot\nabla_H\theta
      -\partial_\sigma\omega\,\partial_\sigma\theta.
 \end{split}
 \label{eq:topographic-q-equation}
\end{equation}
Sobolev embedding and \eqref{eq:topographic-omega-derivatives}
bound the $L_h^2$ norm of the right-hand side by
\[
 C_h\bigl(\|v\|_{H^2}+\|\omega\|_\infty\bigr)\|\theta\|_{H^2}.
\]
This belongs to $L^2(0,t_0)$: $v$ and $\theta$ are bounded in $H^2$, and the estimate from Step~1 gives $\omega\in L^2(0,t_0;L^\infty)$. Testing \eqref{eq:topographic-q-equation} with $\partial_\sigma\theta_t$ and using the preceding bound gives
\[
 \begin{aligned}
 &\frac12\frac{d}{dt}
       \|\nabla_H\partial_\sigma\theta\|_{2,h}^2
       +\frac12\|\partial_\sigma\theta_t\|_{2,h}^2\\
 &\qquad\leq C_h\bigl(\|v\|_{H^2}^2+\|\omega\|_\infty^2\bigr)
                        \|\theta\|_{H^2}^2.
 \end{aligned}
\]
Integrating in time, using
$\|\nabla_H\partial_\sigma\theta_0\|_{2,h}
 \leq C_h\|\theta_0\|_{H^2}$, and then estimating
$A_h\partial_\sigma\theta$ from
\eqref{eq:topographic-q-equation}, we obtain
\begin{equation}
 \begin{split}
 &\mathop{\mathrm{ess\,sup}}_{0<t<t_0}
       \|\nabla_H\partial_\sigma\theta(t)\|_{2,h}^2\\
 &\quad+\int_0^{t_0}\bigl(
       \|\partial_\sigma\theta_t\|_{2,h}^2
             +\|A_h\partial_\sigma\theta\|_{2,h}^2\bigr)\,dt
 \leq C_{h,t_0}.
 \end{split}
 \label{eq:topographic-mixed-H2-energy}
\end{equation}
Combining this with the horizontal time-derivative bound in
\eqref{eq:top-spatial-known-bounds}, we obtain
\begin{equation}
 \theta_t\in L^2(0,t_0;H^1(Q)).
 \label{eq:topographic-temperature-time-H1}
\end{equation}
\emph{Step 4: higher horizontal regularity and physical coordinates.}
We follow the proof of Proposition~\ref{prop:horizontal-T-higher-reg},
using the second inequality in \eqref{eq:2d-product}.
Taking $\varepsilon=1/2$ in
\eqref{eq:topographic-higher-velocity-assumption} gives $v\in L^2(0,t_0;H^{5/2}(Q))$.
\
\[
 \|v\|_{L^\infty_\sigma H^{3/2}_H}
 +\|\omega\|_{L^\infty_\sigma H^{3/2}_H}
 \leq C_h\|v\|_{H^{5/2}},
\]
since the additional term $h^{-1}\nabla_Hh\cdot v$ has lower differential order. Expanding the temperature equation gives
\[
 \Delta_H\theta
 =\partial_t\theta+v\cdot\nabla_H\theta
   +\omega\partial_\sigma\theta
   -\frac{\nabla_Hh}{h}\cdot\nabla_H\theta.
\]
The horizontal product estimate therefore yields
\[
 \|\Delta_H\theta\|_{L^2_\sigma H^1_H}
 \leq C_h\bigl(\|\partial_t\theta\|_{H^1}
                 +(1+\|v\|_{H^{5/2}})\|\theta\|_{H^2}\bigr).
\]
The right-hand side belongs to $L^2(0,t_0)$ by Steps~1--3.
Horizontal Dirichlet regularity gives
$\theta\in L^2(0,t_0;L^2_\sigma H^3_H)$.
Applying Lemma~\ref{lem:anisotropic-isotropic} with $s=2$ to
$\nabla_H\theta$, and using \eqref{eq:topographic-H2-bound} to control
$\nabla_H\partial_{\sigma\sigma}\theta$, we obtain
\begin{equation}
 \nabla_H\theta\in L^2(0,t_0;H^2(Q)).
 \label{eq:topographic-higher-horizontal-regularity}
\end{equation}
Only the lateral condition on $\theta$ is used in the elliptic step;
no boundary condition on $\Delta_H\theta$ is imposed.

The bounds for $\theta$ and $\partial_t\theta$ imply
$\theta\in H^1(0,t_0;H^1(Q))\subset C([0,t_0];H^1(Q))$.
The fixed map $\Phi$ and its inverse are $C^3$, with Jacobian bounded
above and away from zero. Their Sobolev norm equivalences give
\begin{equation}
 C_h^{-1}\|T\|_{H^2(\Omega_h)}
 \leq\|T\circ\Phi\|_{H^2(Q)}
 \leq C_h\|T\|_{H^2(\Omega_h)},
 \label{eq:H2-coordinate-equivalence}
\end{equation}
and the corresponding norm equivalence at order $1$. The chain rule gives
\[
 \left(\nabla_HT+\frac{z}{h}\nabla_Hh\,\partial_zT\right)
       \circ\Phi=\nabla_H\theta,
 \qquad
 (\partial_tT)\circ\Phi=\partial_t\theta.
\]
Consequently,
\begin{equation}
 \begin{gathered}
 T\in L^\infty(0,t_0;H^2(\Omega_h))
                  \cap C([0,t_0];H^1(\Omega_h)),\\
 \nabla_HT+\frac{z}{h}\nabla_Hh\,\partial_zT
   \in L^2(0,t_0;H^2(\Omega_h)),\\
 \partial_tT\in L^2(0,t_0;H^1(\Omega_h)).
 \end{gathered}
 \label{eq:topographic-physical-H2-conclusion}
\end{equation}
For every $0<\varepsilon<1$, the velocity estimates similarly give
\[
 \begin{gathered}
 v^{\mathrm{phys}}\in L^\infty(0,t_0;H^2(\Omega_h))
                         \cap C([0,t_0];H^1(\Omega_h)),\\
  v^{\mathrm{phys}}\in L^2(0,t_0;H^{3-\varepsilon}(\Omega_h)),
 \qquad \partial_t v^{\mathrm{phys}}\in L^2(0,t_0;H^1(\Omega_h)).
 \end{gathered}
\]
This proves part~\textup{(ii)} of Theorem~\ref{prop:varying-bottom-main}, under \eqref{eq:top-stokes-orthogonal}, with the original boundary conditions. For constant $h$, the term containing $\nabla_Hh$ vanishes, and these conclusions reduce to the flat-bottom estimates.
\appendix
\section{Analytic Tools}
\label{app:analytic-tools}
We first record several inequalities and compactness results used below.
\begin{lemma}\label{anisotropic inequality}
(See \cite{https://doi.org/10.1002/cpa.10056} or \cite{CAO20144108}.)
Let $h>0$ be the constant depth. For functions $f,g,\varphi$
on $\Omega=M\times(-h,0)$ such that the right-hand sides below
are finite, the following inequalities hold:
\begin{align*}
&\int_M
 \left(\int_{-h}^0 f(x,y,z)\,dz\right)
 \left(\int_{-h}^0 g(x,y,z)\varphi(x,y,z)\,dz\right)\,dx\,dy\\
&\quad\leq
 C\|f\|_2^{1/2}
 \bigl(\|f\|_2^{1/2}+\|\nabla_Hf\|_2^{1/2}\bigr)
 \|g\|_2\|\varphi\|_2^{1/2}
 \bigl(\|\varphi\|_2^{1/2}
       +\|\nabla_H\varphi\|_2^{1/2}\bigr),
\end{align*}
and
\begin{align*}
&\int_M
 \left(\int_{-h}^0 f(x,y,z)\,dz\right)
 \left(\int_{-h}^0 g(x,y,z)\varphi(x,y,z)\,dz\right)\,dx\,dy\\
&\quad\leq
 C\|f\|_2\|g\|_2^{1/2}
 \bigl(\|g\|_2^{1/2}+\|\nabla_Hg\|_2^{1/2}\bigr)
 \|\varphi\|_2^{1/2}
 \bigl(\|\varphi\|_2^{1/2}
       +\|\nabla_H\varphi\|_2^{1/2}\bigr).
\end{align*}
\end{lemma}

\begin{lemma}\label{lemma:Gagliardo}
\noindent\textup{(i) Gagliardo--Nirenberg interpolation inequality.}
(see \cite{BREZIS20181355})Let $\Omega \subset \mathbb{R}^n$ be either the whole space, a half-space or a bounded Lipschitz domain. Let $1 \leq p, p_1, p_2 \leq +\infty$ be three positive extended real quantities and let $s, s_1, s_2$ be non-negative real numbers. Furthermore, let $\theta \in (0, 1)$ and assume that
\[
s_1 \leq s_2, \qquad s = \theta s_1 + (1 - \theta)s_2, \qquad \frac{1}{p} = \frac{\theta}{p_1} + \frac{1 - \theta}{p_2}
\]
hold. Then,
\[
\|u\|_{W^{s,p}(\Omega)} \leq C \|u\|_{W^{s_1,p_1}(\Omega)}^{\theta} \|u\|_{W^{s_2,p_2}(\Omega)}^{1-\theta}
\]
for any $u \in W^{s_1,p_1}(\Omega) \cap W^{s_2,p_2}(\Omega)$ if and only if
\[
\text{at least one of }
\begin{cases}
s_2 \in \mathbb{N} \text{ and } s_2 \geq 1, \\
p_2 = 1, \\
0 < s_2 - s_1 \leq 1 - \frac{1}{p_1}
\end{cases}
\text{ is false.}
\]
The constant $C > 0$ depends on the parameters $p, p_1, p_2, s, s_1, s_2, \theta$, on the domain $\Omega$, but not on $u$.

\medskip
\noindent\textup{(ii) Complex interpolation on bounded domains.}
Let $D\subset\mathbb R^n$ be a bounded Lipschitz domain. Let
$s_0,s_1\geq0$ with $s_0\neq s_1$, $1<p_0,p_1<\infty$, and
$0<\theta<1$, and set
\[
 s=(1-\theta)s_0+\theta s_1,\qquad
 \frac1p=\frac{1-\theta}{p_0}+\frac{\theta}{p_1}.
\]
Then
\begin{equation}
 [H_{p_0}^{s_0}(D),H_{p_1}^{s_1}(D)]_\theta=H_p^s(D)
 \label{eq:bounded-domain-complex-interpolation}
\end{equation}
with equivalent norms. Here $[\cdot,\cdot]_\theta$ denotes complex
interpolation, and $H_p^s(D)$ is the Bessel potential space defined
by restriction from $H_p^s(\mathbb R^n)$. These are the full spaces,
with no boundary trace conditions imposed.

In particular, if $m_0,m_1\in\mathbb N_0$, $m_0<m_1$, and
$(1-\theta)/p_0+\theta/p_1=1/2$, then
\begin{equation}
 [W^{m_0,p_0}(D),W^{m_1,p_1}(D)]_\theta
 =H^{(1-\theta)m_0+\theta m_1}(D)
 \label{eq:bounded-domain-integer-interpolation}
\end{equation}
with equivalent norms. We use here the integer-order identification
$H_p^m(D)=W^{m,p}(D)$ for $m\in\mathbb N_0$ and $1<p<\infty$,
and the identity $H_2^s(D)=H^s(D)$.
\end{lemma}

\begin{proof}
For $s_1<s_2$, part~\textup{(i)} is
\cite[Theorem~1]{BREZIS20181355}; the case $s_1=s_2$ follows from
H\"older's inequality.
Formula \eqref{eq:bounded-domain-complex-interpolation} is the
Lipschitz-domain version of the Bessel potential interpolation formula
in \cite[Appendix~A, formula~(A11) and the paragraph following~(A12)]{GKMW17}.
The integer-order identifications stated above then give
\eqref{eq:bounded-domain-integer-interpolation}.
\end{proof}

\begin{lemma}[Aubin--Lions lemma]\label{lem:aubin-lions}\cite{SimonCompactness}.
Let $X$, $B$ and $Y$ be three Banach spaces such that
\[
X\hookrightarrow\hookrightarrow B\hookrightarrow Y .
\]
Then the following results hold:
\begin{itemize}
\item[(i)]
If $F$ is bounded in $L^p(0,T;X)$, $1\le p<\infty$, and $\left\{\partial_t f;\ f\in F\right\}$ is bounded in $L^1(0,T;Y)$, then $F$ is relatively compact in
$L^p(0,T;B)$
\item[(ii)]
If $F$ is bounded in $L^\infty(0,T;X)$ and
$\{\partial_t f;\ f\in F\}$ is bounded in $L^r(0,T;Y)$ for some
$r>1$, then $F$ is relatively compact in
$C([0,T];B)$
\end{itemize}
\end{lemma}
In the following two lemmas, let $M\subset\mathbb R^2$ be a bounded
connected $C^3$ domain and let $h\in C^3(\overline M)$ satisfy
$\min_{\overline M}h>0$. We set
\[
 \Omega_h=\{(x,y,z):(x,y)\in M,\ -h(x,y)<z<0\}.
\]
When $h\equiv h_0>0$, we write $\Omega=M\times(-h_0,0)$.
All derivatives in these lemmas are taken in physical coordinates,
and $\Delta$ denotes the full three-dimensional Laplacian.
\begin{equation}
 \begin{cases}
 -\Delta v+\nabla_Hp_s=f&\text{in }\Omega_h,\\
 \displaystyle\nabla_H\cdot\int_{-h(x,y)}^0
          v(x,y,z)\,dz=0&\text{in }M,\\
 v=0&\text{on }\Gamma_b\cup\Gamma_s,\\
 \partial_zv=0&\text{on }\Gamma_t,
 \end{cases}
 \qquad \int_Mp_s=0.
 \label{eq:top-stokes-system}
\end{equation}
\begin{lemma}[Variable-depth $H^2$ Stokes estimate]
 \label{lem:top-stokes-H2}
For $f\in L^2(\Omega_h)^2$,
\begin{equation}
 \|v\|_{H^2(\Omega_h)}
 +\|p_s\|_{H^1(M)}\leq C_h\|f\|_{L^2(\Omega_h)}.
 \label{eq:top-stokes-H2}
\end{equation}
Here $C_h$ depends only on $M$ and $h$.
This conclusion does not require \eqref{eq:top-stokes-orthogonal}.
In particular, it also applies when $h(x,y)\equiv h_0>0$ is constant,
in which case $\Omega_h=M\times(-h_0,0)$.
\end{lemma}

\begin{proof}
Apply \cite{Ziane01111995} or \cite[Theorem~2.1]{GGRBRM04} with zero prescribed
depth-integrated divergence and both viscosity coefficients equal to
one.
\end{proof}

\begin{lemma}[Higher hydrostatic Stokes estimate at orthogonal junctions]
\label{lem:regularity}
\label{lem:top-stokes-higher}
Under the geometric assumptions of Lemma~\ref{lem:top-stokes-H2}, assume
in addition that
\[
 \partial_\nu h=0\qquad\text{on }\partial M,
\]
where $\nu$ is the horizontal outward unit normal. This is condition
\eqref{eq:top-stokes-orthogonal}: the bottom and sidewall meet at right
angles along their common edge. The top--sidewall junction is
automatically orthogonal.
For $f\in H^1(\Omega_h)^2$ and every $0<\varepsilon<1$, 
\begin{equation}
 \begin{split}
 &\|v\|_{H^{3-\varepsilon}(\Omega_h)}
 +\|\partial_zv\|_{H^{2-\varepsilon}(\Omega_h)}\\
 &\hspace{25mm}
 +\|\nabla_Hp_s\|_{H^{1-\varepsilon}(M)}
 \leq C_{h,\varepsilon}\|f\|_{H^1(\Omega_h)}.
 \end{split}
 \label{eq:top-stokes-higher}
\end{equation}
Only the boundary conditions in \eqref{eq:top-stokes-system} are imposed
on the solution. In particular, the result includes the constant-depth
case $h\equiv h_0>0$, for which the right-angle condition is automatic.
\end{lemma}

\begin{proof}
We first recall the scalar estimate
\begin{equation}
 \|g\|_{H^{2+a}(\Omega_h)}
 \leq C_{h,a}\bigl(\|-\Delta g\|_{H^a(\Omega_h)}
 +\|g\|_{H^1(\Omega_h)}\bigr),
 \qquad 0<a<1,
 \label{eq:top-stokes-scalar-shift}
\end{equation}
for the homogeneous scalar boundary conditions in
\eqref{eq:top-stokes-system}. Even reflection across $z=0$, using the
top Neumann condition, gives a homogeneous Dirichlet problem on
\[
 \widetilde\Omega_h
 =\{(x,y,z):(x,y)\in M,\ -h(x,y)<z<h(x,y)\}.
\]
By \eqref{eq:top-stokes-orthogonal}, its upper and lower faces meet the sidewall at right angles. The boundary has $C^3$ faces and edges, with no vertices, and the Dirichlet edge exponent is $2$.For background on the dependence of scalar elliptic corner regularity on the opening angles and boundary conditions, see \cite[Chapters~4--5]{Grisvard2011}.
Let $S$ be the bounded linear Dirichlet solution operator. 
\[
 \begin{cases}
 -\Delta(S\phi)=\phi&\text{in }\widetilde\Omega_h,\\
 S\phi=0&\text{on }\partial\widetilde\Omega_h.
 \end{cases}
\]
For $p_0=4/(2-a)$ and $p_1=4/(3-a)$,
\cite[Theorem~3.2 and Corollary~3.7]{Dauge92} therefore gives
\[
 S:L^{p_0}\longrightarrow W^{2,p_0},\qquad
 S:W^{1,p_1}\longrightarrow W^{3,p_1}.
\]
Indeed, $2-2/p_0<2$ and $3-2/p_1<2$, and $C^3$ faces suffice.
Since $(1-a)/p_0+a/p_1=1/2$, complex interpolation, as in
Lemma~\ref{lemma:Gagliardo}\textup{(ii)}, yields
$S:H^a\to H^{2+a}$. Even reflection of the forcing is bounded in
$H^a$ for $0<a<1$; restricting the solution to $\Omega_h$ proves
\eqref{eq:top-stokes-scalar-shift}. This argument also covers constant
depth, when $\widetilde\Omega_h=M\times(-h_0,h_0)$.

\emph{Constant depth.}
Write $\Omega=M\times(-h,0)$, where $h>0$ is constant.
Lemma~\ref{lem:top-stokes-H2} first gives
\[
 \|v\|_{H^2(\Omega)}+\|p_s\|_{H^1(M)}
 \leq C_h\|f\|_{L^2(\Omega)}.
\]
In particular, the bottom trace of $\partial_zv$ is well defined and
\[
 \|\partial_zv|_{z=-h}\|_{H^{1/2}(M)}
 \leq C_h\|f\|_{L^2(\Omega)}.
\]
Using the normalized averages
$\bar v=h^{-1}\int_{-h}^0v\,dz$ and
$\bar f=h^{-1}\int_{-h}^0f\,dz$, integration of
\eqref{eq:top-stokes-system} in $z$ yields
\begin{align}
 -\Delta_H\bar v+\nabla_Hp_s
 &=\bar f-h^{-1}\partial_zv|_{z=-h}
 &&\text{in }M,\label{eq:barv}\\
 \nabla_H\cdot\bar v&=0
 &&\text{in }M,\label{eq:barv-div}\\
 \bar v&=0
 &&\text{on }\partial M.\nonumber
\end{align}
Here the top derivative vanishes by the prescribed Neumann condition.
This averaging only uses the first-derivative traces already supplied
by the $H^2$ estimate.

The Dirichlet Stokes estimates on the smooth planar domain $M$,
together with fractional interpolation, imply
\[
 \begin{aligned}
 \|\nabla_Hp_s\|_{H^{1/2}(M)}
 &\leq C_h\left(
       \|\bar f\|_{H^{1/2}(M)}
       +\|\partial_zv|_{z=-h}\|_{H^{1/2}(M)}\right)\\
 &\leq C_h\|f\|_{H^1(\Omega)};
 \end{aligned}
\]
see \cite[Chapter~IV]{Galdi11} and \cite{BerghLofstrom76}.
Since $p_s$ is independent of $z$, the scalar equation
$-\Delta v=f-\nabla_Hp_s$ and
\eqref{eq:top-stokes-scalar-shift} with $a=1/2$ give
\[
 \|v\|_{H^{5/2}(\Omega)}
 \leq C_h\left(
       \|f\|_{H^{1/2}(\Omega)}
       +\|\nabla_Hp_s\|_{H^{1/2}(M)}
       +\|v\|_{H^1(\Omega)}\right)
 \leq C_h\|f\|_{H^1(\Omega)}.
\]
The trace theorem now yields
\[
 \|\partial_zv|_{z=-h}\|_{H^1(M)}
 \leq C_h\|f\|_{H^1(\Omega)}.
\]
Applying the planar Stokes estimate once more to
\eqref{eq:barv}--\eqref{eq:barv-div}, with the same homogeneous
Dirichlet condition, gives
\[
 \|\nabla_Hp_s\|_{H^1(M)}
 \leq C_h\left(
       \|\bar f\|_{H^1(M)}
       +\|\partial_zv|_{z=-h}\|_{H^1(M)}\right)
 \leq C_h\|f\|_{H^1(\Omega)}.
\]
Finally, for every $0<\varepsilon<1$, apply
\eqref{eq:top-stokes-scalar-shift} with $a=1-\varepsilon$ to
$-\Delta v=f-\nabla_Hp_s$. Together with
$\|\partial_zv\|_{H^{2-\varepsilon}(\Omega)}
\leq C\|v\|_{H^{3-\varepsilon}(\Omega)}$, this proves
\[
 \|v\|_{H^{3-\varepsilon}(\Omega)}
 +\|\partial_zv\|_{H^{2-\varepsilon}(\Omega)}
 +\|\nabla_Hp_s\|_{H^{1-\varepsilon}(M)}
 \leq C_{h,\varepsilon}\|f\|_{H^1(\Omega)}.
\]

\emph{Variable depth.}
Fix $0<\varepsilon<1/2$ and use the normalized average $M_h$
defined in Subsection~\ref{subsec:topographic-notation}.
Tangential differentiation of the bottom condition gives
\[
 (\partial_i v)_b=(\partial_i h)(\partial_zv)_b,
 \qquad i=1,2,
\]
where the subscript $b$ denotes restriction to $z=-h(x,y)$.
Consequently, Leibniz' rule and the top condition $\partial_zv=0$
yield
\[
 \int_{-h(x,y)}^0\Delta v\,dz
 =\Delta_H(hM_hv)-(1+|\nabla_Hh|^2)(\partial_zv)_b.
\]
Integration of \eqref{eq:top-stokes-system} therefore gives
\begin{equation}
 \begin{split}
 -\Delta_H(hM_hv)+\nabla_H(hp_s)
 ={}&hM_hf-(1+|\nabla_Hh|^2)(\partial_zv)_b
       +p_s\nabla_Hh,\\
 \nabla_H\cdot(hM_hv)={}&0,
 \qquad (hM_hv)|_{\partial M}=0.
 \end{split}
 \label{eq:top-stokes-higher-average}
\end{equation}
The bottom derivative is retained; no homogeneous condition on
$(\partial_zv)_b$ is used.
Lemma~\ref{lem:top-stokes-H2}, the trace theorem, and the boundedness
of vertical averaging give
\[
 \|(\partial_zv)_b\|_{H^{1/2}(M)}
 +\|hM_hf\|_{H^1(M)}+\|p_s\nabla_Hh\|_{H^1(M)}
 \leq C_h\|f\|_{H^1(\Omega_h)}.
\]
Apply the planar Dirichlet Stokes estimate to
\eqref{eq:top-stokes-higher-average}, with velocity $hM_hv$ and
pressure $hp_s$. Using
$\nabla_Hp_s=h^{-1}[\nabla_H(hp_s)-p_s\nabla_Hh]$, we obtain
\[
 \|\nabla_Hp_s\|_{H^{1/2-\varepsilon}(M)}
 \leq C_{h,\varepsilon}\|f\|_{H^1(\Omega_h)};
\]
see \cite[Chapter~IV]{Galdi11} and \cite{BerghLofstrom76}.
Apply \eqref{eq:top-stokes-scalar-shift} with
$a=1/2-\varepsilon$ to $-\Delta v=f-\nabla_Hp_s$ and then take
the bottom trace. This gives
\[
 \|v\|_{H^{5/2-\varepsilon}(\Omega_h)}
 +\|(\partial_zv)_b\|_{H^{1-\varepsilon}(M)}
 \leq C_{h,\varepsilon}\|f\|_{H^1(\Omega_h)}.
\]
A second application of the planar Stokes estimate to
\eqref{eq:top-stokes-higher-average} gives the same bound for
$\|\nabla_Hp_s\|_{H^{1-\varepsilon}(M)}$. Finally,
\eqref{eq:top-stokes-scalar-shift} with $a=1-\varepsilon$ yields
\eqref{eq:top-stokes-higher}. The cases $1/2\leq\varepsilon<1$
follow by Sobolev embedding from any smaller loss, for example
$\varepsilon=1/4$.
\end{proof}
\begin{lemma}[Square-root Gronwall inequality]
\label{lem:sqrt-gronwall}
Let $\tau>0$, let $Y:[0,\tau]\to[0,\infty)$ be absolutely continuous, and let $a,b\in L^1(0,\tau)$ be nonnegative. Suppose that
\[
Y'(t)\leq a(t)Y(t)+2b(t)\sqrt{Y(t)}
\qquad\text{for a.e. }t\in(0,\tau).
\]
Then, for every $t\in[0,\tau]$,
\[
\sqrt{Y(t)}
\leq
\sqrt{Y(0)}
\exp\left(\frac12\int_0^t a(r)\,dr\right)
+
\int_0^t b(s)
\exp\left(\frac12\int_s^t a(r)\,dr\right)\,ds.
\]
In particular, if $Y(0)=0$, $a(t)\equiv a_0\geq0$,
and $b(t)\equiv b_0\geq0$, then
\[
\sqrt{Y(t)}\leq b_0t\,e^{a_0t/2},
\qquad
Y(t)\leq b_0^2t^2e^{a_0t},
\qquad 0\leq t\leq\tau.
\]
\end{lemma}

\begin{proof}
For $\varepsilon>0$, set
\[
Z_\varepsilon(t):=\sqrt{Y(t)+\varepsilon}.
\]
Then $Z_\varepsilon$ is absolutely continuous and strictly
positive. The chain rule gives, for almost every $t$,
\[
\begin{aligned}
Z_\varepsilon'(t)
&=\frac{Y'(t)}{2\sqrt{Y(t)+\varepsilon}}\\
&\leq
\frac{a(t)}{2}
\frac{Y(t)}{\sqrt{Y(t)+\varepsilon}}
+
b(t)\frac{\sqrt{Y(t)}}{\sqrt{Y(t)+\varepsilon}}\\
&\leq \frac{a(t)}{2}Z_\varepsilon(t)+b(t).
\end{aligned}
\]
Multiplying by the integrating factor
$\exp(-\frac12\int_0^t a(r)\,dr)$ yields
\[
\frac{d}{dt}
\left[
\exp\left(-\frac12\int_0^t a(r)\,dr\right)
Z_\varepsilon(t)
\right]
\leq
b(t)\exp\left(-\frac12\int_0^t a(r)\,dr\right).
\]
Integrating from $0$ to $t$, and using
$Z_\varepsilon(0)=\sqrt{Y(0)+\varepsilon}$, we obtain
\[
Z_\varepsilon(t)
\leq
\sqrt{Y(0)+\varepsilon}
\exp\left(\frac12\int_0^t a(r)\,dr\right)
+
\int_0^t b(s)
\exp\left(\frac12\int_s^t a(r)\,dr\right)\,ds.
\]
Letting $\varepsilon\downarrow0$ proves the first assertion.
For constant coefficients and $Y(0)=0$, it follows that
\[
\sqrt{Y(t)}
\leq b_0\int_0^t e^{a_0(t-s)/2}\,ds
\leq b_0t\,e^{a_0t/2}.
\]
Squaring gives the remaining estimate.
\end{proof}
\begin{lemma}[Anisotropic-to-isotropic Sobolev estimate]
\label{lem:anisotropic-isotropic}
Let $\Omega=M\times(-h,0)$, where $M\subset\mathbb{R}^2$ is a bounded
$C^3$ domain. Let $0<s\leq2$. Suppose that
\[
    f\in L^2((-h,0);H^s(M))
\]
and
\[
    f,\ \partial_z f,\ \partial_z^2f\in L^2(\Omega).
\]
Then $f\in H^s(\Omega)$ and
\begin{equation}
\label{eq:anisotropic-isotropic}
    \|f\|_{H^s(\Omega)}
    \le C
    \left(
        \|f\|_{L^2_zH^s_H}
        +
        \|f\|_{H^2_zL^2_H}
    \right),
\end{equation}
where
\[
    \|f\|_{H^2_zL^2_H}^2
    :=
    \|f\|_{L^2(\Omega)}^2
    +
    \|\partial_zf\|_{L^2(\Omega)}^2
    +
    \|\partial_z^2f\|_{L^2(\Omega)}^2.
\]
\end{lemma}

\begin{proof}
Since $M$ is smooth and bounded and $(-h,0)$ is a bounded interval,
the standard Sobolev extension theorem, applied in the horizontal
and vertical variables, gives an extension $\widetilde f$ of $f$ to
$\mathbb{R}^3$ such that
\[
    \|\widetilde f\|_{L^2(\mathbb{R};H^s(\mathbb{R}^2))}
    \le
    C\|f\|_{L^2((-h,0);H^s(M))}
\]
and
\[
    \|\widetilde f\|_{H^2(\mathbb{R};L^2(\mathbb{R}^2))}
    \le
    C\|f\|_{H^2_zL^2_H}.
\]
We may therefore apply the Fourier transform to $\widetilde f$.
Writing $\xi=(\xi_H,\xi_3)$, where
$\xi_H=(\xi_1,\xi_2)$, and using $0<s\leq2$, we have
\[
    (1+|\xi_H|^2+|\xi_3|^2)^s
    \le
    C_s\left[
        (1+|\xi_H|^2)^s
        +(1+|\xi_3|^2)^2
    \right].
\]
Hence, by Plancherel's theorem,
\begin{align*}
    \|\widetilde f\|_{H^s(\mathbb{R}^3)}^2
    &\le
    C_s
    \int_{\mathbb{R}^3}
    (1+|\xi_H|^2)^s
    |\widehat{\widetilde f}(\xi)|^2\,d\xi
    \\
    &\quad
    +C_s
    \int_{\mathbb{R}^3}
    (1+|\xi_3|^2)^2
    |\widehat{\widetilde f}(\xi)|^2\,d\xi
    \\
    &\le
    C_s
    \left(
        \|\widetilde f\|_{L^2_zH^s_H}^2
        +
        \|\widetilde f\|_{H^2_zL^2_H}^2
    \right).
\end{align*}
Restricting $\widetilde f$ back to $\Omega$ gives
\eqref{eq:anisotropic-isotropic}.
\end{proof}

\begin{lemma}[Vertical temperature regularity with a fixed horizontal weight]
\label{lem:weighted-vertical-temperature}
Let $Q_\ell=M\times(-\ell,0)$, where $\ell>0$ and $M$ is a bounded
$C^3$ domain, and let $\rho\in C^3(\overline M)$ be positive. Write
\[
 \|g\|_{2,\rho}^2=\int_{Q_\ell}\rho|g|^2,
 \qquad A_\rho=-\rho^{-1}\nabla_H\cdot(\rho\nabla_H),
\]
with horizontal Dirichlet form domain
$L^2((-\ell,0);H_0^1(M))$. Suppose that
\[
 v\in L^\infty(0,t_0;H^2(Q_\ell)^2)
       \cap L^2(0,t_0;H^{11/4}(Q_\ell)^2),
 \qquad v|_{\partial M\times(-\ell,0)}=0,
\]
and that
\[
 \omega(\sigma)=-\rho^{-1}\int_{-\ell}^{\sigma}
                       \nabla_H\cdot(\rho v)(s)\,ds,
 \qquad \omega|_{\sigma=0}=0.
\]
Let $\theta$ solve
\[
 \theta_t+A_\rho\theta+v\cdot\nabla_H\theta
                   +\omega\partial_\sigma\theta=0
\]
in distributions, with zero lateral trace, initial value
$\theta_0\in H^2(Q_\ell)$ with zero lateral trace, and
\[
 \theta\in L^\infty(0,t_0;H^1),\qquad
 \nabla_H\theta\in L^2(0,t_0;H^1),\qquad
 \theta_t\in L^2(0,t_0;L^2).
\]
No temperature condition is imposed at $\sigma=-\ell,0$.
Then
\[
 \partial_{\sigma\sigma}\theta\in L^\infty(0,t_0;L^2),
 \qquad
 \nabla_H\partial_{\sigma\sigma}\theta\in L^2(0,t_0;L^2).
\]
More precisely, for $0\leq t\leq t_0$,
\begin{equation}
 \begin{split}
 &\sup_{0\leq s\leq t}
       \|\partial_{\sigma\sigma}\theta(s)\|_{2,\rho}^2
       +\int_0^t\|\nabla_H\partial_{\sigma\sigma}\theta\|_{2,\rho}^2\,ds\\
 &\qquad\leq C\bigl(1+\|\partial_{\sigma\sigma}\theta_0\|_{2,\rho}^2\bigr)
        \exp\!\Bigg(C\int_0^t\|v(s)\|_{H^{11/4}}\\
 &\hspace{16mm}\cdot\bigl(1+\|\theta\|_{L^\infty(0,t_0;H^1)}
             +\|\nabla_H\theta(s)\|_{H^1}\bigr)\,ds\Bigg),
 \end{split}
 \label{eq:weighted-vertical-temperature-bound}
\end{equation}
where $C$ depends only on $M,\ell,\rho$. The supremum uses the
weakly $L^2$-continuous representative of the second derivative.
\end{lemma}

\begin{proof}
This is the fixed-weight version of Proposition~\ref{prop:Tzz-estimate};
the varying-bottom case corresponds to $\ell=1$ and $\rho=h$.
We record only the changes in its proof. With
$(f,g)_\rho=\int_{Q_\ell}\rho fg$, positivity of $\rho$ gives equivalence
of weighted and unweighted norms, while
\[
 (A_\rho f,g)_\rho=(\nabla_Hf,\nabla_Hg)_\rho,
 \qquad [\partial_\sigma,A_\rho]=0.
\]
Moreover,
\[
 \nabla_H\cdot(\rho v)+\partial_\sigma(\rho\omega)=0,
 \qquad \omega|_{\sigma=-\ell,0}=0,
\]
so the principal transport cancels in the weighted energy identity.
\[
 \begin{aligned}
 \partial_\sigma\omega
   &=-\nabla_H\cdot v-\nabla_H\log\rho\cdot v,\\
 \partial_{\sigma\sigma}\omega
   &=-\nabla_H\cdot\partial_\sigma v
     -\nabla_H\log\rho\cdot\partial_\sigma v.
 \end{aligned}
\]
The additional terms contain the fixed coefficient
$\nabla_H\log\rho\in C^2(\overline M)$ and lower-order derivatives of $v$. Hence the coefficient bound \eqref{eq:flat-vertical-coefficients} and the product estimates in Proposition~\ref{prop:Tzz-estimate} remain valid with $(z,w,T)$ replaced by $(\sigma,\omega,\theta)$ and constants depending
on $M,\ell,\rho$.

Repeating that proof in $L^2_\rho$, with $r=\partial_{\sigma\sigma}\theta$,
gives the same energy inequality
\begin{equation}
 \frac{d}{dt}\bigl(1+\|r\|_{2,\rho}^2\bigr)
       +2\|\nabla_Hr\|_{2,\rho}^2
 \leq a_\rho(t)\bigl(1+\|r\|_{2,\rho}^2\bigr),
 \label{eq:weighted-vertical-energy}
\end{equation}
where
\[
 a_\rho(t)=C\|v(t)\|_{H^{11/4}}
 \bigl(1+\|\theta\|_{L^\infty(0,t_0;H^1)}
           +\|\nabla_H\theta(t)\|_{H^1}\bigr)
 \in L^1(0,t_0).
\]
Gronwall's inequality yields \eqref{eq:weighted-vertical-temperature-bound}.
As in the cited proposition, the supremum is taken over the weakly
$L^2_\rho$-continuous representative; no vertical endpoint condition
on $\theta$ is used.
\end{proof}

\section{Auxiliary A Priori Estimates}
\label{app:auxiliary-estimates}
In this section $\Omega=M\times(-h,0)$ has constant depth $h>0$, and $(v,T)$ is the flat-bottom strong solution on its maximal
interval $[0,\mathcal T^*)$, with the original boundary conditions.
We use the spaces $H,V,V_T$ and the hydrostatic Stokes operator $A$ and write $u=\partial_zv$.

\begin{proposition}
\label{prop:appendix-basic-energy}
For every $0<t<\mathcal T^*$,
\[
 \sup_{0\leq s\leq t}\|v(s)\|_2^2
 +\int_0^t\|\nabla v(s)\|_2^2\,ds
 \leq C\left(\|v_0\|_2^2+\int_0^t\|T(s)\|_2^2\,ds\right)
 \leq C\left(\|v_0\|_2^2+t\|T_0\|_2^2\right),
\]
where $C$ depends only on $M$ and $h$. Moreover, whenever
$T_0\in L^q(\Omega)$, $2\leq q\leq\infty$,
\[
 \sup_{0\leq s\leq t}\|T(s)\|_q\leq\|T_0\|_q.
\]
\end{proposition}

\begin{proof}
Set $Q_T(x,y,z)=\int_{-h}^zT(x,y,\xi)\,d\xi$.
Multiply the momentum equation with $v$, 
\[
 \int_\Omega\nabla_Hp_s\cdot v
 =-\int_Mp_s\nabla_H\cdot\left(\int_{-h}^0v\,dz\right)=0.
\]
Integration by parts, yields
\[
 \begin{aligned}
 \frac12\frac{d}{dt}\|v\|_2^2+\|\nabla v\|_2^2
 &=\int_\Omega\nabla_H Q_T\cdot v
 =-\int_\Omega Q_T\,\nabla_H\cdot v\\
 &\leq C_h\|T\|_2\|\nabla_Hv\|_2
 \leq\frac12\|\nabla_Hv\|_2^2+C_h\|T\|_2^2.
 \end{aligned}
\]
Consequently,
\begin{equation}
 \frac{d}{dt}\|v\|_2^2+\|\nabla_Hv\|_2^2
       +2\|\partial_zv\|_2^2\leq C_h\|T\|_2^2.
\end{equation}
For $2\leq q<\infty$, testing the temperature equation with
$|T|^{q-2}T$ gives
\begin{equation}
 \frac1q\frac{d}{dt}\|T\|_q^q
 +(q-1)\int_\Omega |T|^{q-2}|\nabla_HT|^2=0.
\end{equation}
Thus the $L^q$ norm is nonincreasing. For bounded $T_0$, letting
$q\to\infty$ on the bounded domain gives
\begin{equation}
 \sup_{0\leq s\leq t}\|T(s)\|_\infty\leq\|T_0\|_\infty.
\end{equation}
In particular,
\[
 \|T(t)\|_2^2+2\int_0^t\|\nabla_HT(s)\|_2^2\,ds
 \leq\|T_0\|_2^2.
\]
Integrating the velocity inequality and taking the time supremum
proves the first assertion. The identities are justified by the
energy approximations used in the construction.
\end{proof}

\begin{proposition}
\label{prop:appendix-velocity-L6-vertical}
Assume $v_0\in V$ and $T_0\in V_T$. There exists a continuous
nondecreasing function $M_1$ on $[0,\mathcal T^*)$ such that
\[
 \sup_{0\leq s\leq t}
       \bigl(\|v(s)\|_6^2+\|u(s)\|_2^2\bigr)
 +\int_0^t\|\nabla u(s)\|_2^2\,ds\leq M_1(t).
\]
If $\mathcal T^*<\infty$, this bound remains finite as
$t\uparrow\mathcal T^*$.
\end{proposition}

\begin{proof}
On the existing interval of the solution, write
\[
 v_t+(v\cdot\nabla_H)v+w\partial_zv-\Delta v+\nabla_Hp_s=f,
 \qquad
 f:=\nabla_H\int_{-h}^zT\,d\xi-f_0k\times v.
\]
Proposition~\ref{prop:appendix-basic-energy} and Poincar\'e's inequality imply
\[
 \begin{aligned}
 \int_0^t\|f(s)\|_2^2\,ds
 &\leq C\int_0^t
              \bigl(\|\nabla_HT(s)\|_2^2+\|v(s)\|_2^2\bigr)\,ds\\
 &\leq C\left(\|v_0\|_2^2+(1+t)\|T_0\|_2^2\right),
 \end{aligned}
\]
where $C$ also depends on the fixed Coriolis parameter.
The forced velocity estimates of \cite{Kukavica_2007}, with
$v_0\in V$ and $f\in L^2_tL^2_x$, therefore give the assertion;
see also \cite[Theorem~2.1]{KukavicaZiane1}.
If $\mathcal T^*<\infty$, the displayed estimate first gives
$f\in L^2(0,\mathcal T^*;L^2)$. Extend this actual force by zero to
$(0,\mathcal T^*+1)$ and apply the forced velocity theorem there.
Uniqueness identifies its solution with $v$ before $\mathcal T^*$,
so the resulting bounds are uniform up to that time.
\end{proof}

Recall the velocity form norm
\begin{equation}
 L(t)=\|A^{1/2}v(t)\|_H,
 \qquad L(t)^2=(Av(t),v(t))_H=\|\nabla v(t)\|_2^2
\end{equation}
for almost every time.

\begin{proposition}
\label{prop:appendix-velocity-H1}
Under the assumptions of Proposition~\ref{prop:appendix-velocity-L6-vertical}, there exists
a continuous nondecreasing function $M_2$ on $[0,\mathcal T^*)$
such that
\[
 \sup_{0\leq s\leq t}\|A^{1/2}v(s)\|_H^2
 +\int_0^t\|Av(s)\|_H^2\,ds\leq M_2(t).
\]
If $\mathcal T^*<\infty$, this bound remains finite as
$t\uparrow\mathcal T^*$.
\end{proposition}

\begin{proof}
Testing the projected momentum equation with $Av$ and applying
the Sobolev and anisotropic product estimates gives
\[
 \begin{aligned}
 \frac12\frac{d}{dt}L^2+\|Av\|_H^2
 &\leq\frac14\|Av\|_H^2
 +C\bigl(\|v\|_6^2L\|Av\|_H
          +L\|Av\|_H\|u\|_2\|\nabla u\|_2\\
 &\hspace{42mm}+\|v\|_2^2+\|\nabla_HT\|_2^2\bigr).
 \end{aligned}
\]
Here the $H^2$ Stokes estimate controls the second derivatives of
$v$, and the Poincar\'e estimate for $u$ uses its zero top trace.
Young's inequality therefore yields
\begin{equation}
 \begin{split}
 \frac{d}{dt}\|A^{1/2}v\|_H^2+\|Av\|_H^2
 &\leq C\bigl(\|v\|_6^4+\|u\|_2^2\|\nabla u\|_2^2\bigr)
                         \|A^{1/2}v\|_H^2\\
 &\quad+C\bigl(\|v\|_2^2+\|\nabla_HT\|_2^2\bigr).
 \end{split}
\end{equation}
The coefficient and the last two terms are integrable by
Propositions~\ref{prop:appendix-basic-energy}
and~\ref{prop:appendix-velocity-L6-vertical}. Gronwall's inequality proves the result.
\end{proof}

\begin{proposition}
\label{prop:appendix-temperature-H1}
Assume $v_0\in V$ and $T_0\in V_T$. There exists a continuous
nondecreasing function $M_3$ on $[0,\mathcal T^*)$ such that
\[
 \sup_{0\leq s\leq t}\|\nabla T(s)\|_2^2
 +\int_0^t\|\nabla_H\nabla T(s)\|_2^2\,ds\leq M_3(t).
\]
If $\mathcal T^*<\infty$, this bound remains finite as
$t\uparrow\mathcal T^*$.
\end{proposition}

\begin{proof}
Apply the temperature $H^1$ calculation in
Proposition~\ref{thm:H1 T} with the transporting velocity $v$.
Its coefficient is bounded by
\[
 C\bigl[1+\|v\|_{H^1}^4
        +(1+\|v\|_{H^1}^2)\|v\|_{H^2}^2\bigr],
\]
which belongs to $L^1(0,t)$ by Proposition~\ref{prop:appendix-velocity-H1} and the
$H^2$ Stokes estimate. Since $T_0\in V_T$, Gronwall's inequality
gives the assertion.
\end{proof}
\section*{Acknowledgments}

The authors acknowledge the use of Codex (OpenAI) for assistance with language polishing, LaTeX preparation, and mathematical proofreading. All AI-assisted material has been reviewed and edited by the authors. The authors have independently checked the mathematical results and proofs and assume full responsibility for the final content.

\bibliographystyle{alpha}
\bibliography{template}

\end{document}